\documentclass{article}
\usepackage[top=2cm,left=2.5cm,right=3cm,right=2cm]{geometry}
\usepackage[utf8]{inputenc}
\usepackage[english]{babel}
\usepackage{amsmath}
\usepackage{mathtools}
\usepackage{amsfonts}
\usepackage{amssymb}
\usepackage{amsthm}
\usepackage{bbold}
\usepackage{ifthen}
\usepackage{xcolor}
\usepackage[makeroom]{cancel}
\usepackage{graphicx}
\usepackage[square,numbers,sort&compress]{natbib}

\usepackage{hyperref}
\hypersetup{colorlinks=true,urlcolor=[rgb]{0,0.2,0.5},
           linkcolor=[rgb]{0.2,0.2,0.8},citecolor=[rgb]{0.8,0.1,0.1}}
\usepackage[vcentermath,enableskew]{youngtab}
\usepackage{ytableau}
\title{Traceless Projector on (Mixed) Tensor Products}
\date{}
\usepackage{float}
\usepackage{multirow}
\usepackage{accents}
\usepackage{tensor}

\DeclareFontFamily{U} {MnSymbolC}{}

\DeclareFontShape{U}{MnSymbolC}{m}{n}{
  <-6> MnSymbolC5
  <6-7> MnSymbolC6
  <7-8> MnSymbolC7
  <8-9> MnSymbolC8
  <9-10> MnSymbolC9
  <10-12> MnSymbolC10
  <12-> MnSymbolC12}{}
  
\DeclareSymbolFont{MnSyC} {U} {MnSymbolC}{m}{n}
\DeclareMathSymbol{\slashdiv}{\mathbin}{MnSyC}{29}

\newcommand*{\Scale}[2][4]{\scalebox{#1}{\ensuremath{#2}}}

\newcommand{\Sn}[1]{\mathfrak{S}_{#1}}

\newcommand{\bb}[1]{\mathbf{#1}}
\newcommand{\db}[1]{\dot{\mathbf{#1}}}
\newcommand{\dd}[1]{\ddot{\mathbf{#1}}}
\newcommand{\e}{\gamma}
\newcommand{\mione}{\scriptstyle{\scalebox{0.4}[1.0]{\( - \)} 1}}
\newcommand{\mitwo}{\scriptstyle{\scalebox{0.4}[1.0]{\( - \)} 2}}
\newcommand{\mithree}{\scriptstyle{\scalebox{0.4}[1.0]{\( - \)} 3}}
\newcommand{\one}{\scriptstyle{1}}
\newcommand{\two}{\scriptstyle{2}}
\newcommand{\three}{\scriptstyle{3}}
\newcommand{\zero}{\scriptstyle{0}}

\newcommand{\nmione}{\scriptstyle{n\scalebox{0.4}[1.0]{\( - \)} 1}}
\newcommand{\LRp}{\, {\scriptstyle \otimes}\,}
\newcommand{\svdots}{%
  \vbox{
    \scriptsize \baselineskip 2.5pt \lineskiplimit 0pt
    \hbox {.}\hbox {.}\hbox {.}\kern-0.75pt
  }%
}

\makeatletter
\MHInternalSyntaxOn
\def\MT_leftarrow_fill:{%
  \arrowfill@\leftarrow\relbar\relbar}
\def\MT_rightarrow_fill:{%
  \arrowfill@\relbar\relbar\rightarrow}
\newcommand{\xrightleftarrows}[2][]{\mathrel{%
  \raise.55ex\hbox{%
    $\ext@arrow 0359\MT_rightarrow_fill:{\phantom{#1}}{#2}$}%
  \setbox0=\hbox{%
    $\ext@arrow 3095\MT_leftarrow_fill:{#1}{\phantom{#2}}$}%
  \kern-\wd0 \lower.55ex\box0}}
\MHInternalSyntaxOff
\makeatother

\newtheorem{proposition}{Proposition}[section]
\newtheorem{lemma}[proposition]{Lemma}
\newtheorem{corollary}[proposition]{Corollary}
\newtheorem{theorem}[proposition]{Theorem}

\numberwithin{equation}{section}


\begin{document}
$\ $
\vspace{0.5 cm}
\begin{center}
{\Large \textbf{Construction of the traceless projection of tensors\\via the Brauer algebra
}}

{\small
\vspace{.45cm} 
{\large \textbf{D.~V. Bulgakova, Y.~O. Goncharov$^{1,2}$, T. Helpin$^{2}$ }}

\vskip .35cm \texttt{dvbulgakova@gmail.com $\qquad$ yegor.goncharov@gmail.com $\qquad$ thomas.helpin@gmail.com}

\vskip .35cm $^{1}$ Service de Physique de l’Univers, Champs et Gravitation,
Universit\'e de Mons -- UMONS,\\ 20 Place du Parc, B-7000 Mons, Belgique

\vskip .35cm $^{2}$ Institut Denis Poisson,
Universit\'e de Tours, Universit\'e d’Orl\'eans, CNRS,\\
Parc de Grandmont, 37200 Tours, France

}

\end{center}

\begin{abstract}\noindent
We describe how traceless projection of tensors of a given rank can be constructed in a closed form. On the way to this goal we invoke the representation theory of the Brauer algebra and the related Schur-Weyl dualities. The resulting traceless projector is constructed from purely combinatorial data involving Young diagrams. By construction, the projector manifestly commutes with the symmetric group and is well-adapted to restrictions to $GL$-irreducible tensor representations. We develop auxiliary computational techniques which serve to take advantage of the obtained results for applications. The proposed method of constructing traceless projectors leads to a particular central idempotent in the semisimple regime of the Brauer algebra.
\end{abstract}

\section{Introduction}

Numerous models in classical and quantum field theory assume a vector bundle ({\it i.e.} fields taking values in a vector space), with fibers endowed with a metric which is preserved by local (fiber-wise) symmetry transformations. From the algebraic point of view, the space of multiples of the metric forms a trivial representation of the transformations in question. In this respect, given a tensor power of the vector bundle, subspaces obtained by evaluation of scalar products between certain pairs of tensor components ({\it i.e.} traces of a tensor) are also preserved. According to Wigner, elementary physical constituents of the theory correspond to irreducible representations of the underlying symmetry group/algebra, {\it i.e.} those which contain no non-trivial proper subspaces preserved by symmetry transformations. Thus, in a theory with an invariant metric, tensor fields which represent the elementary physical degrees of freedom are necessarily traceless.
\vskip 4 pt

This work is motivated by the purely engineering question of constructing the traceless projection of a tensor in a closed form. More precisely, given a vector space with a non-degenerate metric (either symmetric or skew-symmetric), can one describe the traceless subspace of any tensor power of this space in a systematic manner? Despite the apparent simplicity of the question, computational complexity
of solving the traceless Ansatz starting from a general linear combination of a tensor and its single, double, triple, {\it etc.}, traces grows drastically with the rank of a tensor. In the current work we propose a shortcut based on the natural algebraic structures and the related representation theories {of the underlying symmetry algebras}.
\vskip 4 pt

With a non-degenerate metric on a vector space $V$ of dimension $N$, either symmetric or skew-symmetric, one canonically identifies a subgroup $G(N)\subset GL(N)$ of invertible linear transformations of the vector space which preserve the metric -- either the orthogonal group $O(N)$ or the symplectic group $Sp(N)$, respectively. Due to considerable similarity between the two cases, they are presented in parallel with the use of the same notations. 
\vskip 4 pt

Representation theory of the classical Lie groups $GL(N)$ and $G(N)$ acting on tensor products $V^{\otimes n}$ is well understood since the seminal works by H.~Weyl \cite{Weyl}. This classical knowledge is supplemented by a complementary piece -- the centralizer algebras $\mathfrak{B}_{n}(N)$ and  $\mathfrak{C}_n(N)\subset\mathfrak{B}_{n}(N)$ commuting with the action of $G(N)$ and $GL(N)$. With any above pair of the mutually centralising algebras at hand, representation theories of the two meet at the same space $V^{\otimes n}$, which gives rise to remarkable interplay between their irreducible representations. This constitutes the subject of the so-called Schur-Weyl duality reviewed in Section \ref{sec:basics}. In particular, there is a one-to-one correspondence between the irreducible representations of $GL(N)$ (respectively, $G(N)$) and $\mathfrak{C}_{n}(N)$ (respectively, $\mathfrak{B}_{n}(N)$) occurring in $V^{\otimes n}$, as well as between the branching rules for the embeddings $\mathfrak{C}_{n}(N)\subset \mathfrak{B}_{n}(N)$ and $G(N) \subset GL(N)$.
\vskip 4 pt

The algebra $\mathfrak{B}_{n}(N)$ is conveniently realised as a homomorphic image in $\mathrm{End}\big(V^{\otimes n}\big)$ of a unital associative algebra introduced by R. Brauer in \cite{Brauer}, referred to as the Brauer algebra $B_{n}(\varepsilon N)$ (with $\varepsilon = 1$ for $G(N) = O(N)$ and $\varepsilon = -1$ for $G(N) = Sp(N)$). The action of elements of $B_{n}(\varepsilon N)$ on $V^{\otimes n}$ is very intuitive and consists of permutations of factors in the tensor product, as well as evaluations of scalar products on certain pairs among them. Permutations form a basis in the symmetric group algebra $\mathbb{C}\Sn{n}\subset B_{n}(\varepsilon N)$, whose homomorphic image in $\mathrm{End}\big(V^{\otimes n}\big)$ gives $\mathfrak{C}_{n}(N)$. The action of $B_{n}(\varepsilon N)$ (and thus, of $\mathbb{C}\Sn{n}$) admits a convenient diagrammatic representation described in Section \ref{sec:Brauer_algebra}. The interplay between the classical Lie groups and the algebras $B_{n}(\varepsilon N)$ and $\mathbb{C}\Sn{n}$ can be summarised via the following diagram\footnote{Here we reproduce the diagram similar to the one presented in \cite{CDVM_Br_blocks}.}:
\begin{figure}[H]
    \centering
    \includegraphics[scale=1.0]{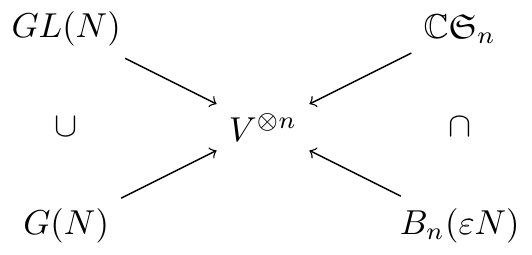}
    \caption{Schur-Weyl dualities for the classical Lie groups $GL(N)$ and $G(N)$.}
    \label{fig:SW_dualities}
\end{figure}
\vskip 4 pt

In our approach, we use the representation theory of $B_{n}(\varepsilon N)$ to write down the {\it universal traceless projector} in a closed form (see Theorem \ref{thm:projector}), totally avoiding the problem of solving any systems of linear equations. We use the term ``universal'' to emphasize the manifest commutation of the traceless projector in question with the action of $\mathbb{C}\Sn{n}$ (equivalently, with $\mathfrak{C}_n(N)$). 
\vskip 4 pt

Taking into account the interplay between representations of $GL(N)$ and $G(N)$, and thus of $\mathfrak{B}_{n}(N)$ and $\mathfrak{C}_{n}(N)$, we also construct the reduced form of the traceless projector via the restriction of the universal traceless projector to a specific representation of $GL(N)$ (see Theorem \ref{thm:reduced_projector}). Recall that irreducible representations of $G(N)$ can be realised as traceless projections of irreducible representations of $GL(N)$ (see, {\it e.g.}, \cite[Section 3.2]{Nazarov_tracelessness} or \cite[Chapter 10]{Hamermesh:100343}). Technically this implies commutativity of the two operations: {\it i)} subtracting traces, {\it ii)} projecting onto a $GL(N)$-irreducible subspace of $V^{\otimes n}$ by applying primitive idempotents in $\mathbb{C}\Sn{n}$ (for example, Young symmetrisers). This is what one commonly encounters in applications concerning irreducible $G(N)$-tensors: the major attention is paid to an appropriate symmetrisation of a tensor, while a compatible projection on the traceless subspace is postponed or even assumed implicitly. The universal traceless projector constructed in this work makes commutativity of the two operations manifest.
\vskip 4 pt

Note that the problem of projecting $V^{\otimes n}$ onto traceless irreducible representations of $G(N)$ in terms of the Brauer algebra was solved in \cite{Nazarov_tracelessness}. The form of projectors presented therein clearly distinguishes the subtraction of traces and the projection onto an irreducible representation of $GL(N)$. The part which is responsible for subtracting traces is specific for each particular primitive idempotent in $\mathbb{C}\Sn{n}$, and the two do not commute. While in our approach, the reduced traceless projector comes from the universal traceless projector, and thus commutes with the projector onto irreducible representations on $GL(N)$. This uniformises the projectors in \cite{Nazarov_tracelessness} and in addition allows one to utilise whatever primitive idempotents in $\mathbb{C}\Sn{n}$.
\vskip 4 pt

The traceless projector presented in Theorem \ref{thm:projector} has the form of a product of linear polynomials in a particular element $A_n\in B_{n}(\varepsilon N)$, which commutes with $\mathbb{C}\Sn{n}$. The element $A_n$ represents the sum of traces among all pairs of indices and is remarkable for its capability of distinguishing the traceless subspace of $V^{\otimes n}$ via its representation-theoretic properties combined with the Schur-Weyl-type duality for $G(N)$: {\it i)} $A_n$ is proportional to identity on irreducible representations of $\mathbb{C}\Sn{n}$ which are subspaces in irreducible representations of $B_{n}(\varepsilon N)$ in $V^{\otimes n}$ (see Lemma \ref{lem:D_block_diagonal}), {\it ii)} the kernel of $A_n$ is exactly the traceless subspace of $V^{\otimes n}$ (see Lemma \ref{lem:KerA}). Therefore, given a non-zero eigenvalue $\alpha$ of $A_n$, the element $\left(1 - \alpha^{-1}\,A_n\right)$ annihilates the corresponding traceful subspace in $V^{\otimes n}$. Taking the product over all non-zero eigenvalues, one arrives at an element in $B_{n}(\varepsilon N)$
\begin{equation}\label{eq:projector_introduction}
    P_n = \prod_{\def\arraystretch{0.8}\begin{array}{c}
        \scriptstyle \alpha\neq 0\,\\ 
        \scriptstyle (\text{eigenvalues of}\;\;A_n)
        \end{array}} \left(1-\frac{1}{\alpha}\,A_n\right)\,,
\end{equation}
which annihilates the subspaces with non-zero eigenvalues, and acts by identity on the ones with the zero eigenvalue, {\it i.e.} the traceless ones. As a result, application of $P_n$ projects $V^{\otimes n}$ onto the traceless subspace. Note that one does not need to diagonalize $A_n$, but only to construct the set of its eigenvalues. The latter are obtained explicitly from the set of skew-shape Young diagrams arising from the branching rules for irreducible representations of $B_{n}(\varepsilon N)$ in $V^{\otimes n}$ upon their restriction to $\mathbb{C}\Sn{n}$ (see Proposition \ref{prop:spectre_A}). A particular useful property about the factorised form \eqref{eq:projector_introduction} consists in a straightforward simplification of $P_n$ when restricted to an irreducible representations of $GL(N)$ in $V^{\otimes n}$, such that one needs to take into account only specific eigenvalues of $A_n$ (see Theorem \ref{thm:reduced_projector}).
\vskip 4 pt

We also aim at adapting our construction to numerical applications. While the combinatorial data (Young diagrams and the derived numeral ingredients) is constructed directly from the related definitions, expanding the factorised expression \eqref{eq:projector_introduction} requires optimisation. This problem is partially resolved due to the alternative \textit{quasi-additive} form of the universal traceless projector (see Corollary \ref{cor:projector_semi-sum}) given as a polynomial expression in $A_n$ whose degree is smaller than that in \eqref{eq:projector_introduction}. Additionally, in Section \ref{sec:bracelets} we formulate the known bijection between the  centraliser algebra of $\mathbb{C}\Sn{n}$, $C_{n}(\delta) \subset B_{n}(\delta)$ (for all $\delta\in\mathbb{C}$), in terms of ternary bracelets, and propose a technique which allows one to expand \eqref{eq:projector_introduction} successively by expressing the left multiplication by $A_n$ as a differential operator on the space of ternary bracelets. Besides, the relation to ternary bracelets allows us to derive the criterion for commutativily of $C_{n}(\delta)$ and of its particular sub-/quotient algebras. An extensive amount of examples demonstrating the computational efficiency of our construction are given in the companion Mathematica notebook which follows the narrative of this article. The notebook operates with a newly developed package \cite{BrauerAlgebraPackage} designed for studies and applications\footnote{For applications to tensor calculus the package is linked to the xAct bundle \cite{xActPackage} via the xBrauer extension.} of the Brauer algebra.
\vskip 4 pt

The choice of $B_{n}(\varepsilon N)$ and its representation theory is not at all unique for the purpose of constructing traceless projectors. While Schur-Weyl-type dualities relate the representation theories of $G(N)$ and $B_{n}(\varepsilon N)$ on the finite-dimensional space $V^{\otimes n}$, one has a duality of another type -- the Howe duality \cite{Howe} -- when one considers an infinite-dimensional subspace of particularly (anti)symmetrised tensors in the tensor algebra $U\subset T(V)$. In this case the subalgebra of transformations of $\mathrm{End}(T(V))$ which centralises the action of $G(N)$ is generated by a Lie algebra $\mathfrak{d}(r)$ (where the number $r\geqslant 2$ is specified by $U$): $\mathfrak{d}(r) = \mathfrak{sp}(2r)$ when $G(N) = O(N)$ and $\mathfrak{d}(r) = \mathfrak{o}(2r)$ when $G(N) = Sp(N)$. Simple $G(N)$-modules are singled out by highest-weight conditions for $\mathfrak{d}(r)$. The projection on the subspace of highest-weight vectors is performed by the so-called {\it extremal projectors} \cite{AST_71}, \cite{AST_73} which can be constructed for any reductive Lie algebra \cite{Zhelobenko_rep_reductive_Lie_alg} (see \cite{Tolstoy_2004} for review). In Section \ref{sec:Howe_duality} we analyse the structure of extremal projectors from the point of view of traceless projection of tensors. We show that the extremal projector for $\mathfrak{d}(r)$ is divisible by the extremal projector for $\mathfrak{sl}(r)$ (Lemma \ref{lem:projector_Howe_factorised}), which is reminiscent of the structure of projectors in \cite{Nazarov_tracelessness}. An interesting open question would be to find a universal traceless projector which projects the whole space $U$ onto its traceless counterpart. Let us mention that our approach, with the representation theory of $B_{n}(\varepsilon N)$ at hand, applies in the context of the fixed rank $n$, and is uniform for tensors of any symmetry.
\vskip 4 pt

As a by-product of the technique applied for constructing the traceless projection of $V^{\otimes n}$, the expression \eqref{eq:projector_introduction} for $P_n\in B_{n}(\delta)$ (with $\delta\in \mathbb{C}$) gives a particular central idempotent in the Brauer algebra when the latter is semisimple (the result presented in Theorem \ref{thm:splitting_idempotent}). Namely, the central idempotent in question is one of the splitting idempotents described in \cite{KMP_central_idempotents} (see also references therein for an overview of the problem of constructing central idempotents in $B_{n}(\delta)$).
\vskip 4 pt

The presentation is organised as follows. In Section \ref{sec:basics} we recall some basic facts about decomposition of tensor products into irreducible representations upon the action of the metric-preserving group and its centraliser algebra, and formulate the related Schur-Weyl dualities. This results in fixing the properties of the sought traceless projector. In Section \ref{sec:Brauer_algebra} we introduce the diagrammatic representation of the Brauer algebra and its action on tensors. Therein we mention some properties of the irreducible representations of $B_{n}(\varepsilon N)$ and consider their decomposition into irreducible $\mathbb{C}\Sn{n}$-summands upon restriction to $\mathbb{C}\Sn{n}$. In Section \ref{sec:traceless_projector} we give explicit expressions for the traceless projection and its restrictions to irreducible representations of $GL(N)$. In Section \ref{sec:bracelets} we introduce auxiliary techniques which make the obtained expressions for traceless projectors accessible for applications. As an outcome of the proposed approach, in Section \ref{sec:splitting_idempotent} we relate the traceless projector to a particular central idempotent in $B_{n}(\delta)$ when the latter is semisimple. Section \ref{sec:summary} serves as a summary where we sum up all the introduced ingredients and construct the traceless projector on $V^{\otimes 4}$ step by step. Technical proofs and some examples are given in the Appendix section.

\section{Metric-preserving groups on tensor products \\and their centraliser algebras}\label{sec:basics}

In this section we recall some basic material about the classical Lie groups $GL(N)$ and $G(N)$ acting on tensor products, with the related Schur-Weyl dualities. Apart from merely fixing the notations, our aim here consists in motivating the usage of the Brauer algebra.
\vskip 4 pt

\paragraph{Tensor representations of classical Lie groups.} To a finite-dimensional $\mathbb{C}$-vector space  $V$ of dimension $\dim V = N$, one associates canonically the group $GL(N)\subset \mathrm{End}(V)$ of all invertible linear transformations of this space. Any pair of non-zero vectors are related by a $GL(N)$-transformation, which means that $V$ is a simple $GL(N)$-module\footnote{Throughout the paper we make no distinction between the terms {\it representation} and {\it module}, which are equivalent when one considers the action of a group/algebra on a vector space. In this context, a {\it simple} (respectively, {\it semisimple}) {\it module} corresponds to an {\it irreducible} (respectively, {\it completely reducible}) {\it representation}.}. The set $\mathrm{End}(V)$ is naturally a $GL(N)$-module under the adjoint action: $F\mapsto RFR^{-1}$ for $F\in \mathrm{End}(V)$ and $R\in GL(N)$.
\vskip 4 pt

If one fixes a basis $\{e_a\}\subset V$ ($a = 1,\dots, N$), any vector is identified with an element of $\mathbb{C}^N$ (the set of its components): $v = e_a v^a$. Here, by the Einstein's convention, for a pair of indices denoted by the same letter (one up and one down) one performs summation. To any $F\in\mathrm{End}(V)$ one associates a square matrix $F^{a}{}_{b}$ via $F(e_b) = e_a\,F^{a}{}_{b}$ (with a slight abuse of notation, we will use the same letter for an element of a vector space and its components with respect to a particular basis). There exists a unique, up to a factor, $GL(N)$-invariant linear function on $\mathrm{End}(V)$, which is called {\it trace}, which reads
\begin{equation}\label{eq:trace}
    \mathrm{tr} \,F = F^{a}{}_a\quad \text{for any $F\in\mathrm{End}(V)$.}
\end{equation}

If $V$ is additionally equipped with a non-degenerate bilinear form (metric) $\langle\,\cdot,\cdot\,\rangle$, symmetric or anti-symmetric, one identifies the group of metric-preserving linear transformations $G(N)\subset GL(N)$:
\begin{equation}\label{eq:orthogonal_group}
    \text{for any}\quad R\in G(N) \quad \text{holds}\quad \langle R(v),R(w)\rangle = \langle v,w\rangle\,.
\end{equation}
In the case of a symmetric metric one has $G(N) = O(N)$ (the orthogonal group), while an anti-symmetric metric leads to $G(N) = Sp(N)$ (the symplectic group) with $N \in 2\mathbb{N}$. As a $G(N)$-module $V$ is simple. Component-wise, the metric gives rise to a non-degenerate matrix $g_{ab} = \langle e_a,e_b\rangle$, whose inverse we denote by $g^{ab}$: one has $g_{ab} = \pm g_{ba}$ and $g^{ac}g_{cb} = \delta^{a}{}_b$. The metric fixes a particular isomorphism between $V$ and its dual space $V^{*}$, which allows transitions between vectors and co-vectors by raising and lowering the indices of the components:
\begin{equation}\label{eq:covariant-contravariant}
    v_{a} = g_{ab}v^{b}\quad \text{or, equivalently}\quad v^{a} = g^{ab} v_b\,.
\end{equation}
To this end, instead of having separately $V$ and $V^{*}$ it is sufficient to work only with the space $V$.
\vskip 4 pt

One can construct the $n$-fold tensor product $V^{\otimes n} = \underbrace{V\otimes \dots \otimes V}_n$ of the space $V$, with the basis $e_{a_1}\otimes \dots\otimes e_{a_n}$,
\begin{equation}\label{eq:tensor_basis}
    \text{so that any}\;\; T\in V^{\otimes n} \;\;\text{can be decomposed as}\;\; T = e_{a_1}\otimes \dots\otimes e_{a_n}\, T^{a_1\dots a_n}\,.
\end{equation}
In this case we will say that $T$ has rank $n$. The space $V^{\otimes n}$ is equipped with a $G(N)$-invariant non-degenerate scalar product in a canonical way:
\begin{equation}\label{eq:metric_tensor_product}
   \langle v_1\otimes \dots \otimes v_n, w_1\otimes\dots \otimes w_n\rangle = \langle v_1,w_1\rangle\dots \langle v_n,w_n\rangle\,.
\end{equation}
The latter allows one to identify the pairs of dual operators $F,F^{*}\in\mathrm{End}(V^{\otimes n})$ by the requirement
\begin{equation}\label{eq:adjoint}
    \left<T_1,F T_2\right> = \left<F^{*} T_1,T_2\right>\quad \text{(for all $T_1,T_2\in V^{\otimes n}$).}
\end{equation}
\vskip 4 pt

The trace \eqref{eq:trace}, in combination with the isomorphism $V\cong V^{*}$, can be viewed as a map $\mathrm{tr}^{(g)}:V^{\otimes 2} \to \mathbb{C}$. As a straightforward generalisation, one considers linear maps $\mathrm{tr}^{(g)}_{ij}: V^{\otimes n} \to V^{\otimes n - 2}$ (with $1\leqslant i<j\leqslant n$) which evaluate the scalar product between the two vectors at the $i$th and $j$th positions:
\begin{equation}\label{eq:traces_tensors}
    \mathrm{tr}^{(g)}_{ij} (v_1\otimes\dots \otimes v_n) = \langle v_i,v_j\rangle\; v_1\otimes \dots \otimes \bcancel{v_i}\otimes \dots \otimes \bcancel{v_j} \otimes \dots \otimes v_n\,.
\end{equation}
Component-wise, $\mathrm{tr}^{(g)}_{ij}$ acts on \eqref{eq:tensor_basis} as $t^{a_1\dots a_n} \mapsto g_{a_i a_j}\,t^{a_1\dots a_i\dots a_j \dots a_n}$. A tensor $T\in V^{\otimes n}$ is called {\it traceless} if for all $1\leqslant i<j\leqslant n$ the corresponding traces vanish, {\it i.e.} $\mathrm{tr}^{(g)}_{ij} T = 0$.
\vskip 4 pt

The space $V^{\otimes n}$ (for all $n\geqslant 1$) is a $GL(N)$-module, with an element $R\in GL(N)$ applied to each tensor component:
\begin{equation}\label{eq:diagonal_embedding}
    R(v_1\otimes \dots \otimes v_n) = Rv_1\otimes \dots \otimes Rv_n\,.
\end{equation}
In terms of the decomposition over the basis \eqref{eq:tensor_basis}, $T^{a_1\dots a_n} \mapsto R^{a_1}{}_{b_1}\dots R^{a_n}{}_{b_n}\,T^{b_1\dots b_n}$\,. As soon as $G(N)\subset GL(N)$, $V^{\otimes n}$ is a $G(N)$-module as well. In the sequel, whenever $V^{\otimes n}$ is considered as a $GL(N)$-module and a $G(N)$-module in parallel, we will call it a {\it module} without specifying the Lie group.
\vskip 4 pt

Due to the classical result of Weyl, the module $V^{\otimes n}$ is semisimple: it decomposes into a direct sum of simple $GL(N)$-modules $V^{(\mu)}$ (on the left) or simple $G(N)$-modules $D^{(\lambda)}$ (on the right): respectively,
\begin{equation}\label{eq:tensor_GL-O_decomposition}
    V^{\otimes n} \cong \bigoplus_{\mu\in \hat{\Sigma}_{n,N}} \big(V^{(\mu)}\big)^{\oplus g_{\mu}}\quad\quad\text{and}\quad\quad V^{\otimes n} \cong \bigoplus_{\begin{array}{c}
        \scriptstyle \lambda\in \hat{\Lambda}_{n-2f,N}\\
        \scriptstyle f = 0,\dots,\lfloor\tfrac{n}{2}\rfloor 
    \end{array}} \big(D^{(\lambda)}\big)^{\oplus h_{\lambda}}\,.
\end{equation}
Here $\lfloor x\rfloor$ denotes the integer part of $x$, while the index sets $\hat{\Sigma}_{n,N}$ and $\hat{\Lambda}_{l,N}$ are described in \eqref{eq:def_Lambda}, \eqref{eq:def_Sigma} and \eqref{eq:hat}.

\vskip 4 pt

In order to recall the structure of irreducible components $D^{(\lambda)}\subset V^{\otimes n}$, note that the property \eqref{eq:orthogonal_group} translated to the tensor 
\begin{equation}\label{eq:pure_trace}
    E = e_a\otimes e_b\, g^{ab}
\end{equation}
implies that the $\mathbb{C}$-span of any tensor power $E^{\otimes f}$ is a trivial $G(N)$-module and, as a consequence, the subspace $E^{\otimes f}\otimes V^{\otimes n - 2f}\subset V^{\otimes n}$ is invariant with respect to the action of $G(N)$. Further, invariant subspaces in $E^{\otimes f}\otimes V^{\otimes n - 2f}$ correspond exactly to invariant subspaces in $V^{\otimes n - 2f}$, which are either traceless or proportional to $E$.
\vskip 4 pt

\paragraph{Centraliser algebras.} A powerful tool for studying the decomposition of $V^{\otimes n}$ into simple components is the {\it centralizer algebra}. The latter is identified as the maximal subalgebra in $\mathrm{End}(V^{\otimes n})$ whose elements commute with any transformation of the Lie group. The prior fact about the centraliser algebras in question is that they are semisimple because the group action is \cite[Theorem 3.5.B]{Weyl}. As a representation of two mutually-commuting algebras of transformations, $V^{\otimes n}$ decomposes into a multiplicity-free direct sum of pairs of simple modules: one of the Lie group, the other of its centraliser algebra. More in detail, denote $\mathfrak{C}_n(N)$ and $\mathfrak{B}_n(N)$ the centraliser algebras for the action of the groups $GL(N)$ and $G(N)$ respectively, and let $\mathfrak{L}^{(\mu)}$ and $\mathfrak{M}_n^{(\lambda)}$ stand for the simple modules over $\mathfrak{C}_n(N)$ and $\mathfrak{B}_n(N)$ respectively. Then $V^{\otimes n}$ admits the following decompositions, with respect to $\big(GL(N),\mathfrak{C}_{n}(N)\big)$- and $\big(G(N),\mathfrak{B}_{n}(N)\big)$-actions:
\begin{equation}\label{eq:Schur-Weyl_raw}
    V^{\otimes n} \cong \bigoplus_{\mu\in \hat{\Sigma}_{n,N}} V^{(\mu)}\otimes \mathfrak{L}^{(\mu)}\quad\quad\text{and}\quad\quad
    V^{\otimes n} \cong \bigoplus_{\begin{array}{c}
        \scriptstyle \lambda\in \hat{\Lambda}_{n-2f,N}\\
        \scriptstyle f = 0,\dots,\lfloor\tfrac{n}{2}\rfloor 
    \end{array}} D^{(\lambda)}\otimes \mathfrak{M}_n^{(\lambda)}\,.
\end{equation}
By comparing \eqref{eq:tensor_GL-O_decomposition} and \eqref{eq:Schur-Weyl_raw}, $g_{\mu} = \dim \mathfrak{L}^{(\mu)}$ and $h_{\lambda} = \dim \mathfrak{M}_n^{(\lambda)}$, {\it i.e.} the elements of the latter module count the multiplicity of the former (and {\it vice versa}). 
\vskip 4 pt

\paragraph{Realisation of centraliser algebras.} The efficiency of applications of the centraliser algebras in question is due to the possibility of parametrising their elements via a ``somewhat enigmatic algebra'' \cite{Weyl}, whose action on $V^{\otimes n}$ is diagrammatic and takes care of certain adjacency relations among the factors $V$ in $V^{\otimes n}$, treating them as ``atomary'' and totally ignoring its ``intrinsic'' vector-space structure. From the fact that $G(N)\subset GL(N)$ one gets the inclusion $\mathfrak{C}_{n}(N)\subset \mathfrak{B}_{n}(N)$, so we concentrate on the latter algebra while the former will be treated via its embedding.
\vskip 4 pt

The algebra $\mathfrak{B}_n(N)$ is conveniently realised as a homomorphic image of an associative algebra $B_{n}(\varepsilon N)$ -- the Brauer algebra \cite{Brauer} (here and in what follows $\varepsilon = 1$ for $G(N) = O(N)$ and $\varepsilon = -1$ when $G(N) = Sp(N)$). The corresponding $\varepsilon$-dependent homomorphism $\mathfrak{r}$ is surjective, and it is also injective for $N\geqslant n$ when $\varepsilon = 1$ and $N\geqslant 2n$ when $\varepsilon = -1$ \cite{Brown_semisimplicity}. The action of $B_{n}(\varepsilon N)$ on $V^{\otimes n}$ is very intuitive (see Section \ref{sec:Brauer_diagrams}), so instead of working with the centraliser algebra $\mathfrak{B}_{n}(N)$ directly, we will address to its elements via the $\mathfrak{r}$-image of the Brauer algebra $B_{n}(\varepsilon N)$. In this respect, instead of the simple $\mathfrak{B}_{n}(N)$-modules $\mathfrak{M}^{(\lambda)}_n$ we will use the (isomorphic) simple $B_{n}(\varepsilon N)$-modules $M^{(\hat{\lambda})}_n$, where $\lambda\mapsto \hat{\lambda}$ is an $\varepsilon$-dependent involutive operation \eqref{eq:hat}.
\vskip 4 pt

The algebra $B_{n}(\varepsilon N)$ contains the symmetric group algebra $\mathbb{C}\Sn{n}$ as a subalgebra, and the $\mathfrak{r}$-image of the latter gives exactly $\mathfrak{C}_{n}(N)$ \cite{Weyl}. So in the sequel we work in terms of $\mathbb{C}\Sn{n}\subset B_{n}(\varepsilon N)$, and use the simple $\mathbb{C}\Sn{n}$-modules $L^{(\hat{\mu})}$ instead of $\mathfrak{L}^{(\mu)}$ (with the same map $\mu\mapsto \hat{\mu}$ as for the $B_{n}(\varepsilon N)$-modules above). 
\vskip 4 pt

Summarising the above, $V^{\otimes n}$ is a $\big(GL(N),\mathbb{C}\Sn{n}\big)$-bimodule or a $\big(G(N),B_{n}(\varepsilon N)\big)$-bimodule. The two decompositions in \eqref{eq:Schur-Weyl_raw} are rewritten, respectively, as follows:
\begin{equation}\label{eq:Schur-Weyl}
    V^{\otimes n} \cong \bigoplus_{\mu\in \hat{\Sigma}_{n,N}} V^{(\mu)}\otimes L^{(\hat{\mu})}\quad\quad\text{and}\quad\quad
    V^{\otimes n} \cong \bigoplus_{\begin{array}{c}
        \scriptstyle \lambda\in \hat{\Lambda}_{n-2f,N}\\
        \scriptstyle f = 0,\dots,\lfloor\tfrac{n}{2}\rfloor 
    \end{array}} D^{(\lambda)}\otimes M_n^{(\hat{\lambda})}\,.
\end{equation}
The interplay between simple $GL(N)$- and $\mathbb{C}\Sn{n}$-modules, as well as $G(N)$- and $B_{n}(\varepsilon N)$-modules, in the decompositions \eqref{eq:Schur-Weyl} is known as {\it Schur-Weyl duality}. It was originally established for the former decomposition \cite{Weyl}, and then adapted to the case of $G(N)$ through the seminal works \cite{Brauer,Brown_Br}, see also \cite{DipperDotyHu_SpBrauer} (Weyl utilised the elements of the $\big(GL(N),\mathbb{C}\Sn{n}\big)$-duality to study the decomposition of the $G(N)$-module $V^{\otimes n}$ \cite{Weyl}). 
\vskip 4 pt

Projection to a simple $GL(N)$-component $V^{(\mu)}$ (respectively, $G(N)$-component $D^{(\lambda)}$) is obtained by fixing an element in the corresponding simple module $L^{(\hat{\mu})}$ (respectively, $M_n^{(\hat{\lambda})}$). A particular well-known use of the classical Schur-Weyl duality consists in application of primitive idempotents in $\mathbb{C}\Sn{n}$ to $V^{\otimes n}$ ({\it e.g.} Young symmetrisers \cite[Chapter 4]{Ful_Har} or orthogonal primitive idempotents \cite{Jucys_idempotents,Murphy}) in order to obtain irreducible $GL(N)$-modules. Primitive orthogonal idempotents in $\mathfrak{B}_{n}(N)$ in terms of the Brauer algebra were constructed in \cite{Nazarov_tracelessness}.
\vskip 4 pt

Before going into technical details about the algebras $\mathbb{C}\Sn{n}\subset B_{n}(\varepsilon N)$ and their representation theory, let us declare
\vskip 4 pt

\noindent {\bf The main goal of the present work.} We aim at constructing the {\it universal traceless projector} $\mathfrak{P}_n$, which is specified by the following natural requirements. First of all, it is an idempotent in $\mathrm{End}(V^{\otimes n})$, {\it i.e.} $\mathfrak{P}_n \mathfrak{P}_n = \mathfrak{P}_n$. Second, for any $T\in V^{\otimes n}$, for all $1\leqslant i < j \leqslant n$:
\begin{equation}\label{eq:projector_properties}
\def\arraystretch{1.5}
    \begin{array}{ll}
        (\mathrm{P}1) & \mathrm{tr}^{(g)}_{ij}(\mathfrak{P}_n T) = 0\quad \text{(the image is traceless),}\\
        (\mathrm{P}2) & \text{if}\quad \mathrm{tr}^{(g)}_{ij}T = 0\,,\quad \text{then}\quad \mathfrak{P}_n T = T\quad \text{(among traceless projectors, $\mathfrak{P}_n$ has maximal rank),}\\
        (\mathrm{P}3) & \text{for any}\;\; s\in\Sn{n}\,,\quad \mathfrak{P}_n\mathfrak{r}(s) \,T = \mathfrak{r}(s)\mathfrak{P}_n \,T \quad\text{($\mathfrak{P}_n$ preserves symmetries of a tensor).}
    \end{array}
\end{equation}
The prior positive fact is that, given $n$ and $N$, the traceless projector satisfying $(\mathrm{P}1)$ and $(\mathrm{P}2)$ exists and is unique \cite[Theorem 5.6A]{Weyl} (see also \cite[Chapter 10]{Hamermesh:100343}). Let us briefly note that the proof is purely representation-theoretic and simply relies on the fact that in a representation space equipped with a non-degenerate invariant scalar product, for any invariant subspace its orthogonal complement is an invariant subspace as well. If one takes the invariant subspace $W\subset V^{\otimes n}$ of all tensors proportional to the metric, then its orthogonal complement $W^{\prime}$ is exactly the subspace of traceless tensors. The fact that $\mathfrak{P}_n$ performs an orthogonal projection assures that $(\mathfrak{P}_n)^{*} = \mathfrak{P}_n$.
\vskip 4 pt

The properties $(\mathrm{P}1)$, $(\mathrm{P}2)$ imply that $\mathfrak{P}_n$ is proportional to identity on the simple $G(N)$-submodules in $V^{\otimes n}$: it acts by identity on those simple modules which occur in $W^{\prime}$ and annihilates all the others. As a result, $\mathfrak{P}_n\in \mathfrak{B}_n(N)$. Note also that permutations of tensor factors preserves the decomposition $V^{\otimes n} \cong W\oplus W^{\prime}$. From the fact that $\mathfrak{P}_n$ is proportional to identity on each direct component and decomposing them into simple $\mathbb{C}\Sn{n}$-modules, one concludes that the property $(\mathrm{P}3)$ holds as a consequence of $(\mathrm{P}1)$ and $(\mathrm{P}2)$. Nevertheless we choose to keep it explicit because it will serve as a starting point in our construction in Section \ref{sec:conjugacy_classes}.
\vskip 4 pt

To this end, one naturally arrives at the idea of application of the Brauer algebra to the construction of $\mathfrak{P}_n$. The properties \eqref{eq:projector_properties} can be verified in a systematic manner, which leads to an element $P_n\in\ B_{n}(\varepsilon N)$ such that $\mathfrak{P}_n = \mathfrak{r}(P_n)$. This result is presented in Theorem \ref{thm:projector}.

\section{The Brauer algebra}\label{sec:Brauer_algebra}
\subsection{Diagrammatic representation and action on tensors}\label{sec:Brauer_diagrams}

\paragraph{The algebra of Brauer diagrams.} In order to introduce the Brauer algebra $B_{n}(\delta)$ (for any $\delta\in \mathbb{C}$) it is convenient to start by recalling the diagrammatic representation of the symmetric group. To any element $s\in\Sn{n}$ one associates a {\it permutation diagram}: two rows of $n$ vertices (nodes) aligned horizontally and placed one above another, with the $i$th node in the upper row and the $s(i)$th node in the lower row joined by a line.
\begin{equation}
s = \ \begin{pmatrix}
    	1 & 2 & 3 & 4 \\
    	3 & 1 & 4 & 2
 \end{pmatrix} \mapsto \raisebox{-.45\height}{\includegraphics[width=60pt,height=60pt]{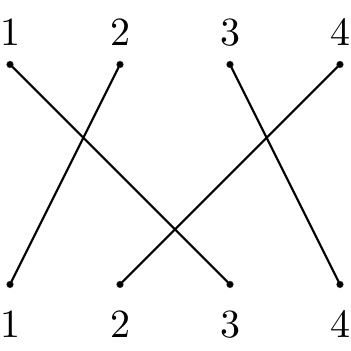}}\quad,\hspace{1cm} t = \ \begin{pmatrix}
    1 & 2 & 3 & 4 \\
    4 & 1 & 2 & 3
  \end{pmatrix} \mapsto \raisebox{-.45\height}{\includegraphics[width=60pt,height=60pt]{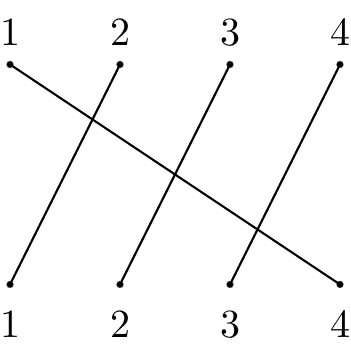}} 
\end{equation}

In order to calculate the product $st$ of two elements $s,t\in\Sn{n}$ diagrammatically, one places the diagram $s$ below $t$ and identifies the upper nodes of the former with the lower nodes of the latter. Then straightening the lines gives the diagram associated with the element $st$:
\begin{equation}
s\, t   \;=\; \raisebox{-.45\height}{\includegraphics[width=60pt,height=56pt]{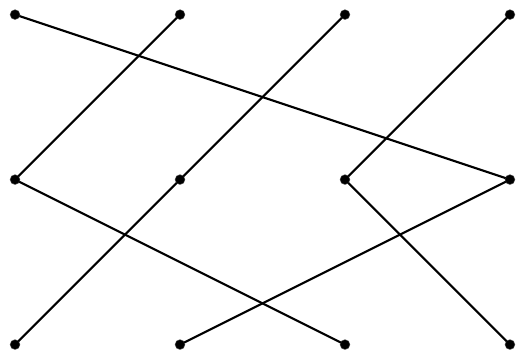}}\;=\; \raisebox{-.45\height}{\includegraphics[width=60pt,height=40pt]{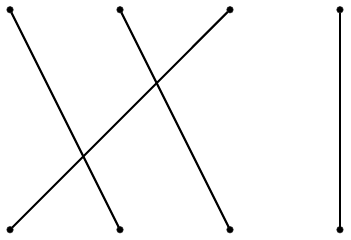}}
\end{equation}

\noindent In the sequel we make no distinction between the elements of $\Sn{n}$ and the corresponding diagrams. Passing to $\mathbb{C}\Sn{n}$ consists in allowing linear combinations of diagrams treated as basis vectors. 
\vskip 4 pt

As a vector space, the Brauer algebra $B_{n}(\delta)$ is obtained by extending the basis of permutation diagrams to {\it Brauer diagrams}. Namely, for the same set of nodes as for the permutation diagrams, each node is required to be an endpoint of exactly one line. For example, one can consider the following two Brauer diagrams $b_1$ and $b_2$ for $n = 5$. 
\begin{equation}
b_1 \,=\,   \raisebox{-.45\height}{\includegraphics[width=60pt,height=40pt]{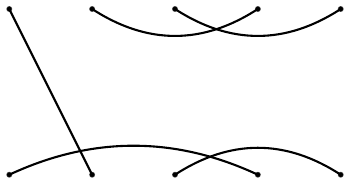}}\;,\hspace{1cm} 
b_2\, =\, \raisebox{-.45\height}{\includegraphics[width=60pt,height=40pt]{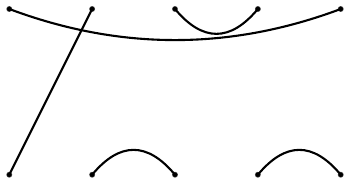}}
\end{equation}

\noindent By a straightforward combinatorial computation one finds $\dim B_{n}(\delta) = \frac{(2n)!}{2^{n}\, n!} = (2n-1)!!$.
\vskip 4 pt

The lines joining an upper node with a lower node will be referred to as {\it vertical lines}, while the lines joining a pair of nodes in the same row will be referred to as {\it arcs}. We will say that a diagram contains $f \geqslant 0$ arcs if there is $f$ arcs in either of its rows. By construction, $\mathbb{C}\Sn{n}$ is embedded in $B_{n}(\delta)$ via diagrams containing $0$ arcs.

\vskip 4 pt

The product of two diagrams $b_1 b_2$ is parametrised by a variable $\delta\in \mathbb{C}$ and defined in a similar fashion as before. Place $b_1$ below $b_2$, identify the upper nodes of the former with the lower nodes of the latter. If $\ell\geqslant 0$ is the number of loops, the resulting vector $b_1 b_2$ is obtained by omitting the loops, straightening the lines and multiplying the so-constructed diagram by $\delta^{\ell}$.
\begin{equation}
b_1 \, b_2 \; =\; \raisebox{-.45\height}{\includegraphics[width=60pt,height=55pt]{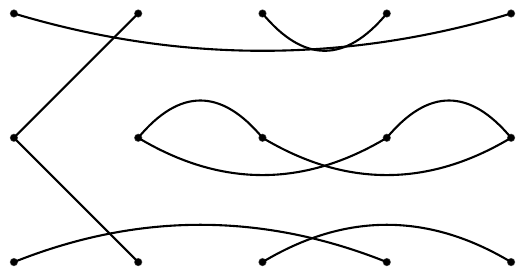}}\hspace{0.3cm}\;=\; \hspace{0.3cm} \delta\; \raisebox{-.45\height}{\includegraphics[width=60pt,height=40pt]{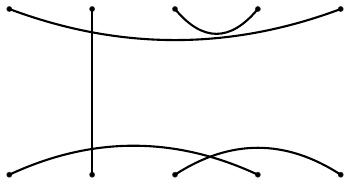}}
\end{equation}
\noindent It can be proven that the so-defined product in $B_{n}(\delta)$ is associative.
\vskip 4 pt

From the definition of the product of Brauer diagrams it follows that the number of arcs in a diagram can not be reduced via multiplication by other diagrams. Let $J^{(f)} \subset B_{n}(\delta)$ denote the span of all diagrams with at least $f$ arcs. These subspaces are two-sided ideals in $ B_{n}(\delta)$ which form a chain of embeddings:
\begin{equation}\label{eq:cell_ideals}
    B_{n}(\delta) = J^{(0)} \supset J^{(1)} \supset \dots \supset J^{\left(\lfloor\frac{n}{2}\rfloor\right)}\,.
\end{equation}
The above chain is the major tool in studying the algebra $B_{n}(\delta)$ since the seminal works \cite{Brown_Br,Brown_semisimplicity}. Note also the simple fact that 
\begin{equation}\label{eq:arc_factor}
    B_{n}(\delta)\slash J^{(1)} \cong \mathbb{C}\Sn{n}\,.
\end{equation}
\vskip 4 pt

\paragraph{Generators and relations.} The algebra $B_n(\delta)$ is generated by the following set of  diagrams $s_i$ and $d_i$ ($i=1\dots n-1$): 

\begin{equation}
s_{i} \;=\; \raisebox{-.4\height}{\includegraphics[width=130pt,height=50pt]{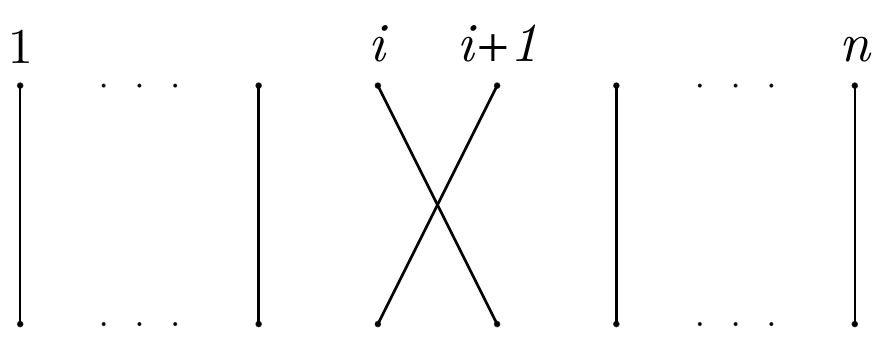}}\;, \quad\quad d_{i} \;= \;\raisebox{-.4\height}{\includegraphics[width=130pt,height=50pt]{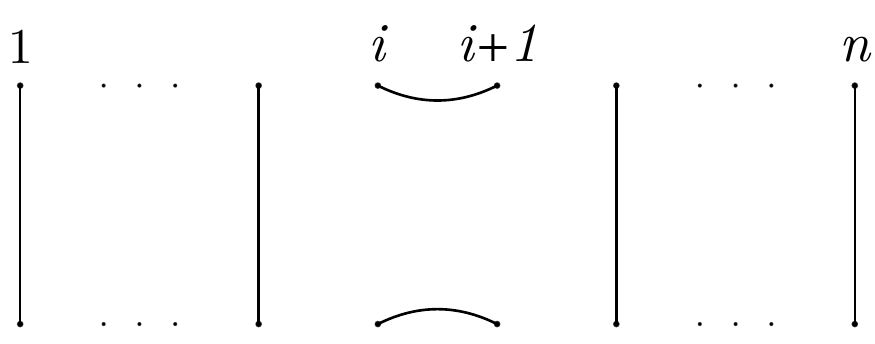}}
\end{equation}

\noindent The generators verify the following set of defining relations\footnote{Equivalently, one can define the algebra $B_n(\delta)$ as generated by $s_i,d_i$ ($i = 1,\dots, n-1$) modulo the defining relations \eqref{eq:rel_1}-\eqref{eq:rel_4}, and prove that it is isomorphic to the algebra of Brauer diagrams introduced above.}:
\begin{align}
    s_i{}^2 & = 1\,,\quad d_i{}^2 = \delta\,d_i\,,\quad d_i s_i = s_i d_i = d_i\,, \label{eq:rel_1}\\
 s_i s_j = s_j& s_i\,, \quad d_i s_j = s_j d_i\,, \quad d_i d_j = d_j d_i\quad \text{for}\quad |i-j|\geqslant 2\,,\label{eq:rel_2}\\
    s_i s_{i+1} s_i & = s_{i+1} s_i s_{i+1}\,,\quad d_{i} d_{i+1} d_{i} = d_i\,,\quad d_{i+1} d_i d_{i+1} = d_{i+1}\,,\label{eq:rel_3}\\
    s_{i} d_{i+1}d_{i} & = s_{i+1}d_{i}\,,\quad d_{i+1} d_{i} s_{i+1} = d_{i+1} s_{i}\,,\label{eq:rel_4}
\end{align}
while any other relation in $B_{n}(\delta)$ is a result of a composition of the relations \eqref{eq:rel_1}-\eqref{eq:rel_4}. The elements $s_i$ satisfy the known relations for simple transpositions in the symmetric group $\Sn{n}$, so linear combinations of their products constitutes the symmetric group algebra $\mathbb{C}\Sn{n}$.
\vskip 4 pt

We also introduce the following elements, for all possible $i<j$:

\begin{equation}
s_{ij} \; =\; \raisebox{-.4\height}{\includegraphics[width=140pt,height=50pt]{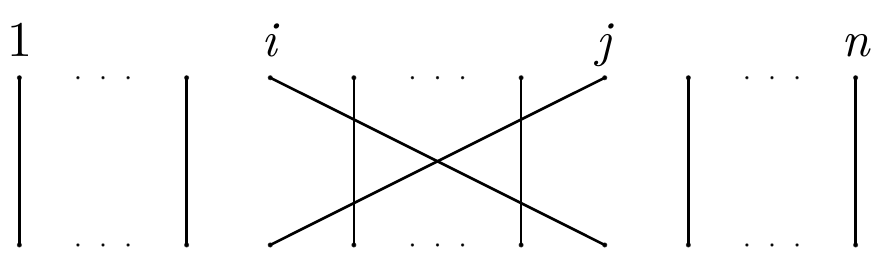}}\;,\quad\quad d_{ij} \;=\; \raisebox{-.4\height}{\includegraphics[width=140pt,height=50pt]{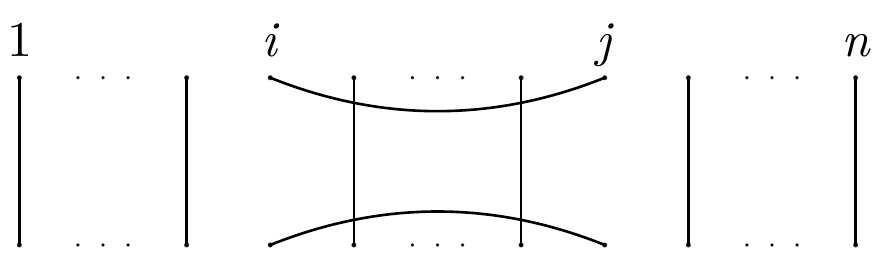}}
\end{equation}

\vskip 4 pt

\paragraph{Action on $V^{\otimes n}$.} Application of $B_{n}(\delta)$ in the context of Schur-Weyl-type dualities for the group $G(N)$ takes place for integer values of the parameter $\delta = \varepsilon N$, where $\varepsilon = +1$, $N\in \mathbb{N}$ when $G(N) = O(N)$ and $\varepsilon = -1$, $N\in 2\mathbb{N}$ when $G(N) = Sp(N)$. To describe the left\footnote{In the present work we consider the left action of $B_{n}(\varepsilon N)$. In the context of bi-modules, it is also common to consider the right action defined as $\mathfrak{r}_{\mathrm{op}}(b) = \mathfrak{r}(b^{*})$, where $b^{*}$ is the diagram obtained by reflecting $b$ with respect to the horizontal middle line. In this case, $\mathfrak{r}_{\mathrm{op}}(b_1)\mathfrak{r}_{\mathrm{op}}(b_2) = \mathfrak{r}_{\mathrm{op}}(b_2 b_1)$.} action of any diagram $b\in B_{n}(\varepsilon N)$ on $V^{\otimes n}$, denote $\mathtt{i}(b)$ the minimal number of intersections among its lines. Then the transformation $\mathfrak{r}(b)$ is applied  to $v_1\otimes \dots  \otimes v_n$ as follows (see Theorem 2.10 in \cite{HW_decomposition_Br}): place the tensor factors in the upper nodes of $b$ and perform {\it i)} permutations by following the passing lines from top to bottom, {\it ii)} contractions of pairs joined by arcs in the upper row, {\it iii)} insertion of $E = e_{a}\otimes e_{b}\,g^{ab}$ at each pair of the lower nodes joined by an arc; the result is multiplied by $\varepsilon^{\mathtt{i}(b)}$. For example, for the diagrams $b_1, b_2$ introduced above one has: 
\begin{equation}
\def\arraystretch{1.4}
\begin{array}{l}
    \mathfrak{r}(b_1)\, v_1\otimes v_2 \otimes v_3 \otimes v_4 \otimes v_5 = \varepsilon\,e_a\otimes v_1\otimes e_b\otimes e_c \otimes e_d\; g^{ac} g^{bd} \langle v_2,v_4\rangle \langle v_3,v_5\rangle\,,\\
    \mathfrak{r}(b_2)\, v_1\otimes v_2 \otimes v_3 \otimes v_4 \otimes v_5 = \varepsilon\,v_2\otimes e_a\otimes e_b \otimes e_c\otimes e_d\; g^{ab} g^{cd} \langle v_1,v_5\rangle \langle v_3,v_4\rangle\,.
\end{array}
\end{equation}
The action of the generators $s_i$, $d_i$ is particularly simple:
\begin{equation}\label{eq:action_generators}
\begin{array}{l}
    \mathfrak{r}(s_i)\,v_1\otimes \dots\otimes v_i\otimes v_{i+1}\otimes\dots \otimes v_n = \varepsilon\, v_1\otimes \dots\otimes v_{i+1}\otimes v_i\otimes\dots \otimes v_n\,,\\
    \mathfrak{r}(d_i)\,v_1\otimes \dots\otimes v_i\otimes v_{i+1}\otimes\dots \otimes v_n = \langle v_i,v_{i+1}\rangle\, v_1\otimes \dots\otimes E \otimes\dots \otimes v_n\,.
\end{array}
\end{equation}
One can check that the above transformations \eqref{eq:action_generators} verify the relations \eqref{eq:rel_1}-\eqref{eq:rel_4}, and hence $\mathfrak{r}$ is indeed the homomorphism of algebras. As was pointed out, the homomorphism is injective when $N\geqslant n$ for $\varepsilon = 1$ and $N\geqslant 2n$ for $\varepsilon = -1$ \cite{Brown_semisimplicity}. When it is not the case, the non-trivial kernel is generated by particular anti-symmetrisations of tensor factors \cite[Theorem 3.7]{Gavarini_LitRich} (in this relation, see also \cite[Theorems 2.17.A and 6.1.B]{Weyl}).
\vskip 4 pt

The diagrammatic representation of $B_{n}(\varepsilon N)$, when translated to the action on tensors by the homomorphism $\mathfrak{r}$, elucidates the operations on the components of the tensor products (or, simply, indices). Due to this convenience, in the sequel we will often consider $V^{\otimes n}$ directly as a module over $B_{n}(\varepsilon N)$ and treat the two notations for the action of $x\in B_{n}(\varepsilon N)$ on $T\in V^{\otimes n}$ -- $\mathfrak{r}(x)T$ and $x(T)$ -- on equal footing.
\vskip 4 pt

The operation $(\,\cdot\,)^{*}$ defined in \eqref{eq:adjoint} is an involutive anti-automorphism (or anti-involution) of the algebra $\mathrm{End}(V^{\otimes n})$. There is a natural involutive anti-automorphism of the Brauer algebra given by flipping each diagram with respect to the horizontal middle line. As a matter of a simple check, the $\mathfrak{r}$-image of a flipped diagram coincides with the adjoint to the $\mathfrak{r}$-image of the initial diagram. Thus, with a slight abuse of notation,
\begin{equation}\label{eq:flip}
    \text{for any diagram $b\in B_{n}(\varepsilon N)$ denote its flip $b^{*}$. Then one has $\mathfrak{r}(b^{*}) = \mathfrak{r}(b)^{*}$.}
\end{equation}
Note that for any $s\in \Sn{n}$, $s^{*} = s^{-1}$. The action of $(\,\cdot\,)^{*}$ on linear combinations of diagrams is defined by linearity.

\subsection{The centraliser of $\mathbb{C}\Sn{n}$}\label{sec:conjugacy_classes}

Two diagrams $b,b^{\prime}\in B_{n}(\delta)$, are said to be {\it conjugate}, which is denoted $b\sim b^{\prime}$, if there exists an element $s\in\Sn{n}$ (a monomial in generators $s_i$) such that $b^{\prime} = s b s^{-1}$. Two Brauer diagrams are conjugate iff they are related by a permutation of pairs of vertically aligned nodes.
\vskip 4 pt

Define the {\it average} of $b\in B_{n}(\delta)$ as follows:
\begin{equation}\label{eq:average}
    \e_b = \sum_{s\in\Sn{n}} s b s^{-1} \,.
\end{equation}
Let $C_{\Sn{n}}(b)$ denote the centralizer of a diagram $b\in B_{n}(\delta)$ in $\Sn{n}$, {\it i.e.} the set of elements $r\in\Sn{n}$ such that $rbr^{-1} = b$. The cardinalities of centralisers of two conjugate diagrams coincide, so one has
\begin{equation}\label{eq:normalised_classes}
    \e_b  = \big|C_{\Sn{n}}(b)\big|\,\sum_{b^{\prime}\sim b} b^{\prime} \,.
\end{equation}
The subalgebra $C_n(\delta)$ spanned by all averages forms the centraliser of $\mathbb{C}\Sn{n}$ in $B_{n}(\delta)$:
\begin{equation}\label{eq:centraliser}
  \text{for any } s\in\Sn{n}\;\; \text{holds}\;\; us = su\quad\Leftrightarrow\quad  u = \gamma_{b}\;\; \text{for some $b\in B_{n}(\delta)$}\,.  
\end{equation}
The proof of the above implication from left to right is straightforward, while for the opposite implication note that for the decomposition over the diagram basis in $B_{n}(\delta)$,
\begin{equation}
    u = \sum_{b\in B_{n}(\delta)} c_{b}\,b\,,
\end{equation}
the condition $us = su$ for all $s\in \Sn{n}$ is equivalent to $c_{s b s^{-1}} = c_{b}$ for each diagram $b$ in the above decomposition.
\vskip 4 pt

\paragraph{The principal building block of $\mathfrak{P}_n$.} The significance of the algebra $C_{n}(\delta)$ for the purpose of constructing the universal traceless projector is due to $(\mathrm{P}3)$ in \eqref{eq:projector_properties}, so it is natural to expect that 
\begin{equation}
    \mathfrak{P}_n = \mathfrak{r}(P_n)\quad\text{for some}\quad P_n\in C_{n}(\delta)\,.
\end{equation}
In this respect, one can look for $P_n$ as a linear combination of averages \eqref{eq:average} and fix the coefficients from the requirement $\mathfrak{r}(J^{(1)} P_{n}) = 0$ (this procedure was analysed in \cite{KMP_central_idempotents} for the construction of particular idempotents in $B_{n}(\delta)$, see Section \ref{sec:splitting_idempotent}). We choose a different direction and invoke representation theory of the Brauer algebra, focusing on the following normalised average: 
\begin{equation}\label{eq:master_class}
    A_n = \frac{1}{2(n-2)!}\,\e_{d_1} = \sum_{1\leqslant i<j\leqslant n} d_{ij} \in C_n(N)\cap J^{(1)}\,,\quad (A_n)^{*} = A_n\,.
\end{equation}

\begin{lemma}\label{lem:KerA}
    The action of $A_n$ on $V^{\otimes n}$ is diagonalisable. The subspace $\mathrm{Ker}\,A_n \subset V^{\otimes n}$ is exactly the space of traceless tensors, while non-zero eigenvalues of $A_n$ are in $\varepsilon \mathbb{N}$.
\end{lemma}
\begin{proof}
    We concentrate on the case $G(N) = O(N)$ (with $\varepsilon = 1$), which is instructive and transparent at the same time. The case $G(N) = Sp(N)$ (with $\varepsilon = -1$) utilises the same idea, but is more involved technically, so we postpone it to Appendix \ref{app:KerA_proof_Sp}. Fix a real structure in $V$ such that $\langle\,\cdot,\cdot\,\rangle$ (as defined in \eqref{eq:metric_tensor_product}) has Euclidean signature, and denote $V_{\mathbb{R}}$ the real subspace.
    \vskip 4 pt
    
    The action of $d_{ij}$ on decomposable tensors is given by: 
    \begin{equation}\label{eq:Im_d}
        \mathfrak{r}(d_{ij})\, v_{1}\otimes \dots\otimes v_{n}=\langle v_{i},v_{j}\rangle\, g^{a_i a_j} v_{1}\otimes\dots\otimes e_{a_i}\otimes\dots\otimes e_{a_j}\otimes \dots\otimes v_{n}\,,
    \end{equation}
    so for two tensors $T_1,T_2\in V_{\mathbb{R}}^{\otimes n}$ one has
    \begin{equation}\label{eq:scalar_product_tensors_trace}
        \big< T_1,d_{ij}(T_2)\big> =\big< \mathrm{tr}^{(g)}_{ij} T_1,\mathrm{tr}^{(g)}_{ij} T_2\big> = \big< d_{ij}(T_1),T_2\big> \quad \text{(for any $i<j$)}\,,
    \end{equation}
    where the trace operation was defined in \eqref{eq:traces_tensors}. Thus, $\mathfrak{r}(d_{ij})$ (for any $i<j$) is self-adjoint, so it is diagonalisable with respect to some orthonormal basis. Since the scalar product is positive-definite, \eqref{eq:scalar_product_tensors_trace} implies that the eigenvalues of $\mathfrak{r}(d_{ij})$ are non-negative. As a result, the action of $A_n = \sum_{1\leqslant i < j \leqslant n} d_{ij}$ on $V^{\otimes n}$ is self-adjoint, and therefore diagonalisable with respect to some orthonormal basis. With this at hand, it is a simple exercise to prove that its eigenvalues are non-negative as well.
    \vskip 4 pt
    
    Let us show that $T\in V_{\mathbb{R}}^{\otimes n}$ is traceless iff $A_n(T) = 0$. One way is simple: any traceless tensor is in $\mathrm{Ker}\,A_n$. Other way around, assume $A_n(T) = 0$ and use \eqref{eq:scalar_product_tensors_trace} to write
    \begin{equation}
        0 = \big<T,A_n(T)\big> = \sum_{1\leqslant i < j \leqslant n} \big< \mathrm{tr}^{(g)}_{ij} T,\mathrm{tr}^{(g)}_{ij} T\big>\,.
    \end{equation}
    Each term on the right-hand-side is non-negative, so the whole sum vanishes only if each term vanishes individually, which is the case only when $\mathrm{tr}_{ij}^{(g)}T = 0$. Taking into account that $V \cong \mathbb{C} \otimes V_{\mathbb{R}}$, the above arguments clearly extend to $V^{\otimes n}$.
\end{proof}
\noindent Our next goal consists in describing the properties of $A_n$ from the point of view of the representation theory of $B_{n}(\varepsilon N)$. In particular, the latter will provide the eigenvalues of $A_n$ on $V^{\otimes n}$. 
\vskip 4 pt

\subsection{Simple $\mathbb{C}\Sn{n}$- and $B_{n}(\varepsilon N)$-modules in $V^{\otimes n}$}\label{sec:modules}

Combinatorial and algebraic properties of simple $\mathbb{C}\Sn{n}$- and $B_{n}(\varepsilon N)$-modules are expressed via partitions of integers, which leads inevitably to Young diagrams and makes it natural to start by recalling some related basics.
\vskip 4 pt

\paragraph{Partitions and Young diagrams.} A $r$-partition $\lambda$ of an integer $n\in\mathbb{N}$, denoted as $\lambda\vdash n$, 
is a weakly decreasing sequence of positive integers $\lambda = (\lambda_1,\dots,\lambda_r)$ such that $|\lambda| = \sum_{i=1}^{r} \lambda_i = n$. If each $\lambda_i$ is even, the partition $\lambda$ is said to be {\it even}. Partitions are conveniently identified with the {\it Young diagrams}: arrays of squares placed at matrix entries $(i,j)$, such that for a given $r$-partition $\lambda$, $1\leqslant i\leqslant r$ and $1\leqslant j\leqslant \lambda_i$. To any partition $\lambda = (\lambda_1,\dots,\lambda_r)$ of $n$ one constructs the {\it dual partition} $\lambda^{\prime} = (\lambda_1^{\prime},\dots,\lambda^{\prime}_{\lambda_1})$ by transposing the corresponding Young diagram. For convenience, partitions and the corresponding Young diagrams are identified.
For example, partitions of $n=4$ are: 
\begin{equation}\label{eq:partitions}
\left( 4 \right)=\Yboxdim{9pt}\yng(4) \qquad \left( 3,1 \right)=\Yboxdim{9pt}\yng(3,1) \qquad  \left( 2^2 \right)=\Yboxdim{9pt}\yng(2,2) \qquad \left( 2,1^2 \right)=\Yboxdim{9pt}\yng(2,1,1) \qquad \left( 1^4 \right)=\Yboxdim{9pt}\yng(1,1,1,1)
\end{equation}
with two even partitions $(4)$ and $(2,2)$ among them. For dual partitions one has $(4)^{\prime} = (1^4)$, $(3,1)^{\prime} = (2,1^2)$ and $(2^2)^{\prime} = (2^2)$. For completeness, one also considers the partition of zero (the empty partition) $\varnothing$, with the empty set of boxes for the corresponding Young diagram. By definition, the empty partition is even.
\vskip 5 pt

The set of all Young diagrams is weakly ordered by inclusion: $\lambda\subset \mu$ implies that any box of $\lambda$ is also present in $\mu$. For a pair $\lambda\subset\mu $ define the {\it skew-shape Young diagram} $\mu\backslash \lambda$ as a set-theoretical difference of the corresponding Young diagrams. We set by definition $|\mu\backslash \lambda| = |\mu| - |\lambda|$. For example, 
\begin{equation}\label{eq:skew-shape_example}
\text{given}\quad\mu=\Yboxdim{9pt}\yng(4,2,2,1) \quad\text{and}\quad \lambda=\Yboxdim{9pt}\yng(2,1)\,,\quad\text{one has}\quad \mu\backslash \lambda =\Yboxdim{9pt}\young(\times\times ~~,\times ~,~~,~)
\end{equation}
\vskip 4 pt
\noindent
To each box of a (skew-shape) Young diagram at a position $(i,j)$ we associate its {\it content} $c(i,j) = j - i$. In this respect we define the content of any Young diagram as a sum:
\begin{equation}
    c(\lambda) = \sum_{(i,j)\in \lambda} c(i,j)\,.
\end{equation}
One defines the content of a skew-shape Young diagram to be $c(\mu\backslash\lambda) = c(\mu)-c(\lambda)$. For example, for the contents of the skew shape in \eqref{eq:skew-shape_example} one has
\begin{equation*}
    \mu\backslash \lambda =\Yboxdim{9pt}\young(\times\times \two\three,\times \zero,\mitwo \mione,\mithree) \quad \Rightarrow \quad c(\mu\backslash \lambda) = -1\,.
\end{equation*}
\vskip 4 pt

\paragraph{Littlewood-Richardson rule.} The $\mathbb{N}$-span of Young diagrams is endowed with the structure of an associative commutative monoid with the unit element given by $\varnothing$. For the product of two diagrams $\lambda,\nu$ we shall write 
\begin{equation}\label{eq:LRRule}
    \lambda \LRp \nu = \sum_{\mu} c^{\mu}_{\lambda\nu}\,\mu\,.
\end{equation}
The structure constants $c^{\mu}_{\lambda\nu}$ are referred to as Littlewood-Richardson coefficients. They are calculated via the {\it Littlewood-Richardson rule}, which admits a number of equivalent ways to formulate it in terms of {\it semi-standard tableaux} \cite{Fulton}. Let us recall that a tableau of shape $\mu\backslash\lambda$ is any map which associates a positive integer to each box of $\mu\backslash\lambda$. A tableau is called semi-standard if the numbers in each row (respectively, column) form a weakly (respectively, strongly) increasing sequence. We will say that a partition $\nu = (\nu_1\geqslant \dots\geqslant \nu_r)$ is a {\it weight} of a semi-standard tableau $\mathsf{t}(\mu\backslash\lambda)$ if the latter contains exactly $\nu_i$ occurrences of the entry $i$. To any tableau $\mathsf{t}$ one associates a {\it row word} $w(\mathsf{t})$ by reading the entries of boxes along each line from left to right proceeding from the bottom line to the top one. A word is called {\it Yamanouchi word} (equivalently, Littlewood-Richardson word or reverse lattice word) if any its suffix contains at least as many $1$'s as $2$'s, at least as many $2$'s as $3$'s, {\it etc.} For example,
\begin{equation}\label{eq:example_Yamanouchi}
\def\arraystretch{1.4}
\begin{array}{l}
    \text{for}\;\;\mu\backslash \lambda =\Yboxdim{8pt}\young(\times\times~~,\times~,~)\;\;\text{and the weight}\;\;\nu = \Yboxdim{8pt}\young(~~~,~)\\
    \text{one has exactly two semistandard tableaux}\;\; \mathsf{t}_1 = \Yboxdim{8pt}\young(\times\times\one\one,\times\one,\two)\,,\;\;\mathsf{t}_2 = \Yboxdim{8pt}\young(\times\times\one\one,\times\two,\one)\,,\\
    \text{such that the row words are Yamanouchi words}\;\;w(\mathsf{t}_1) = 2111\;\;\text{and}\;\;w(\mathsf{t}_1) = 1211\,.
\end{array}
\end{equation}
\vskip 4 pt
\noindent With this at hand, one arrives at the following definition of Littlewood-Richardson coefficients:
\begin{equation}\label{eq:definition_LR_1}
\def\arraystretch{1.4}
\begin{array}{ll}
    c^{\mu}_{\lambda\nu} & \text{is the number of semi-standard tableaux of the shape}\;\;\mu\backslash\lambda\;\;\text{and weight}\;\;\nu\\
    \hfill & \text{whose row word is a Yamanouchi word.}
\end{array}
\end{equation}
If either $\lambda\not \subset \mu$ or $|\nu| \neq |\mu| - |\lambda|$, one puts $c^{\mu}_{\lambda\nu} = 0$. It appears that $c^{\mu}_{\lambda\nu} = c^{\mu}_{\nu\lambda}$, so $c^{\mu}_{\lambda\nu} \neq 0$ also implies $\nu \subset \mu$. From the above example one obtains $c^{(4,2,1)}_{(2,1),(3,1)} = 2$.
\vskip 4 pt

We will also need an equivalent definition of Littlewood-Richardson coefficients based on the {\it jeu de taquin}. Recall that a corner of a Young diagram is any box with no other boxes on the right and below. An inside corner of a skew-shape $\mu\backslash\lambda$ (with $\lambda \subset \mu$) is a corner of $\lambda$ which is not a corner of $\mu$. Any chosen inner corner of a semi-standard tableau $\mathsf{t}$ (thought as an empty box) can be removed by the following {\it sliding process}: at any step consider the neighbour(s) on the right and below the empty box, then slide the smallest one into the empty box, while if the two are equal, slide the one below. The process continues until the empty box becomes a corner of $\mu$, and is removed afterwards. The resulting tableau is again semi-standard, so one can repeatedly perform the sliding process until the shape of the semi-standard tableau becomes a Young diagram. The whole process is called jeu de taquin, and the resulting semi-standard tableau $\mathsf{t}^{\prime}$ is the same for any order of processing the inner corners. It is called the {\it rectification of} $\mathsf{t}$, $\mathsf{t}^{\prime} = \mathrm{Rect}(\mathsf{t})$. For example, for the tableau $\mathsf{t}_1$ in \eqref{eq:example_Yamanouchi} one has
\begin{equation}\label{eq:example_Rect}
\def\arraystretch{2.4}
\begin{array}{ll}
    \mathsf{t}^{\prime}_1 = \mathrm{Rect}\left(\Yboxdim{8pt}\young(\times\times\one\one,\times \one,\two)\right) =  \Yboxdim{8pt}\young(\one\one\one,\two)\,,\;\;\text{namely,} \;\;& \Yboxdim{8pt}\young(\times\times\one\one,\times \one,\two) \;\;\rightarrow \;\;\Yboxdim{8pt}\young(\times \one \one \one,\times \times,\two) \;\; \rightarrow \;\;\Yboxdim{8pt}\young(\times\one \one \one,\times,\two)\,,\\
    \hfill & \Yboxdim{8pt}\young(\times \one \one\one,\times ,\two) \;\;\rightarrow\;\; \young(\times \one \one\one,\two ,\times) \;\; \rightarrow \;\; \young(\times \one \one\one,\two)\;\;,\\
    \hfill & \Yboxdim{8pt}\young(\times \one \one\one,\two)\;\;\rightarrow\;\;\young(\one \times\one \one,\two)\;\;\rightarrow\;\;\young(\one \one \times\one,\two)\;\; \rightarrow\;\;\young(\one \one \one\times,\two)\;\;\rightarrow\;\;\young(\one \one \one,\two)\,.
\end{array}
\end{equation}
One can check that for the other tableau in \eqref{eq:example_Yamanouchi} one has the same $\mathrm{Rect}(\mathsf{t}_2) = \mathsf{t}^{\prime}_1$.
\vskip 4 pt

For any Young diagram $\nu$ denote $\mathsf{E}(\nu)$ to be the semi-standard tableau of weight $\nu$, {\it i.e.} such that each $i$th row is filled with $i$. Then we arrive at the following equivalent definition of Littlewood-Richardson coefficients:
\begin{equation}\label{eq:definition_LR_2}
\def\arraystretch{1.4}
\begin{array}{ll}
    c^{\mu}_{\lambda\nu} & \text{is the number of semi-standard tableaux of the shape}\;\;\mu\backslash\lambda\;\;\text{and weight}\;\;\nu\\
    \hfill & \text{whose rectification is $\mathsf{E}(\nu)$.}
\end{array}
\end{equation}
By comparing the two examples \eqref{eq:example_Yamanouchi} and \eqref{eq:example_Rect}, one can verify that both definitions \eqref{eq:definition_LR_1} and \eqref{eq:definition_LR_2} lead to the same result $c^{(4,2,1)}_{(2,1)\, (3,1)} = 2$.
\vskip 4 pt

Let us consider a configuration obtained after a number of sliding processes during the jeu de taquin applied to a semi-standard tableau of a shape $\mu\backslash\lambda$, and let us keep the empty boxes at the end of each sliding process. Then one has a chain of three diagrams $\tau \subset \sigma \subset \mu$, where $\mu\backslash\sigma$ is the set of empty boxed resulting from the sliding processes, $\sigma\backslash\tau$ is a semi-standard tableau, and $\tau$ is the set of empty boxes not involved in the performed sliding processes. Let us say that a box of $\mu\backslash\sigma$ is an {\it addable corner} if adding it to $\sigma$ leads to a Young diagram. Note that each addable corner results from a sliding process applied to an inner corner, and that each particular sliding process is invertible. Let us define the {\it reverse sliding process} for any addable corner: at any step consider the neighbour(s) on the left and above the empty box, then slide the greater one into the empty box, while if the two are equal, slide the one above. The process continues until the empty box becomes an inner corner. In this respect, any chain $\tau \subset \sigma \subset \mu$, with a  semi-standard skew-shape diagram $\sigma\backslash\tau$, can be considered as an intermediate configuration of the jeu de taquin, with both types of slidings possible. Upon exhausting direct sliding processes, the unique terminal configuration (the rectification) with $\tau = \varnothing$ was considered above. On the other hand, starting from the terminal configuration of the shape $\mu\backslash\nu$, with the semi-standard tableau $\mathsf{E}(\nu)$, and going backwards by different sequences of reverse slidings until $\sigma = \mu$ leads to different semi-standard tableaux $\mathsf{t}^{(\mathrm{init})}$ of different shapes $\mu\backslash\lambda^{(\mathrm{init})}$. Define $\mu\slashdiv\nu$ to be the set of so obtained diagrams $\lambda^{(\mathrm{init})}$. By construction, the terminal configuration $\mathsf{E}(\nu) = \mathrm{Rect}\big(\mathsf{t}^{(\mathrm{init})}\big)$ is the same for all $\mathsf{t}^{(\mathrm{init})}$, so according to the definition \eqref{eq:definition_LR_2}, the set $\mu\slashdiv\nu$ contains such diagrams $\lambda$ that $c^{\mu}_{\lambda\nu} \neq 0$, and only them. For example, keeping empty boxes upon constructing $\mathrm{Rect}(\mathsf{t}_1)$ in \eqref{eq:example_Rect} and applying different sequences of reverse sliding processes gives three skew-shape diagrams:
\begin{equation}
    \Yboxdim{8pt}\young(\one \one \one \times,\two\times,\times)\;\;\xrightarrow{\text{rev. slides}}\;\; \Yboxdim{8pt}\young(\times\times\times\one ,\one \one,\two)\,,\;\; \Yboxdim{8pt}\young(\times\times\one\one ,\times \one,\two)\,,\;\; \Yboxdim{8pt}\young(\times\times\one\one ,\times \two,\one)\,,\;\; \Yboxdim{8pt}\young(\times\one\one\one ,\times \two,\times)\,,\quad \text{thus}\quad \Yboxdim{8pt}\young(~~~~,~~,~)\slashdiv\young(~~~,~) = \left\{\young(~,~,~)\,,\; \young(~~,~)\,,\;\young(~~~)\right\}\,.
\end{equation}

\paragraph{Simple $B_{n}(\varepsilon N)$- and $\mathbb{C}\Sn{n}$-modules in $V^{\otimes n}$.} The simple (left) $B_{n}(\delta)$-modules $M^{(\lambda)}_n$ are indexed by particular Young diagrams $\lambda$ with $|\lambda| = n - 2f$, with $f = 0,\dots, \lfloor \frac{n}{2} \rfloor$ \cite[Corollary 3.5]{Wenzl_structure-Br}. We are interested in singling out only those of them which occur in $V^{\otimes n}$ in the context of Schur-Weyl-type duality \eqref{eq:Schur-Weyl}, when $\delta = \varepsilon N$. To do so, for a given $l\in \mathbb{N}$ consider the set
\begin{equation}\label{eq:def_Lambda}
\def\arraystretch{1.4}
    \begin{array}{rcl}
        \Lambda_{l,N} = & \big\{ \, \lambda \vdash l \;\; : \; & \lambda_1^{\prime} + \lambda_2^{\prime} \leqslant N\;\; \text{(for the dual partition $\lambda^{\prime}$) when $G(N) = O(N)$, and} \\
        \hfill & \hfill &  \lambda_1 \leqslant N\slash 2\;\; \text{when $G(N) = Sp(N)$} \big\} \,.
    \end{array}
\end{equation}
Then a $B_{n}(\varepsilon N)$-module $M^{(\lambda)}_n$ appears in $V^{\otimes n}$ iff $\lambda \in \Lambda_{n-2f,N}$ for some $f = 0,\dots,\lfloor\tfrac{n}{2}\rfloor$ \cite{Wenzl_structure-Br,Nazarov,DipperDotyHu_SpBrauer}.
\vskip 4 pt

The restrictions on Young diagrams in $\Lambda_{l,N}$  can be interpreted as the absence of certain $B_{n}(\varepsilon N)$-modules in $V^{\otimes n}$, which occurs due to non-injectivity of $\mathfrak{r}$. It is known that $\mathfrak{r}$ is injective for $N\geqslant n$ when $G(N) = O(N)$ and for $N\geqslant 2n$ when $G(N) = Sp(N)$ \cite{Brown_semisimplicity}, and this is exactly where the constraints trivialise. In the sequel, unless otherwise is specified, simple $B_{n}(\varepsilon N)$-modules will be viewed as submodules in $V^{\otimes n}$ (with $n\geqslant 2$ and $N = \dim V$).
\vskip 4 pt

We will also need some basic facts about the representation theory of $\mathbb{C}\Sn{n}\subset B_{n}(\varepsilon N)$. Namely, that simple $\mathbb{C}\Sn{n}$-modules $L^{(\mu)}$ are indexed by Young diagrams with $|\mu| = n$. Those of them which occur in $V^{\otimes n}$ carry a label from the following set:
\begin{equation}\label{eq:def_Sigma}
    \def\arraystretch{1.4}
    \begin{array}{rcl}
        \Sigma_{n,N} = & \big\{ \, \mu \vdash n \;\; : \; & \mu_1^{\prime} \leqslant N\;\;\text{for $\varepsilon = 1$, and} \\
        \hfill & \hfill & \mu_1 \leqslant N\;\; \text{for $\varepsilon = -1$} \big\}\,,
    \end{array}
\end{equation}
where $\varepsilon$ distinguishes between the two ways for permutations to act on $V^{\otimes n}$ according to \eqref{eq:action_generators}.
\vskip 4 pt

The sets $\hat{\Lambda}_{n-2f,N}$ and $\hat{\Sigma}_{n,N}$ in \eqref{eq:tensor_GL-O_decomposition} and \eqref{eq:Schur-Weyl} are obtained from $\Lambda_{n-2f,N}$ and $\Sigma_{n,N}$ via element-wise application of the following $\varepsilon$-dependent mapping:
\begin{equation}\label{eq:hat}
\def\arraystretch{1.7}
\begin{array}{rl}
    \text{for any partition $\lambda$,} & \text{$\hat{\lambda} = \lambda$ when $\varepsilon = 1$ (in particular, when $G(N) = O(N)$), and} \\
    \hfill & \text{$\hat{\lambda} = \lambda^{\prime}$ when $\varepsilon = -1$ (in particular, when $G(N) = Sp(N)$).}
\end{array}
\end{equation} 
\vskip 4 pt

\paragraph{Restriction to $\mathbb{C}\Sn{n}$.} Upon restriction to the subgroup $G(N)\subset GL(N)$, the simple $GL(N)$-modules decompose into a direct sum of simple $G(N)$-modules. Similarly, the simple $B_{n}(\varepsilon N)$-modules decompose into a direct sum of simple $\mathbb{C}\Sn{n}$-modules upon restriction to the subalgebra $\mathbb{C}\Sn{n} \subset B_{n}(\varepsilon N)$. The two restrictions are related via comparing the two decompositions \eqref{eq:Schur-Weyl} of the same space $V^{\otimes n}$ (see, {\it e.g.}, \cite[Lemma 4.2]{Gavarini_LitRich}):
\begin{equation}\label{eq:branching}
\begin{array}{l}
    \displaystyle\def\arraystretch{0.7}
    \text{if}\quad V^{(\mu)} \cong \bigoplus_{\begin{array}{c}
        {\scriptstyle \hat{\nu}\;\;\text{even}\,,}\\
        {\scriptstyle \nu^{\prime}_1 \leqslant N} 
    \end{array}}\big(D^{(\lambda)}\big)^{\oplus \overline{c}^{\mu}_{\nu\lambda}}\quad\text{(upon $GL(N)\downarrow G(N)$)}\\
    \text{then}\displaystyle\def\arraystretch{0.7}\quad M^{(\hat{\lambda})}_n \cong \bigoplus_{\begin{array}{c}
        {\scriptstyle \hat{\nu}\;\;\text{even}\,,}\\
        {\scriptstyle \nu^{\prime}_1 \leqslant N} 
    \end{array}} \big(L^{(\hat{\mu})}\big)^{\oplus \overline{c}^{\mu}_{\nu\lambda}}\quad\text{(upon $B_{n}(\varepsilon N)\downarrow \mathbb{C}\Sn{n}$)}\,,\quad\text{and {\it vice versa}}.
\end{array}
\end{equation}
When $\mu\in\hat{\Lambda}_{n,N}\subset \hat{\Sigma}_{n,N}$ one has $\overline{c}^{\mu}_{\nu\lambda} = c^{\mu}_{\nu\lambda}$ (the Littlewood-Richardson coefficients): this is the case when the Littlewood restriction rules apply \cite{Littlewood}. The branching rules on the left-hand-side of \eqref{eq:branching} are extensively presented in the literature \cite{Littlewood,Koike_Terada_banching,Enright_Willenbring_branching,Kwon_branching}. In particular, note the combinatorial approach proposed in \cite[Theorem 4.17 and Remark 4.19]{Jang_Kwon_branching}, where in the case $\mu\notin\hat{\Lambda}_{n,N}$ the multiplicities $\overline{c}^{\mu}_{\nu\lambda}$ are defined via additional $N$-dependent constraints on the tableaux in the definition \eqref{eq:definition_LR_1}.
\vskip 4 pt

In this respect, the subset of labels $\Lambda_{n,N}\subset \Sigma_{n,N}$ will be called {\it Littlewood-admissible}. By applying Schur-Weyl duality \eqref{eq:Schur-Weyl} in this case, one can detect the occurrence $L^{(\mu)}\subset M^{(\lambda)}_n$ upon restriction to $\mathbb{C}\Sn{n}$. In more detail, for any integer $f\geqslant 0$ define the following set of Young diagrams ({\it the closure}):
\begin{equation}\label{eq:LR-closure}
\def\arraystretch{1.4}
        \mathrm{cl}^{(f)}_{N}(\lambda) =  \left\{\mu \supset \lambda \;\; :\;\; \mu\in \Sigma_{n,N}\;\;\text{and}\;\; |\mu\backslash\lambda| = 2f\,,\;\; c^{\mu}_{\nu\lambda}\neq 0\;\; \text{for some even}\;\; \nu\;\; \text{with}\;\; |\nu| = 2f\right\}\,.
\end{equation}
When $B_{n}(\varepsilon N)$ is {\it semisimple} ({\it i.e.} when $N\geqslant n-1$ for $\varepsilon = 1$ and $N\geqslant 2(n-1)$ for $\varepsilon = -1$ \cite{Brown_semisimplicity}), the constraints on the size of diagrams in \eqref{eq:LR-closure} trivialise when $\lambda \in \Lambda_{n-2f}$ with $f \geqslant 1$. Occurrence of a particular simple $\mathbb{C}\Sn{n}$-module in a given simple $B_{n}(\varepsilon N)$-module $M^{(\lambda)}_{n}$ can be analysed via the following lemma.

\begin{lemma}\label{lem:restriction_decomposition}
    Let $M^{(\lambda)}_n$ be a simple $B_{n}(\varepsilon N)$-module, with $|\lambda| = n - 2f$. Upon its restriction to $\mathbb{C}\Sn{n}$,
    \begin{equation*}
        \text{if}\;\;L^{(\mu)}\subset M^{(\lambda)}_n\;\;(\text{with}\;\; |\mu| = n) \;\;\text{then}\;\; \mu\in \mathrm{cl}^{(f)}_N (\lambda)\,.
    \end{equation*}
    The converse is also true whenever one of the possibilities hold:
    \begin{itemize}
        \item[i)] $\mu \in \Lambda_{n,N}$ ({\it i.e.} $\mu$ is Littlewood-admissible),
        \item[ii)] $B_{n}(\varepsilon N)$ is semisimple.
    \end{itemize}
\end{lemma}
\begin{proof}
    We need to recall some basic facts about the {\it standard $B_{n}(\delta)$-modules} $\Delta^{(\lambda)}_{n}$, where $\delta \in \mathbb{C}$ (see \cite[formula (2.5)]{CDVM_Br_blocks} and references therein, more detailed discussion is postponed to Section \ref{sec:splitting_idempotent}). Standard modules are labelled by all partitions $\lambda\vdash n-2f$ for $f = 0,\dots,\lfloor{\frac{n}{2}\rfloor}$. Each $\Delta^{(\lambda)}_n$ is an indecomposable module (but not necessarily simple when $\delta\in\mathbb{Z}$), such that $M^{(\lambda)}_n \cong \Delta^{(\lambda)}_n\slash K$ for the maximal proper submodule $K$.  In the semisimple regime of the Brauer algebra standard modules are simple, so $\Delta^{(\lambda)}_n \cong M^{(\lambda)}_n$.
    \vskip 4 pt
    
   The decomposition of standard modules into simple summands upon restriction to $\mathbb{C}\Sn{n}$ was described in \cite[Theorem 4.1]{HW_discriminants_Br_algebra}, which implies $L^{(\mu)} \subset \Delta^{(\lambda)}_n$ iff $c^{\mu}_{\nu\lambda}\neq 0$ for some even $\nu$. Moreover, for the Littlewood-Richardson coefficient to be non-zero, $\lambda \subset \mu$ and $|\nu| = |\mu| - |\lambda| = 2f$. The additional restrictions on the number of rows/columns in the definition of $\mathrm{cl}^{(f)}_{N}(\lambda)$ reflect the restrictions on the $\mathbb{C}\Sn{n}$-modules appearing in $V^{\otimes n}$ according to the classical Schur-Weyl duality \cite{Weyl}.
   \vskip 4 pt
   
   To finish the proof, for $\mu\in \Lambda_{n,N}$ the multiplicities of $L^{(\mu)}$ in the branching rules \eqref{eq:branching} are Littlewood-Richardson coefficients.
\end{proof}

\paragraph{Eigenvalues of $A_n$.} In what follows we make use of the results of \cite{Nazarov} (see also \cite{CDVM_Br_blocks} in order to cover all $\delta \in \mathbb{Z}$). Consider the following set of pairwise-commuting elements known as {\it Jucys-Murphy elements} $x_k$ ($k = 1,\dots, n$): 
\begin{equation}
    x_{1} = \frac{\varepsilon N-1}{2}\quad \text{and} \quad x_{k} = \frac{\varepsilon N-1}{2} + \sum_{j=1}^{k-1} (s_{jk} - d_{jk})\quad \text{for}\quad k\geqslant 2\,.
\end{equation}
Among the applications of the latter in the context of the present work, we mention that they can be used for constructing the maximal commutative subalgebra in $B_{n}(\varepsilon N)$, whose elements are diagonalisable on $B_{n}(\varepsilon N)$-modules in $V^{\otimes n}$. We concentrate on the following central element
\begin{equation}\label{eq:master_class_JM}
    X_B = \sum_{k = 1}^{n} x_k = n\,\frac{\varepsilon N-1}{2} + X_S - A_n\,\quad \text{where} \quad X_S = \sum_{i<j} s_{ij}\quad\text{and \eqref{eq:master_class} was used.}
\end{equation}
Its value on the simple $B_{n}(\varepsilon N)$-module $M^{(\lambda)}_n$ is $(n - 2f)\,\frac{\varepsilon N-1}{2} + c(\lambda)$. The term $X_S$ is a sum of Jucys-Murphy elements in the algebra $\mathbb{C}\Sn{n}$ \cite{Jucys,Murphy}: it is central in $\mathbb{C}\Sn{n}$ and proportional to identity on any simple $\mathbb{C}\Sn{n}$-module $L^{(\mu)}$ with the coefficient $c(\mu)$ (a possible way to check this is to apply \eqref{eq:master_class_JM} to $M^{(\mu)}_{n} \cong L^{(\mu)}$ taking into account that $d_{ij}\big(M^{(\mu)}_n\big) = 0$ for $\mu\vdash n$, see \cite{Nazarov}).
\vskip 4 pt

Define the following set of skew-shape diagrams:
\begin{equation}\label{eq:skew-shape}
    \overline{\Lambda}^{(f)}_{n,N} = \bigcup_{\lambda\in \Lambda_{n-2f,N}} \bigcup_{\mu \in \mathrm{cl}^{(f)}_N(\lambda)} \mu\backslash \lambda\,.
\end{equation}
Variations of the following lemma are known in the literature (see, {\it e.g.}, the proofs in \cite[Theorem 2.6]{Nazarov} or \cite[Proposition 4.2]{CDVM_Br_blocks} and references therein), nevertheless we give the proof to keep the narrative self-contained.
\begin{lemma}\label{lem:D_block_diagonal}
    Let $M^{(\lambda)}_n$ be a simple $B_{n}(\varepsilon N)$-module, and let a simple $\mathbb{C}\Sn{n}$-module $L^{(\mu)}$ occur in the decomposition of $M^{(\lambda)}_n$ into irreducible summands upon restriction to $\mathbb{C}\Sn{n}$. Then $\mu\backslash\lambda\in \overline{\Lambda}_{n,N}^{(f)}$ and
    \begin{equation}\label{eq:eigenvalue_A}
        \text{for any}\quad v\in L^{(\mu)}\,,\quad A_n(v) = \big((\varepsilon N-1)\,f + c(\mu\backslash\lambda)\big)\,v\,.
    \end{equation}
\end{lemma}
\begin{proof}
    The relation \eqref{eq:master_class_JM}, together with the fact that $X_B$ and $X_S$ are both proportional to identity on $L^{(\mu)}\subset M^{(\lambda)}_n$, implies that $A_n$ is proportional to identity on $L^{(\mu)}$ as well. According to the structure of the decomposition of $M^{(\lambda)}_n$ upon restriction to $\mathbb{C}\Sn{n}$ (see Lemma \ref{lem:restriction_decomposition}), one has $\lambda \subset \mu$, so $c(\mu) - c(\lambda) = c(\mu\backslash\lambda)$, and the eigenvalue in the assertion is a direct consequence of \eqref{eq:master_class_JM} for $X_B$ (respectively, $X_S$) restricted to $M^{(\lambda)}_n$ (respectively, to $L^{(\mu)}$).
\end{proof}

\noindent Together with the fact that $V^{\otimes n}$ decomposes as a direct sum of simple $B_{n}(\varepsilon N)$-modules, the above lemma provides an alternative proof (along with Lemma \ref{lem:KerA}) that $A_n$ is diagonalisable on $V^{\otimes n}$. In order to describe the eigenvalues, define the set
\begin{equation}\label{eq:spec_master_class}
        \mathrm{spec}^{\times}(A_n) = \big\{(\varepsilon N-1)\,f + c(\mu\backslash\lambda) \;\; : \;\; f = 1,\dots,\lfloor\tfrac{n}{2}\rfloor\,,\quad \mu\backslash\lambda\in \overline{\Lambda}^{(f)}_{n-2f,N} \big\}\cap \varepsilon \mathbb{N}\,.
\end{equation}
\begin{proposition}\label{prop:spectre_A}
    Any non-zero eigenvalue of $A_n$ on $V^{\otimes n}$ is contained in $\mathrm{spec}^{\times}(A_n)$. When $B_{n}(\varepsilon N)$ is semisimple, any element of $\mathrm{spec}^{\times}(A_n)$ is a non-zero eigenvalue of $A_n$ on $V^{\otimes n}$.
\end{proposition}
\begin{proof}
    Lemma \ref{lem:restriction_decomposition} gives the necessary condition for occurrence of a simple $\mathbb{C}\Sn{n}$-module $L^{(\mu)}$ in $M^{(\lambda)}_n$ upon restriction to $\mathbb{C}\Sn{n}$, while Lemma \ref{lem:D_block_diagonal} gives the corresponding eigenvalue $\alpha_{\mu\backslash\lambda}$ of $A_n$. According to Lemma \ref{lem:KerA} one has: {\it i)} zero eigenvalues correspond to the components $D^{(\hat{\mu})}\otimes M^{(\mu)}_n$ with $\mu \in \Lambda_{n,N}$ in \eqref{eq:Schur-Weyl}, and thus to $f = 0$ in \eqref{eq:eigenvalue_A}, {\it ii)} non-zero eigenvalues are in $\varepsilon \mathbb{N}$. Therefore, if $\alpha_{\mu\backslash\lambda}$ is a non-zero eigenvalue of $A_n$ on $V^{\otimes n}$, then necessarily it is in $\mathrm{spec}^{\times}(A_n)$. When $B_{n}(\varepsilon N)$ is semisimple, Lemma \ref{lem:restriction_decomposition} becomes a criterion, so any element in \eqref{eq:spec_master_class} comes from $L^{(\mu)} \subset M^{(\lambda)}_n$ for some $\lambda \in \Lambda_{n-2f,N}$ ($f \geqslant 1$) and $\mu \in \mathrm{cl}^{(f)}_N(\lambda)$.
    \vskip 4 pt
\end{proof}
Note that imposing intersection with $\varepsilon\mathbb{N}$ is substantial in the definition of $\mathrm{spec}^{\times}(A_n)$ in \eqref{eq:spec_master_class} in order to reduce the number of elements which are not eigenvalues of $A_n$. For example, consider $N = 4$ and $\varepsilon = 1$, and take $\lambda = (2^2) \in \Lambda_{4,4}$, $\mu = (2^4) \in \mathrm{cl}^{(2)}_4(\lambda)$. One has $(N - 1) f + c(\mu\backslash\lambda) = -2$ which is not an eigenvalue of $A_n$ due to Lemma \ref{lem:KerA} (in other words, $L^{(\mu)} \not\subset M^{(\lambda)}_n$). To this end, let us note that the relevant fact here is that $\mu$ is not Littlewood-admissible. A more detailed discussion is postponed to Section \ref{sec:traceless_GL} in relation to traceless projection of simple $GL(N)$-modules.

\paragraph{\underline{Remark} (on the structure of $B_{n}(\varepsilon N)$-modules via Schur-Weyl duality).} In the semisimple regime of $B_{n}(\varepsilon N)$ one can prove Lemma \ref{lem:KerA} within the representation theory of the Brauer algebra (see Appendix \ref{app:KerA_proof_semisimple}). While outside the semisimple regime ($N < n-1$ for $\varepsilon = 1$ and $N < 2(n-1)$ for $\varepsilon = -1$, see \cite{Rui_Br_semisimple}) the following corollary of Lemmas \ref{lem:KerA} and \ref{lem:D_block_diagonal} seems to be hard to prove without addressing to tensorial representations via Schur-Weyl duality (see Appendix \ref{app:proof_KerA_Br} for proof and recall the notion of standard $B_{n}(\delta)$-modules in the proof of Lemma \ref{lem:restriction_decomposition}).
\begin{corollary}\label{cor:KerA_Br}
    Fix $N \geqslant 1$, $\varepsilon = \pm 1$ and $n \geqslant 2$ for $B_{n}(\varepsilon N)$. Take any $\lambda\in \Lambda_{n,N}$ and $\mu \in \mathrm{cl}^{(f)}_N(\lambda)$ for $2f = n - |\lambda|$ such that $L^{(\mu)}\subset M^{(\lambda)}_n$ upon restriction to $\mathbb{C}\Sn{n}$ and denote $\alpha_{\mu\backslash\lambda}$ the corresponding eigenvalue in Lemma \ref{lem:D_block_diagonal}. Then the following assertions hold:
    \begin{itemize}
        \item[{\it i)}] $\alpha_{\mu\backslash\lambda} \geqslant 0$,
        \item[{\it ii)}] $\alpha_{\mu\backslash\lambda} = 0$ iff $\lambda = \mu$.
    \end{itemize}
    As a result, for the standard module $\Delta^{(\lambda)}_n$ the following condition is sufficient to conclude that $L^{(\mu)}$ does not occur in the simple head \footnote{To recall the notion of a simple head for the situation in question, note that a standard $B_{n}(\delta)$-module is indecomposable, so the action of $B_{n}(\delta)$ admits a block upper-triangular form. The simple head is then a simple module obtained by projection to the lowest block on the diagonal.} of $\Delta^{(\lambda)}_n$:
    \begin{equation}
        \text{either}\quad \alpha_{\mu\backslash\lambda} < 0\quad\text{or}\quad\alpha_{\mu\backslash\lambda} = 0\;\;\text{for}\;\;|\lambda| < n\,.
    \end{equation}
\end{corollary}
\noindent 
The above Corollary does not hold outside the domain of applicability of Schur-Weyl duality. For example, consider $B_{6}(\delta)$ with $\delta = \tfrac{1}{2}\notin \mathbb{Z}$. For non-integer values of the parameter the algebra is semisimple \cite{Wenzl_structure-Br,Rui_Br_semisimple}, so $L^{(3,2,1)}\subset M^{(2)}_6$ (upon restriction to $\mathbb{C}\Sn{6}$) by Lemma \ref{lem:restriction_decomposition}, while $A_6\big(L^{(3,2,1)}\big) = 0$ by Lemma \ref{lem:D_block_diagonal}.

\section{Traceless projectors}\label{sec:traceless_projector}
\subsection{Traceless projection of $V^{\otimes n}$}
Consider the following element in the centralizer of $\mathbb{C}\Sn{n}$ in $B_{n}(\varepsilon N)$:
\begin{equation}\label{eq:projector}
    P_n = \prod_{\alpha\in \mathrm{spec}^{\times}(A_n)} \left(1 - \frac{1}{\alpha}\,A_n\right)\in C_{n}(\varepsilon N)\,,\quad \text{and define}\quad \mathfrak{P}_n = \mathfrak{r}(P_{n})\in \mathfrak{B}_{n}(N)\,.
\end{equation}
Due to \eqref{eq:flip} and \eqref{eq:master_class}, the property $(\mathfrak{P}_n)^{*} = \mathfrak{P}_n$ is manifest. The main result is summarised in the following theorem.
\begin{theorem}\label{thm:projector}
    $\mathfrak{P}_n$ is the universal traceless projector on $V^{\otimes n}$, which satisfies the properties \eqref{eq:projector_properties}.
\end{theorem}
\begin{proof}
    To prove that $\mathfrak{P}_n$ is a projector, consider the decomposition of $V^{\otimes n}$ into a direct sum of irreducible $B_{n}(\varepsilon N)$-modules $M^{(\lambda)}_n$, and then decompose each of them into a direct sum of $\mathbb{C}\Sn{n}$-modules $L^{(\mu)}\subset M^{(\lambda)}_n$. $A_n$ is block-diagonal with respect to this decomposition with possible non-zero eigenvalues described by Proposition \ref{prop:spectre_A}. The zero eigenvalue exactly marks the traceless subspace of $V^{\otimes n}$, which is due to Lemma \ref{lem:KerA}. So by design of \eqref{eq:projector}, it annihilates all $\mathbb{C}\Sn{n}$-modules marked by non-zero eigenvalues of $A_n$ (there is no problem if some elements of $\mathrm{spec}^{\times}(A_n)$ do not show up in $V^{\otimes n}$). As a result,
    \begin{equation}
        A_nP_n(V^{\otimes n}) = 0\,,\quad \text{therefore}\quad P_n{}^2(V^{\otimes n}) = P_n(V^{\otimes n})\,,
    \end{equation}
    so the properties $(\mathrm{P}1)$ and $(\mathrm{P}2)$ of \eqref{eq:projector_properties} hold. The property $(\mathrm{P}3)$ of \eqref{eq:projector_properties} is manifest due to the fact that $P_n\in C_{n}(\varepsilon N)$.
\end{proof}
\noindent Let us illustrate the application of Theorem \ref{thm:projector} in some simple cases. In Examples 1 and 2 we consider $G(N) = O(N)$, so $\varepsilon = 1$.

\paragraph{Example 1 (one-dimensional space).} Let us verify that for $N = \dim V = 1$ the projector trivialises: $\mathfrak{P}_n = 0$  for all $n\geqslant 2$ (or, equivalently, that $P_n(V^{\otimes n}) = 0$). Note that in this case for any $T\in V^{\otimes n}$, $d_{ij}(T) = T$ (for all $1\leqslant i < j \leqslant n$), and hence 
\begin{equation}\label{eq:AT_dim_1}
    A_n(T) = \frac{n(n-1)}{2}\,T\,.
\end{equation}
Thus, in order to prove that $P_n(T) = 0$ it suffices to show that $\frac{n(n-1)}{2} \in \mathrm{spec}^{\times}(A_n)$. But this is directly what expression \eqref{eq:AT_dim_1} says. Let us cross-check this straightforward conclusion by computing the eigenvalues of $A_n$ from the representation-theoretic point of view. Constraints on the lengths of columns leaves the only possibility $\lambda = \varnothing$ for $n$ even and $\lambda = \Yboxdim{9pt}\yng(1)$ for $n$ odd, and $\nu = (2f_{\mathrm{max}})$ with $f_{\mathrm{max}} = \lfloor\frac{n}{2}\rfloor$. On the right-hand side of $\lambda\LRp\nu$ there is the only partition $\mu = (n)$ which belongs to $\mathrm{cl}_1^{(f_{\mathrm{max}})}(\lambda)$, so

\begin{displaymath}
\ytableausetup{mathmode,boxsize=1.1em,centertableaux}
\mu\backslash\lambda=\begin{ytableau}{\scriptstyle}
      \zero & \one & \none[\scriptstyle{\cdots}]& \nmione
\end{ytableau}
\quad\text{(for $n$ even)}\quad \text{and}\quad \mu\backslash\lambda =\ytableausetup{mathmode,boxsize=1.1em,centertableaux}
\begin{ytableau}{\scriptstyle}
      \times  & \one & \none[\scriptstyle{\cdots}]& \nmione
\end{ytableau}
\quad\text{(for $n$ odd)}\quad \Rightarrow \quad c(\mu\backslash\lambda) = \tfrac{n(n-1)}{2}\,.
\end{displaymath}

\paragraph{Example 2 (lower-rank projectors).} In the sequel we assume $N\geqslant 2$. Let us derive the basic well-known example: the traceless projector for $n=2$. In this case $\mathrm{spec}^{\times}(A_2)$ is constituted by the only value $f = 1$ in \eqref{eq:spec_master_class} for $\lambda = \varnothing$ and $\nu = \Yboxdim{9pt}\young(~~)$:
\begin{equation*}
    \mu = \varnothing \LRp \Yboxdim{9pt}\young(~~) = \young(~~)\quad\Rightarrow\quad \mu\backslash\lambda = \young(\zero\one)\,,\quad \alpha = (N-1) + c(\mu\backslash\lambda) = N\,,
\end{equation*}
so one immediately arrives at the expected result
\begin{equation}
    P_2 = 1-\frac{1}{N}\,A_2 = \,
    \raisebox{-.4\height}{\includegraphics[width=15pt,height=20pt]{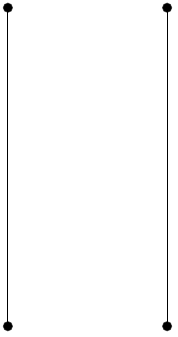}}
    \, -\frac{1}{N} \ 
    \raisebox{-.4\height}{\includegraphics[width=15pt,height=20pt]{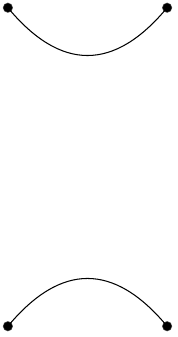}}
\end{equation}

The case $n=3$ is slightly more cumbersome to construct from scratch (by solving the traceless Ansatz for a tensor), but is still elementary from the point of view of $B_{3}(N)$. As in the previous example, $\mathrm{spec}^{\times}(A_3)$ is obtained for $f = 1$ in \eqref{eq:spec_master_class} (again, $\nu = \Yboxdim{9pt}\young(~~)$), with
\begin{equation}
\def\arraystretch{1.4}
\lambda = \Yboxdim{9pt}\young(~)\,,\quad \mu = \left\{\begin{array}{l}
        \Yboxdim{9pt}\young(~~~)\\
        \Yboxdim{9pt}\young(~~,~)
    \end{array}\right.
    \quad \Rightarrow \quad \mu\backslash\lambda = \left\{
    \begin{array}{ll}
        \Yboxdim{9pt}\young(\times \one \two)\,, & \alpha = (N-1) + 3 \\
        \Yboxdim{9pt}\young(\times \one,\mione)\,, & \alpha = (N-1) + 0
    \end{array}
    \right.
\end{equation}
\noindent The operator \eqref{eq:projector} takes the form

\begin{equation}\label{eq:projector_Br_3}
\begin{aligned}
    P_3 &=\left(1-\frac{1}{N-1}\,A_3\right)\left(1-\frac{1}{N+2}\,A_3\right) \\
      &=1-\frac{2N+1}{(N-1)(N+2)}\,A_3+\frac{1}{(N-1)(N+2)}\,(A_3)^2 \\
      &=
      \Scale[0.8]{\raisebox{-.4\height}{\includegraphics[width=15pt,height=20pt]{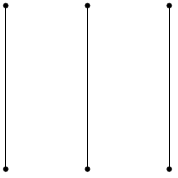}}}
      -\frac{N+1}{(N-1)(N+2)}\,\Bigl(\  
      \Scale[0.8]{\raisebox{-.4\height}{\includegraphics[width=25pt,height=20pt]{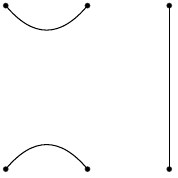}}+\raisebox{-.4\height}{\includegraphics[width=25pt,height=20pt]{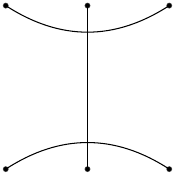}}+\raisebox{-.4\height}{\includegraphics[width=25pt,height=20pt]{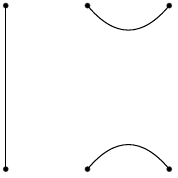}}\  }
      \Bigr)\,+  \frac{1}{(N-1)(N+2)}\, \ \Bigl(
      \Scale[0.8]{ \raisebox{-.4\height}{\includegraphics[width=25pt,height=20pt]{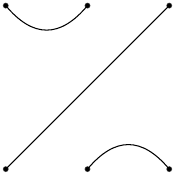}}+\raisebox{-.4\height}{\includegraphics[width=25pt,height=20pt]{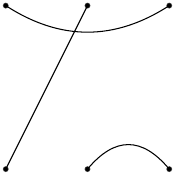}}+\raisebox{-.4\height}{\includegraphics[width=25pt,height=20pt]{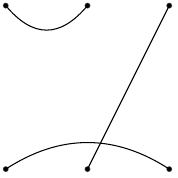}}+\raisebox{-.4\height}{\includegraphics[width=25pt,height=20pt]{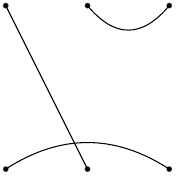}}+\raisebox{-.4\height}{\includegraphics[width=25pt,height=20pt]{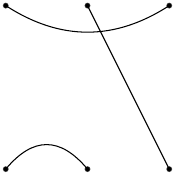}}+\raisebox{-.4\height}{\includegraphics[width=25pt,height=20pt]{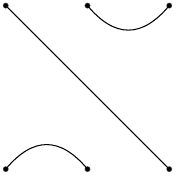}}}
      \Bigr)\,
\end{aligned}
\end{equation}
\noindent Note that the denominators in the above expression are singular at $N=1$, which is not a problem according to Theorem \ref{thm:projector}: one simply omits the factor with the eigenvalue which turns to $0$ upon putting $N=1$, which brings us back to Example 1.
\vskip 6 pt

In the following two examples we consider $G(N) = Sp(N)$, so $\varepsilon = -1$.

\paragraph{Example 3 (lower-rank projectors).} Let us reproduce the projector in the obvious case $n = 2$. There is the only element in $\mathrm{spec}^{\times}(A_n)$ obtained for $f = 1$ from $\lambda = \varnothing$ and $\nu = \Yboxdim{9pt}\young(~~)$ :
\begin{equation*}
    \mu = \varnothing \LRp\Yboxdim{9pt}\young(~~) = \young(~~)\quad\Rightarrow\quad \mu\backslash\lambda = \young(\zero\one)\,,\;\; -(N+1)\,f + c(\mu\backslash\lambda) = -N\,,
\end{equation*}
so one immediately arrives at the expected form of the projector:
\begin{equation}\label{eq:projector_rank-2_Sp}
    P_2 = 1+\frac{1}{N}\,A_2 = \,
    \raisebox{-.4\height}{\includegraphics[width=15pt,height=20pt]{Id2.pdf}}
    \,+\frac{1}{N} \ 
    \raisebox{-.4\height}{\includegraphics[width=15pt,height=20pt]{d1.pdf}}
\end{equation}

For $n=3$ suppose $N > 2$ (the case $N=2$ will be considered below), then there are two diagrams contributing to $\mathrm{spec}^{\times}(A_n)$ arising from $\lambda = \Yboxdim{9pt}\young(~)$ and $\nu = \Yboxdim{9pt}\young(~~)$ ($f = 1$):
\begin{equation}
\def\arraystretch{1.4}
\lambda\LRp\nu = \Yboxdim{9pt}\young(~~,~) + \young(~~~)\,,\;\;\text{so}\quad \mu = \left\{\begin{array}{l}
        \Yboxdim{9pt}\young(~~~)\\
        \Yboxdim{9pt}\young(~~,~)
    \end{array}\right.\quad \Rightarrow \quad \mu\backslash\lambda = \left\{
    \begin{array}{ll}
        \Yboxdim{9pt}\young(\times \one \two)\,, & \alpha = -(N+1) + 3 \\
        \Yboxdim{9pt}\young(\times \one,\mione)\,, & \alpha = -(N+1) + 0
    \end{array}
    \right.
\end{equation}
\noindent The operator \eqref{eq:projector} takes the form
\begin{equation}\label{eq:projector_rank-3_Sp}
\begin{aligned}
    P_3 &=\left(1+\frac{1}{N-2}\,A_3\right)\left(1+\frac{1}{N+1}\,A_3\right) \\
      &=1+\frac{2N-1}{(N-2)(N+1)}\,A_3+\frac{1}{(N-2)(N+1)}\,(A_3)^2 \\
      &=
      \Scale[0.8]{\raisebox{-.4\height}{\includegraphics[width=15pt,height=20pt]{Id3.pdf}}}
     +\frac{N-1}{(N-2)(N+1)}\,\Bigl(\  
      \Scale[0.8]{\raisebox{-.4\height}{\includegraphics[width=25pt,height=20pt]{b3_1.pdf}}+\raisebox{-.4\height}{\includegraphics[width=25pt,height=20pt]{b3_2.pdf}}+\raisebox{-.4\height}{\includegraphics[width=25pt,height=20pt]{b3_3.pdf}}\  }
      \Bigr)\,+  \frac{1}{(N-2)(N+1)}\, \ \Bigl(
      \Scale[0.8]{ \raisebox{-.4\height}{\includegraphics[width=25pt,height=20pt]{b3_4.pdf}}+\raisebox{-.4\height}{\includegraphics[width=25pt,height=20pt]{b3_5.pdf}}+\raisebox{-.4\height}{\includegraphics[width=25pt,height=20pt]{b3_6.pdf}}+\raisebox{-.4\height}{\includegraphics[width=25pt,height=20pt]{b3_7.pdf}}+\raisebox{-.4\height}{\includegraphics[width=25pt,height=20pt]{b3_8.pdf}}+\raisebox{-.4\height}{\includegraphics[width=25pt,height=20pt]{b3_9.pdf}}}
      \Bigr)\,
\end{aligned}
\end{equation}

\paragraph{Example 4 (two-dimensional space, arbitrary rank).}Consider a two-dimensional space $N = 2$. In this case each index set $\Lambda_{n-2f,2}$ is constituted by the single partition $(1^{n-2f})$, so $ \overline{\Lambda}^{(f)}_{n,2}$ is constituted by $(n-f,f)^{\prime}\backslash (1^{n-2f})$. The corresponding eigenvalue is
\begin{equation}
    -(N+1)\,f + c(\mu\backslash\lambda) = -(n-f+1)f\,,\quad\text{so}\quad \mathrm{spec}^{\times}(A_n) = \{-(n-f+1)f\;:\;f = 1,\dots,\lfloor\tfrac{n}{2}\rfloor\}\,.
\end{equation}
Note that for $n=2$ one has $\mathrm{spec}^{\times}(A_2) = \{-2\}$, which leads immediately to \eqref{eq:projector_rank-2_Sp} above. Note that $\mathrm{spec}^{\times}(A_n) \subset -\mathbb{N}$, which is in agreement with Proposition \ref{prop:spectre_A}. Let us consider $n=3$ in detail, such that $\mathfrak{r}$ is not injective. Due to the restrictions on the number of columns, $\mathrm{spec}^{\times}(A_3) = \{-3\}$, and the only factor constituting the projector \eqref{eq:projector_rank-3_Sp} survives: 
\begin{equation}\label{eq:projector_Sp_non-stable}
    P_3 = \left(1+\frac{1}{3}\,A_n\right) =
      \Scale[0.8]{\raisebox{-.4\height}{\includegraphics[width=15pt,height=20pt]{Id3.pdf}}}
     +\frac{1}{3}\,\Bigl(\  
      \Scale[0.8]{\raisebox{-.4\height}{\includegraphics[width=25pt,height=20pt]{b3_1.pdf}}+\raisebox{-.4\height}{\includegraphics[width=25pt,height=20pt]{b3_2.pdf}}+\raisebox{-.4\height}{\includegraphics[width=25pt,height=20pt]{b3_3.pdf}}\  }
      \Bigr)
\end{equation}
By a direct calculation one observes that vanishing of traces of the $\mathfrak{r}(P_3)$-image of a tensor is not due to straightforward cancellation of all terms, but rather due to the fact that anti-symmetrization of tree two-dimensional vectors vanishes identically, so $d_{ij} P_3\in \mathrm{Ker}(\mathfrak{r})$ for all $1\leqslant i<j\leqslant 3$.

\paragraph{\underline{Remark}.} At this stage it becomes clear that the main computational difficulty resides in expanding the factorised formula \eqref{eq:projector} diagram-wise. A technique which allows one to partially circumvent the latter difficulty is presented in Section \ref{sec:bracelets}.

\subsection{Traceless projection of simple $GL(N)$-modules.}\label{sec:traceless_GL}

The projector \eqref{eq:projector} is constructed in a way to take into account {\it all} simple $G(N)$-modules occurring in $V^{\otimes n}$. However in applications to tensor calculus in physics one often starts with a certain $GL(N)$-module $V^{(\mu)} \subset V^{\otimes n}$ and performs its traceless projection. To mention a number of examples, in application to higher-spin theory and strings, symmetric tensor fields serve as a primer \cite{Fronsdal_massless,Bonelli_strings_HS} (see \cite{Isaev_HS} for application of the Brauer algebra in the context of Fronsdal fields). One can also find extensive studies of fields of arbitrary symmetry types (the mixed-symmetry fields), see \cite{Labastida_mixed,Siegel_Zwiebach,BMV_conjecture,Boul-Bek_mixed_GL} and references therein.
\vskip 4 pt

\paragraph{Restriction of $\mathfrak{P}_n$.} In this section we will adapt the construction of the universal traceless projector \eqref{eq:projector} for obtaining traceless projection of each simple $GL(N)$-module. In this respect one ignores the metric $\langle\,\cdot,\cdot\,\rangle$ for a moment and takes the first decomposition in \eqref{eq:Schur-Weyl} as a starting point. A particular choice of $G(N)$ is nevertheless anticipated by fixing one of the two ways for permutations to act on $V^{\otimes n}$ regarding the choice of $\varepsilon = \pm 1$ (recall the definition of $\mathfrak{r}(s_i)$ in \eqref{eq:action_generators}). A simple $GL(N)$-module $V^{(\mu)}$ (with some $\mu\in\hat{\Sigma}_{n,N}$) can be realised by applying a primitive idempotent $I^{(\hat{\mu})}\in\mathbb{C}\Sn{n}$ to $V^{\otimes n}$, such that
\begin{equation}
    \mathfrak{r}\big(I^{(\hat{\mu})}\big) V^{\otimes n} \cong V^{(\mu)}\,,\quad\text{and}\quad\mathbb{C}\Sn{n} I^{(\hat{\mu})} \cong L^{(\hat{\mu})}\quad\text{(with respect to left multiplication in $\mathbb{C}\Sn{n}$)}\,.
\end{equation}
As a particular well-known example of a primitive idempotent, one can take any Young symmetriser $Y_{\mathsf{t}(\hat{\mu})}\in\mathbb{C}\Sn{n}$ ($\mathsf{t}(\hat{\mu})$ being a standard Young tableau of shape $\hat{\mu}$). Let us also mention that aside from the set of Young symmetrisers, which are primitive idempotents but not orthogonal ones, the complete set of primitive orthogonal idempotents is available for $\mathbb{C}\Sn{n}$ (see \cite{Molev} and references therein, see also \cite{Garsia2020}).
\vskip 4 pt

Returning the metric into consideration, we will aim at explicit realisation of the following idea. Traceless projection of $V^{(\mu)}$ means restriction to $G(N)$ and projection to the submodule $D^{(\mu)}\subset V^{(\mu)}$ (see, {\it e.g.}, \cite[Section 3.2]{Nazarov_tracelessness}). With the second decomposition in \eqref{eq:Schur-Weyl} at hand and applying Lemma \ref{lem:KerA}, traceless projection of $V^{(\mu)}$ implies annihilating all simple $B_{n}(\varepsilon N)$-modules which are labelled by $\Lambda_{n-2f,N}$ with $f\geqslant 1$ and which contain a module $L^{(\hat{\mu})}$.
\vskip 4 pt

Define the set of diagrams which parametrises all simple $B_{n}(\varepsilon N)$-modules where a given $\mathbb{C}\Sn{n}$-module $L^{(\rho)}$ can occur by Lemma \ref{lem:restriction_decomposition}: for any $\rho \in \Lambda_{n,N}$ set
\begin{equation}
\def\arraystretch{0.7}
        \Lambda^{(\rho)}_N = \bigcup_{\begin{array}{c}
            {\scriptstyle \text{even}\;\;\nu\subset \rho\,,}\\
            {\scriptstyle |\nu| = 2f\,,\;\;f \geqslant 0}
        \end{array}} (\rho \slashdiv \nu)\cap \Lambda_{n-2f,N} \,,
\end{equation}
where one makes use of the reverse of the jeu de taquin described in Section \ref{sec:Brauer_diagrams}. Note that $\rho \in \Lambda^{(\rho)}_N$, and for any other $\sigma\in \Lambda^{(\rho)}_N$ holds $|\sigma| < |\rho|$. Joining the above sets over all possible labels of $\mathbb{C}\Sn{n}$-modules in $V^{\otimes n}$ one reconstructs the labels of all $B_{n}(\varepsilon N)$-modules in $V^{\otimes n}$, so 
\begin{equation}
    \Lambda_{n,N} = \bigcup_{\rho\in\Sigma_{n,N}} \Lambda^{(\rho)}_N\quad \text{(in general, $\Lambda^{(\rho_1)}_N\cap \Lambda^{(\rho_2)}_N \neq \varnothing$).}
\end{equation}
Continuing along these lines, we consider the following subsets $\mathrm{spec}^{\times}_{\rho}(A_n) \subset \mathrm{spec}^{\times}(A_n)$
\begin{equation}\label{eq:spec_reduced}
    \mathrm{spec}^{\times}_{\rho}(A_n) = \big\{(\varepsilon N - 1)f + c(\rho\backslash\sigma)\;:\;f = 1,\dots,\lfloor{\tfrac{n}{2}\rfloor}\,, \;\;\sigma \in \Lambda^{(\rho)}_N\,,\;\;|\rho| - |\sigma| = 2f \big\}\cap \varepsilon \mathbb{N}\,.
\end{equation}
The following proposition is a restricted version of Proposition \ref{prop:spectre_A} and is proven along the same lines.
\begin{proposition}\label{prop:spectre_A_red}
    Let $\mu \in \Sigma_{n,N}$, so $V^{(\hat{\mu})}\otimes L^{(\mu)}\subset V^{\otimes n}$. Any non-zero eigenvalue of $A_n$ on $V^{(\hat{\mu})}\otimes L^{(\mu)}$ is contained in $\mathrm{spec}^{\times}_{\mu}(A_n)$. Conversely, any element of $\mathrm{spec}^{\times}_{\mu}(A_n)$ is an eigenvalue of $A_n$ on $V^{(\hat{\mu})}\otimes L^{(\mu)}$ if one of the following conditions hold:
    \begin{itemize}
        \item[i)] $\mu \in \Lambda_{n,N}$ ({\it i.e.} $\mu$ is Littlewood-admissible),
        \item[ii)] $B_{n}(\varepsilon N)$ is semisimple.
    \end{itemize}
\end{proposition}
\vskip 4 pt

Consider the reduced operator
\begin{equation}\label{eq:projector_reduced}
    P_n^{(\mu)} = \prod_{\alpha\in \mathrm{spec}^{\times}_{\mu}(A_n)} \left(1 - \frac{1}{\alpha}\,A_n\right)\,, \quad \text{and define}\quad \mathfrak{P}^{(\hat{\mu})}_n = \mathfrak{r}(P_n^{(\mu)})\,.
\end{equation}
The above formula can not be applied directly in the only case $\mu = (1^n)$ where $\mathrm{spec}^{\times}_{\mu}(A_n) = \varnothing$. This corresponds to totally anti-symmetric tensors $V^{(1^n)}\subset V^{\otimes n}$ when $G(N) = O(N)$ and totally symmetric tensors $V^{(n)}\subset V^{\otimes n}$ when $G(N) = Sp(N)$, which are automatically traceless. In this particular case we set by definition $P_{n}^{(\mu)} = 1$. All in all, the property $\big(\mathfrak{P}_{n}^{(\hat{\mu})}\big)^{*} = \mathfrak{P}_{n}^{(\hat{\mu})}$ (for all $\mu \in \Sigma_{n,N}$) is manifest. The whole paragraph is summarised by the following result, which reflects the commonly utilised fact that the two operations -- specific symmetrization of indices of a tensor and subtracting traces -- can be performed separately and in any order relatively to one another.
\begin{theorem}\label{thm:reduced_projector}
    Let $V^{(\mu)}\subset V^{\otimes n}$ be a simple $GL(N)$-module (with $\mu\in \hat{\Sigma}_{n,N}$). Upon restriction to $V^{(\mu)}$ one has $\mathfrak{P}_n^{(\mu)}\big|_{V^{(\mu)}} = \mathfrak{P}_n\big|_{V^{(\mu)}}$. In particular, for a primitive idempotent $I\in\mathbb{C}\Sn{n}$ such that $\mathfrak{r}(I)V^{\otimes n} \cong V^{(\mu)}$, the operator
    \begin{equation}
        \mathfrak{I}_n^{\prime(\mu)} = \mathfrak{P}_n \mathfrak{r}(I) = \mathfrak{P}_n^{(\mu)} \mathfrak{r}(I)\quad \text{(equivalently,}\;\; \mathfrak{I}_n^{\prime(\mu)} =  \mathfrak{r}(I) \mathfrak{P}_n = \mathfrak{r}(I) \mathfrak{P}_n^{(\mu)}\text{)}
    \end{equation}
    projects $V^{\otimes n}$ onto a simple module isomorphic to $D^{(\mu)}$ when $\mu \in \hat{\Lambda}_{n,N}$ ({\it i.e.} $\hat{\mu}$ is Littlewood-admissible), or annihilates it otherwise. In addition, if $I^{*} = I$, then $\big(\mathfrak{I}_n^{\prime(\mu)}\big)^{*} = \mathfrak{I}_n^{\prime(\mu)}$.
\end{theorem}

\paragraph{\underline{Remark.}}Note that Young symmetrisers are not self-adjoint with respect to $(\,\cdot\,)^{*}$. To have the latter property at hand one should take the orthogonal primitive idempotents in $\mathbb{C}\Sn{n}$ \cite{Jucys_idempotents,Murphy,Molev} (called also Young seminormal units \cite{Garsia2020} or Hermitian Young operators \cite{Keppeler:2013yla}).
\vskip 4 pt

The construction in question admits the following straightforward generalisation to the case of a direct sum of simple $GL(N)$-modules (possibly with certain multiplicities) in $V^{\otimes n}$. For $\mathcal{I}\subset \Sigma_{n,N}$ define $\mathrm{spec}^{\times}_{\mathcal{I}}(A_n) = \bigcup_{\mu\in\mathcal{I}} \mathrm{spec}^{\times}_{\mu}(A_n)$, construct
\begin{equation}
    \quad P^{(\mathcal{I})}_n = \prod_{\alpha\in \mathrm{spec}^{\times}_{\mathcal{I}}(A_n)} \left(1 - \frac{1}{\alpha}\,A_n\right)\quad \text{and set}\quad \mathfrak{P}^{(\hat{\mathcal{I}})}_n = \mathfrak{r}\big(P^{(\mathcal{I})}_n\big)
\end{equation}
(where $\hat{\mathcal{I}}$ denotes the mapping $\mu \mapsto \hat{\mu}$ applied element-wise). The following corollary is a simple consequence of Theorem \ref{thm:reduced_projector}
\begin{corollary}\label{cor:projector_reduced_sum}
    Consider $\mathcal{I} \subset \hat{\Sigma}_{n,N}$ and a $GL(N)$-module
    \begin{equation}
        V^{(\mathcal{I})} = \bigoplus_{\mu\in\mathcal{I}} \big(V^{(\mu)}\big)^{\oplus c_{\mu}} \subset V^{\otimes n}\quad (c_{\mu} \geqslant 1)\,.
    \end{equation}
    Then upon restriction to $V^{(\mathcal{I})}$ one has $\mathfrak{P}_n^{(\mathcal{I})}\big|_{V^{(\mathcal{I})}} = \mathfrak{P}_n\big|_{V^{(\mathcal{I})}}$.
\end{corollary}

In particular, as soon as a tensor product of $GL(N)$-modules $V^{(\mu_1)}, V^{(\mu_2)}$ decomposes into a direct sum with the aid of the Littlewood-Richardson rule applied to $\mu_1\LRp \mu_2$, namely
\begin{equation}\label{eq:projector_GL_tens_prod}
\def\arraystretch{0.6}
    V^{(\mu_1)}\otimes V^{(\mu_2)} \cong \bigoplus_{\begin{array}{c}
        {\scriptstyle \sigma\,\vdash\, |\mu_1| + |\mu_2|}\\
        {\scriptstyle c^{\sigma}_{\mu_1 \mu_2}\neq 0\,,\;\; \sigma^{\prime}_{1} \leqslant N}
    \end{array}}
    \big(V^{(\sigma)}\big)^{\oplus c^{\sigma}_{\mu_1 \mu_2}} \;\;\subset\;\; V^{\otimes (|\mu_1| + |\mu_2|)}\,,
\end{equation}
the reduced traceless operator is constructed as $\mathfrak{P}_{|\mu_1| + |\mu_2|}^{(\mathcal{I})}$, with the index set $\mathcal{I}$ constituted by labels $\sigma$ of the modules occurring on the right-hand-side of \eqref{eq:projector_GL_tens_prod} (see Example 8 in Appendix \ref{app:examples}).

\paragraph{Example 5 (totally symmetric $O(N)$-tensors).}
For the fixed partition $\mu =(n)$ one constructs $\mathrm{spec}^{\times}_{(n)}(A_n)$ for the skew-shape diagrams $\mu\backslash\lambda$ for all $\lambda \in (n)\slashdiv (2f)$, $f = 1,\dots,\lfloor{\frac{n}{2}\rfloor}$. This leads to
\begin{displaymath}
\ytableausetup{mathmode,boxsize=1.3em,centertableaux}
\mu\backslash\lambda=
\begin{ytableau}{\scriptstyle}
      \times & \none[\scriptstyle{\cdots}]& \times & {\scriptstyle f} &\none[\scriptstyle{\cdots}] & {\scriptstyle n-1}
\end{ytableau}
\quad \Rightarrow \quad \mathrm{spec}_{(n)}^{\times}(A_n) = \left\{\big(N + 2\ (n - f - 1)\big)\,f\;:\; f = 1,\dots, \lfloor \tfrac{n}{2}\rfloor\right\}\,.
\end{displaymath}
Hence the reduced projector \eqref{eq:projector_reduced} takes the form
\begin{equation}\label{eq:projector_symmetric0}
    P_n^{(n)} = \displaystyle{\prod_{f=1}^{\lfloor \tfrac{n}{2}\rfloor}} \left(1 - \frac{A_n}{\big(N+2\ (n-f-1)\big)f}\right)\,.
\end{equation}
In the next section we will rewrite the above expression in terms of the Lie algebra $\mathfrak{sl}(2)$ which arises in the context of Howe duality for symmetric tensors. This will allow us to rewrite the expression in the expanded form.

\paragraph{Example 6 (maximally-antisymmetric hook $O(N)$-tensors).}
For the partition $\mu=(2,1^{n-2})$ one constructs $\mathrm{spec}^{\times}_{(2,1^{n-2})}(A_n)$ for the skew shape diagrams $\mu\backslash\lambda$ with $\lambda \in (2,1^{n-2})\slashdiv (2)$, which leads to the only possibility: \\
\begin{equation}
    \ytableausetup{mathmode,boxsize=1.3em,centertableaux}
\mu\backslash\lambda=
\begin{ytableau}{\scriptstyle}
      \times & {\scriptstyle 1} \\
      \times\\
      \none[\svdots]\\
      \times\\
      {\scriptstyle 2-n}\\
\end{ytableau} 
\quad \Rightarrow \quad \mathrm{spec}_{(2,1^{n-2})}^{\times}(A_n) = \big\{N - n + 2 \, \big\}\,,\;\; \text{so} \quad P_n^{(2,1^{n-2})}=1- \dfrac{A_n}{N-n+2}\;.
\end{equation}
Note that in order for the $GL(N)$-module in question to be present in $V^{\otimes n}$, one assumes $N \geqslant n-1$, so the denominator is non-singular. Moreover, in accordance with Proposition \ref{prop:spectre_A} the expression in the denominator is always positive. The case of a generic hook $(m,1^{k})$ is considered in Appendix \ref{app:examples}.
\vskip 4 pt

\paragraph{Quasi-additive form of the universal projector $\mathfrak{P}_n$.}From the point of view of applications, the most convenient form of the universal traceless projector would be a sum of averages ({\it i.e.} elements of $C_{n}(\varepsilon N)$, recall the comment at the end of Section \ref{sec:conjugacy_classes}). To this end, along with theoretical transparency of the factorised form \eqref{eq:projector} and its convenience for restrictions to $GL(N)$-modules, the other side of the coin is that regarding applications the expression \eqref{eq:projector} is quite far from being optimal -- first of all due to the necessity to express the powers $(A_n)^p$ as combinations of averages. The latter problem is partially resolved with the aid of reduced traceless projectors described in Theorem \ref{thm:reduced_projector}. Our goal consists in summing up the latter and reconstruct the universal traceless projector as a polynomial in $A_n$ of a smaller degree than that of \eqref{eq:projector}.
\vskip 4 pt

First, we note that the traceless subspace of $V^{(\mu)}\subset V^{\otimes n}$ is non-zero only if $\mu \in \hat{\Lambda}_{n,N}$ ($\hat{\mu}$ is Littlewood-admissible). By restricting our attention to the latter set of $GL(N)$-modules, we construct the following element in $B_{n}(\varepsilon N)$. Let $z^{(\mu)}$ denote the central Young symmetriser associated to a simple $\mathbb{C}\Sn{n}$-module indexed by $\mu\vdash n$ \cite{centralYoung}, then set
\begin{equation}\label{eq:projector_semi_sum}
    \tilde{P}_n = \sum_{\mu \in \Lambda_{n,N}}
    P^{(\mu)}_n z^{(\mu)}\,.
\end{equation}
Note that due to Proposition \ref{prop:spectre_A_red}, each $P^{(\mu)}_n$ in the above formula is constructed with the minimal possible number of factors since each element in $\mathrm{spec}^{\times}_{\mu}(A_n)$ is an eigenvalue of $A_n$. Also, one has the property $(\tilde{P}_n)^{*} = \tilde{P}_n$. Indeed, $(z^{(\mu)})^{*} = z^{(\mu)}$ (see \eqref{eq:ns-flip} and the comment below), and each $P^{(\mu)}_n\in C_{n}(\varepsilon N)$. The $\mathfrak{r}$-image of \eqref{eq:projector_semi_sum} gives the sought resummation of the reduced projectors of Theorem \ref{thm:reduced_projector}.
\begin{corollary}\label{cor:projector_semi-sum}
    The universal traceless projector admits the following form:
    \begin{equation}\label{eq:projector_quasi-additive}
        \mathfrak{P}_n = \mathfrak{r}\big(\tilde{P}_n\big)\,.
    \end{equation}
\end{corollary}
\begin{proof}
    Central Young symmetrisers form a decomposition of unity:
    \begin{equation}\label{eq:projector_semi-sum_1}
        1 = \sum_{\mu\vdash n} z^{(\mu)}\quad \Rightarrow\quad P_n = \sum_{\mu\vdash n} P_n z^{(\mu)}\,.
    \end{equation}
    Each $z^{(\mu)}$ is the sum of orthogonal idempotents whose left ideal in $\mathbb{C}\Sn{n}$ is isomorphic to $L^{(\mu)}$. So, by Theorem \ref{thm:reduced_projector}, $\mathfrak{P}_n \mathfrak{r}\big(z^{(\mu)}\big) = \mathfrak{P}^{(\hat{\mu})}_n \mathfrak{r}\big(z^{(\mu)}\big)$. For $V^{(\hat{\mu})}$ such that $\mu\notin \Lambda_{n,N}$, $D^{(\hat{\mu})}$ is not present in the decomposition of the latter upon restriction to $G(N)$, and hence the corresponding traceless projection vanishes identically, so one has $\mathfrak{P}^{(\hat{\mu})}_n \mathfrak{r}\big(z^{(\mu)}\big) = 0$ in this case. Therefore, from \eqref{eq:projector_semi-sum_1} on obtains
    \begin{equation}
        \mathfrak{P}_n = \mathfrak{r}(P_n) = \sum_{\mu\in\Lambda_{n,N}} \mathfrak{P}^{(\hat{\mu})}_n \,\mathfrak{r}\big(z^{(\mu)}\big) = \mathfrak{r}\big(\tilde{P}_n\big)\,.
    \end{equation}
\end{proof}
\paragraph{\underline{Remark.}} The examples of traceless projectors given in the companion Mathematica notebook utilise the expression \eqref{eq:projector_semi_sum}, which is the one implemented in \cite{BrauerAlgebraPackage}.

\subsection{Tracelessness in the context of Howe duality}\label{sec:Howe_duality}

In the context of the group $G(N)$ acting on tensors, aside from the Schur-Weyl-type duality which concerns representation theory of Brauer algebras, there is another well-known duality which is important through its application in field theories -- namely, the Howe duality \cite{Howe} (which is often referred to as  ``oscillator realisation'' in the physics literature, see, {\it e.g.}, \cite{Metsaev_mixed_sym_AdS,Boul-Iaz-Sund_mixed2,Alk_Grig_Tip} for applications). The latter relates representations of the classical group $G$ to those of an algebra $\mathfrak{A}$ via a bimodule where the actions of the two mutually centralise each other. In particular, a decomposition of a $(G,\mathfrak{A})$-bimodule reminiscent to \eqref{eq:Schur-Weyl} takes place, with finite-dimensional simple $G$-modules in the left slot and a simple $\mathfrak{A}$-modules in the right slot. The main difference with the Schur-Weyl-type dualities consists in considering the {\it infinite-dimensional subspace in the tensor algebra} $T(V)$ where the actions of $G$ and $\mathfrak{A}$ meet (instead of the finite-dimensional component $V^{\otimes n}$). The algebra $\mathfrak{A}$ is generated by transformations which form a Lie algebra $\mathfrak{d}(r)$ for some $r\geqslant 1$: $\mathfrak{d}(r) = \mathfrak{sp}(2r)$ for $G(N) = O(N)$ and $\mathfrak{d}(r) = \mathfrak{o}(2r)$ for $G(N) = Sp(N)$.
\vskip 4 pt

\paragraph{Totally symmetric traceless $G(N)$-tensors.} We start by revisiting Example 5 (with $G(N) = O(N)$) which is a good starting point to introduce the main ideas. The space of totally symmetric tensors (of arbitrary rank) is isomorphic to the space $\mathbb{C}[y]$ of polynomials in $N$ variables $y = \{y_{a}\;:\; a = 1,\dots, N\}$. Rank-$n$ symmetric tensors are isomorphic to the subspace of degree-$n$ homogeneous polynomials which we denote $\mathbb{C}[y]_n$:
\begin{equation*}
    T = e_{a_1}\otimes\dots\otimes e_{a_n} t^{a_1\dots a_n} \in \mathrm{sym} (V^{\otimes n})\quad \mapsto \quad T(y) = t^{a_1\dots a_n}\, y_{a_1}\dots y_{a_n}\in \mathbb{C}[y]_n
\end{equation*}
(abusing notation, we denote the tensor and the corresponding polynomial by the same letter). The trace is the same for any pair of indices, so one has $\mathrm{tr}^{(g)}T(y) = g_{bc}t^{bc\,a_3\dots a_n}\, y_{a_3}\dots y_{a_n}$, which is conveniently expressed via the following second-order differential operator:
\begin{equation*}
    \text{for}\quad \mathsf{e}_{+} = \frac{1}{2}\,g_{ab}\frac{\partial}{\partial y_a}\frac{\partial}{\partial y_b}\quad \text{one has}\quad\mathsf{e}_{+}\,T(y) = \tfrac{n(n-1)}{2}\,\mathrm{tr}^{(g)}T(y)\,.
\end{equation*}
If one additionally considers the quadratic operator $\mathsf{e}_{-} = -\tfrac{1}{2}\,g^{ab}\, y_a y_b$, then for the action of $\mathsf{e}_{-}\mathsf{e}_{+}$ on the polynomials one recognizes the action of $A_n$ on tensors:
\begin{equation}\label{eq:A_to_ee}
    A_n T \quad\mapsto \quad -2 \,\mathsf{e}_{-}\mathsf{e}_{+}\, T(y)\,.
\end{equation}
The commutator of the two operators $\mathsf{e}_{-},\mathsf{e}_{+}$ gives $\mathsf{h} = -\big(\tfrac{N}{2} + y_a\frac{\partial}{\partial y_a}\big)$, and all together they form the Lie algebra $\mathfrak{sl}(2)$:
\begin{equation}\label{eq:commutators_sl}
    \left[\mathsf{e}_{+},\mathsf{e}_{-}\right] = \mathsf{h}\,,\quad \left[\mathsf{h},\mathsf{e}_{\pm}\right] = \pm 2\,\mathsf{e}_{\pm}\,.
\end{equation}

As far as symmetric tensors form the irreducible $GL(N)$-module $\mathbb{C}[y]_n \cong V^{(n)}$, we make use of the reduced projector $P_{n}^{(n)}$ \eqref{eq:projector_symmetric0}. First, note that $\mathsf{h}\,T(y) = -(\tfrac{N}{2} + n)\, T(y)$ on tensors from $\mathbb{C}[y]_n$. Next, substitution \eqref{eq:A_to_ee} leads to the following operator acting on polynomials:
\begin{equation}\label{eq:projector_symmetric}
    P_n^{(n)} \quad\mapsto\quad P_{\mathfrak{sl}(2)} = \prod_{f\geqslant 1} \left(1 - \frac{\mathsf{e}_{-}\mathsf{e}_{+}}{(\mathsf{h} + f + 1)f}\right)\, ,
\end{equation}
which coincides with the form of {\it extremal projector} for $\mathfrak{sl}(2)$ presented in \textsection 7 of Chapter 3 in \cite{Zhelobenko_rep_reductive_Lie_alg} (see also \cite{Tolstoy_2004,Neretin_2009} for a review). The infinite product acts on $\mathbb{C}[y]$ by consecutive application of factors and truncates at $f_{\mathrm{max}} = \lfloor\tfrac{n}{2}\rfloor$ when restricted to $\mathbb{C}[y]_n$. As a result, traceless rank-$n$ tensors can be viewed as the subspace of highest-weight vectors of weight $-\big(\frac{N}{2} + n\big)$ in the $\mathfrak{sl}(2)$-module $\mathbb{C}[y]$.
\vskip 4 pt

The extremal projector $P_{\mathfrak{sl}(2)}$ admits the following (equivalent) additive form
\begin{equation}\label{eq:projector_symmetric_sl}
    P_{\mathfrak{sl}(2)} = P_{+}(c) = \sum_{f \geqslant 0} \frac{(-1)^f}{f!}\dfrac{1}{\displaystyle{\prod_{j=1}^{f}(\mathsf{h}+c+j)}}\,\mathsf{e}_{-}^f\mathsf{e}_{+}^f\,,\quad\text{with}\;\; c = 1\,.
\end{equation}
The products $\mathsf{e}_{-}^f\mathsf{e}_{+}^f$ can be directly mapped to elements $A_n^{(f)}$ ($f = 1,\dots,\lfloor{\frac{n}{2}\rfloor}$) which generalise $A_n = A_n^{(1)}$ to the case of $f$ arcs:
\begin{equation}\label{eq:class_A_f}
    A_{n}^{(f)} = \frac{1}{2^{f} f!(n-2f)!}\,\gamma_{d_1 d_3 \dots d_{2f-1}} \in C_{n}(\delta)\cap J^{(f)}\quad \Rightarrow\quad A_n^{(f)} \mapsto -\frac{1}{2^{f-2} f!}\,\mathsf{e}_{-}^f\mathsf{e}_{+}^f\,.
\end{equation}
Translated to $B_{n}(N)$, the expression \eqref{eq:projector_symmetric_sl} gives the expanded form of the traceless projector $P_n^{(n)}$ bypassing direct computations of the powers $(A_n)^{p}$ and restricting them to $V^{(n)}$. Rederivation of \eqref{eq:projector_symmetric_sl} in other frameworks can be found in the literature \cite{Rejon_Barrera_2016,CHPT_pseudo-Howe}.
\vskip 4 pt

Note that the same construction applies in the case when the metric is skew-symmetric, with $G(N) = Sp(N)$. Symmetric tensors are automatically traceless in this case which is reflected in trivialisation of the trace operator $\mathsf{e}_{+} = 0$, as well as $\mathsf{e}_{-} = 0$. The only non-trivial operator $\mathsf{h}$ constitutes the Lie algebra $\mathfrak{o}(2)$.

\paragraph{Mixed-symmetry $G(N)$-tensors.} In relation to the above example, let us mention the well-known way of realizing tensorial mixed-symmetry $G(N)$-modules $D^{(\rho)} \subset V^{\otimes n}$ via homogeneous polynomials and differential operators acting on them (as before, we start with the case $G(N) = O(N)$). Consider the space of polynomials $\mathbb{C}[\boldsymbol{y}]$ in the variables $\boldsymbol{y} = \{y_a^{i}\;:\;a = 1,\dots, N\,,\;i = 1,\dots, r\}$. For the rank-$n$ tensors of the symmetry type $\rho = (\rho_1,\dots,\rho_r)$ (with $|\rho| = n$) one considers the subspace $\mathbb{C}[\boldsymbol{y}]_{\rho}$ of degree-$n$ homogeneous polynomials which are also degree-$\rho_i$ homogeneous in each subset of variables $\{y_a^{i}\;:\; a = 1,\dots,N\}$ (for each fixed $i$). In other words, the elements of $\mathbb{C}[\boldsymbol{y}]_{\rho}$ are rank-$n$ tensors with $r$ enumerated groups of symmetrised indices, each $i$th group carrying $\rho_i$ indices.
\vskip 4 pt

From the polynomial variables and their derivatives one constructs the following set of differential operators which constitute the Lie algebra $\mathfrak{sp}(2r)$:
\begin{equation}\label{eq:Howe_pair}
    \mathsf{K}^{ij} = -\frac{1}{1+\delta_{ij}}\, g^{ab}\,y_{a}^{i} y_{b}^{j}\,,\;\; \mathsf{P}_{ij} = \frac{1}{1+\delta_{ij}}\, g_{ab}\,\frac{\partial}{\partial y_{a}^{i}} \frac{\partial}{\partial y_{b}^{j}}\,,\;\; \mathsf{H}^{i}{}_{j} = -\big(\dfrac{N}{2}\delta^{i}_{j}+y_{a}^{i} \frac{\partial}{\partial y_{a}^{j}}\big)\,.
\end{equation}
The operators $\mathsf{H}^{i}{}_{j}$ form the subalgebra $\mathfrak{gl}(r) \cong \mathfrak{e}\oplus \mathfrak{sl}(r)\subset \mathfrak{sp}(2r)$, where the center $\mathfrak{e}$ is spanned by the multiples of the Euler operator $\sum_{i=1}^{r}\mathsf{H}^{i}{}_{i}$. When $r=1$ one recognises the above example of symmetric tensors for $\mathfrak{sl}(2) \cong \mathfrak{sp}(2)$. The algebra \eqref{eq:Howe_pair} centralises the action of $O(N)$ on $\mathbb{C}[\boldsymbol{y}]$ and is well-known in the context of Howe duality \cite{Howe}. The polynomials $T(\boldsymbol{y})\in \mathbb{C}[\boldsymbol{y}]_{\rho}$ which constitute the irreducible $O(N)$-module $D^{(\rho)}$ satisfy the highest-weight conditions for the dual algebra $\mathfrak{sp}(2r)$:
\begin{align}
    \mathsf{H}^{i}{}_{i}\, T(\boldsymbol{y}) & = -\big(\frac{N}{2} + \rho_i\big)\,T(\boldsymbol{y})\;\;\text{(no summation over $i$)}\,, \label{eq:HoweDuality_Cartan}\\
    \mathsf{H}^{i}{}_{j}\,T(\boldsymbol{y}) & = 0\;\;\text{(for all $i<j$)}\,, \label{eq:HoweDuality_gl}\\
    \mathsf{P}_{ij}\,T(\boldsymbol{y}) & = 0\;\;\text{(for all $i\leqslant j$)}\,. \label{eq:HoweDuality_trace}
\end{align}
The algebra of the above operators is a Borel subalgebra $\mathfrak{h}\oplus \mathfrak{g}_{+}\subset \mathfrak{sp}(2r)$: homogeneity degrees $\rho_i$ enter the eigenvalues of the Cartan operators $\mathsf{H}^{i}{}_{i}\in \mathfrak{h}$, while the rest of the constraints represent annihilation of the highest-weight vectors by the positive root operators from $\mathfrak{g}_{+}$. In particular, $\mathfrak{g}_{+} = \check{\mathfrak{g}}_{+}\oplus \mathfrak{t}$: the subalgebra $\check{\mathfrak{g}}_{+}$ is constituted by the operators entering \eqref{eq:HoweDuality_gl} (which are positive root operators in $\mathfrak{sl}(r)\subset \mathfrak{gl}(r)$), while $\mathfrak{t}$ is the abelian ideal in $\mathfrak{g}_{+}$ constituted  by the trace operators in \eqref{eq:HoweDuality_trace}. From the point of view of tensor components, the constraints \eqref{eq:HoweDuality_trace} simply mean that the tensor is traceless with respect to any pair of indices. The constraints \eqref{eq:HoweDuality_gl} imply that symmetrisation of indices in the $i$th group with any index in the $j$th group is zero whenever $i<j$. The latter if often referred to as {\it Young property} as soon as it manifests itself for the images of the Young projector $Y_{\mathsf{t}_0(\rho)}V^{\otimes n}$, where each row in the standard Young tableau $\mathsf{t}_0(\rho)$ forms a sequence of consecutive integers.
\vskip 4 pt

The space of highest-weight vectors in any $\mathfrak{sp}(2r)$-module (and thus the solution of the constraints \eqref{eq:HoweDuality_Cartan}-\eqref{eq:HoweDuality_trace}) can be obtained via application of the corresponding {\it extremal projector}, which exists and is unique for any simple Lie algebra \cite{Zhelobenko_rep_reductive_Lie_alg} (see references therein and \cite{Tolstoy_2004} for the historical review). The projector \eqref{eq:projector_symmetric} is the simplest one of a kind. According to the general scheme, extremal projector for $\mathfrak{sp}(2r)$ is written in a form of an ordered product
\begin{equation}\label{eq:projector_extremal_sp}
    P_{\mathfrak{sp}(2r)} = \prod_{\beta \in \Delta(\mathfrak{g}_{+})}^{\longrightarrow}P_{\beta}(c_{\beta})\,.
\end{equation}
Each factor $P_{\beta}(c_{\beta})$ is given by the series \eqref{eq:projector_symmetric_sl} constructed from the operators $\{\mathsf{e}_{\pm \beta},\mathsf{h}_{\beta}\}$ forming a $\mathfrak{sl}(2)$-triple for each positive root $\beta\in \Delta(\mathfrak{g}_{+})$ (recall \eqref{eq:commutators_sl}), with
\begin{equation}\label{eq:extremal_parameters}
    c_{\beta} = \frac{1}{2} \sum_{\varphi\in\Delta(\mathfrak{g}_{+})} \varphi(\mathsf{h}_{\beta})\,.
\end{equation}
To be more specific, $\mathsf{e}_{\beta}$ (respectively, $\mathsf{e}_{-\beta}$) is either $\mathsf{H}^{i}{}_{j}$ (respectively, $\mathsf{H}^{j}{}_{i}$) with $i<j$ or $\mathsf{P}_{ij}$ (respectively, $\mathsf{K}^{ij}$) with $i\leqslant j$, and $\mathsf{h}_{\beta} = [\mathsf{e}_{\beta},\mathsf{e}_{-\beta}]$. The order of factors comes from the normal ordering of the positive roots (see, {\it e.g.}, \textsection 4 of Chapter 1 in \cite{Zhelobenko_rep_reductive_Lie_alg}), while the whole extremal projector is insensitive to a particular choice of normal ordering.
\vskip 4 pt

The case $G(N) = Sp(N)$ is considered along the same lines, with the Lie algebra $\mathfrak{o}(2r)$ realised by the same operators \eqref{eq:Howe_pair} except $\mathsf{P}_{ii} = 0$ and $\mathsf{K}^{ii} = 0$ due to the skew symmetry of the metric. With this remark at hand, root decomposition stays the same as for the case $\mathfrak{d}(r) = \mathfrak{sp}(2r)$, so simple $Sp(N)$-modules are singled out by the constraints \eqref{eq:HoweDuality_Cartan}-\eqref{eq:HoweDuality_trace}. The expression for the extremal projector \eqref{eq:projector_extremal_sp} applies also for $\mathfrak{d}(r) = \mathfrak{o}(2r)$. As a result, we can formulate the following lemma.
\begin{lemma}\label{lem:projector_Howe_factorised}
    The extremal projector $P_{\mathfrak{d}(r)}$ is divisible on the right by the extremal projector $P_{\mathfrak{sl}(r)}$:
    \begin{equation}\label{eq:projector_extremal_sp_gl}
        P_{\mathfrak{d}(r)} = P_{\mathfrak{t}}P_{\mathfrak{sl}(r)}\,,
    \end{equation}
    which reflects the possibility of consecutive implication of the constraints \eqref{eq:HoweDuality_gl} and \eqref{eq:HoweDuality_trace}.
\end{lemma}
\begin{proof}
    The fact that $\check{\mathfrak{g}}_{+}\subset \mathfrak{g}_{+}$ is the subalgebra and $\mathfrak{t}\subset \mathfrak{g}_{+}$ is the ideal allows one to have all the factors with the trace operators on the left. By direct computation\footnote{Recall that for any $\mathsf{h}\in \mathfrak{h}$ the value of any root $\beta(\mathsf{h})$ is obtained via the commutator $[h,\mathsf{e}_{\beta}] = \beta(\mathsf{h})\,\mathsf{e}_{\beta}$.} one finds that for any $\beta\in \Delta(\check{\mathfrak{g}}_{+})$ there is $\sum_{\varphi\in\Delta(\mathfrak{t})}\varphi(\mathsf{h}_{\beta}) = 0$. As a result, for the roots in question, $c_{\beta}$ in \eqref{eq:extremal_parameters} are replaced by $c^{\prime}_{\beta} = \frac{1}{2} \sum_{\varphi\in\Delta(\check{\mathfrak{g}}_{+})}\varphi(\mathsf{h}_{\beta})$, so the corresponding factors contain only the $\mathfrak{sl}(r)$ data. Combining this fact with the comment below \eqref{eq:projector_extremal_sp} about the ordering of factors proves the assertion.
\end{proof}
The operator $P_{\mathfrak{sl}(r)}$, applied to the space $\mathbb{C}[\boldsymbol{y}]$, resolves the constraints \eqref{eq:HoweDuality_gl}, while the trace constraints \eqref{eq:HoweDuality_trace} are resolved by $P_{\mathfrak{t}}$. The factorised form of the projector \eqref{eq:projector_extremal_sp_gl} is reminiscent to the form of the projectors presented in \cite{Nazarov_tracelessness}, where projection to a simple $G(N)$-module $D^{(\rho)}\subset V^{\otimes n}$ is performed in two steps: {\it i)} projection onto a simple $GL(N)$-module $V^{(\rho)}$, and {\it ii)} subtraction of traces. Note that $P_{\mathfrak{t}}$ itself is not a traceless projector on $\mathbb{C}[\boldsymbol{y}]$. 
An interesting related problem would be to look for an analog of the universal traceless projector presented in Theorem \ref{thm:projector}: a projector constructed from the operators $\mathsf{K}^{ij}$, $\mathsf{P}_{ij}$ and $\mathsf{H}^{i}{}_{i}$ which maps the whole space $\mathbb{C}[\boldsymbol{y}]$ (and any irreducible $GL(N)$-module $V^{(\rho)}\subset \mathbb{C}[\boldsymbol{y}]$ in particular) onto its traceless subspace and commutes with any projector to a simple $GL(N)$-module (so, with the extremal projector $P_{\mathfrak{sl}(r)}$ in particular).  
\vskip 4 pt

As a concluding remark, for a fixed $\rho\vdash n$ one can associate particular elements of $B_{n}(\varepsilon N)$ to the operators\footnote{The action of $(\mathsf{K}^{ij})^f(\mathsf{P}_{ij})^f$, when translated to $B_{n}(\varepsilon N)$, does not correspond to an element from $C_{n}(\varepsilon N)$. So in the framework of Howe-duality, with certain tensor components being {\it a priori} symmetrised in $\mathbb{C}[\boldsymbol{y}]_{\rho}$ (with $\rho\vdash n$), controlling commutation with the whole group of permutations and, as a consequence, with projectors on simple $GL(N)$-modules, is not manifest.} $(\mathsf{K}^{ij})^f(\mathsf{P}_{ij})^f$ and rewrite the extremal projector $P_{\mathfrak{d}(r)}$ as an element in $B_{n}(\varepsilon N)$. Then the $\mathfrak{r}$-image of the latter will reproduce the projector $\mathfrak{I}^{\prime(\rho)}$ of Theorem \ref{thm:reduced_projector}, where for the idempotent in $\mathbb{C}\Sn{n}$ one takes the Young projector $Y_{\mathsf{t}_0(\hat{\rho})}$ (recall Young property, see the paragraph below \eqref{eq:HoweDuality_gl}). While the two ways of constructing the same projector are two sides of the same coin, working with the extremal projector $P_{\mathfrak{d}(r)}$ appears to be hard at the level of computations. The advantage of our approach is due to a convenient ``condensed'' way of expanding the factorised form of the traceless projector \eqref{eq:projector} in terms of the elements in $C_{n}(\varepsilon N)$, avoiding diagram-wise computations. This technique is presented in the forthcoming section.

\section{$A_n$ as a second-order differential operator on $\mathbb{C}[\mathfrak{b}(\mathcal{A})]$}\label{sec:bracelets}

The factorised formula for the traceless projector \eqref{eq:projector} is extremely useful for presentation and elucidating its properties. Nevertheless, expanding and calculating the powers of $A_n$ \eqref{eq:master_class} becomes a hard computational problem already for relatively small ranks ({\it e.g.}, $n = 5,6,\dots$) when performed at the level of diagrams constituting the conjugacy classes. We propose a technique which allows one to circumvent this problem by performing the calculations at the level of conjugacy classes, without decomposing them into single diagrams. The results of this section hold for any value $\delta\in \mathbb{C}$ of the parameter of $B_{n}(\delta)$. 
\vskip 4 pt

\subsection{Parametrisation of bases in $C_{n}(\delta)$}\label{sec:classes-bracelets_map}

To avoid diagram-wise computations we make use of the fact that $A_n$ is an element of the algebra $C_{n}(\delta)$, which implies that any power $(A_n)^p$ can be decomposed over a basis in $C_{n}(\delta)$. As it was mentioned in Section \ref{sec:conjugacy_classes}, any maximal independent set among the averages $\gamma_{b}$ (for all $b\in B_{n}(\delta)$) forms the basis in $C_{n}(\delta)$. In order to parametrise it we consider an equivalent reformulation of the one described in \cite{Shalile_Br-center} (see also \cite{KMP_central_idempotents}). Namely, the bases in $C_{n}(\delta)$ are in one-to-one correspondence with a particular subset of so-called ternary bracelets. A {\it ternary bracelet} is an equivalence class of non-empty words over the ternary alphabet $\mathcal{A} = \{\bb{n},\bb{s},\bb{p}\}$ related by cyclic permutations and inversions, {\it i.e.} can be viewed as a word with its letters written along a closed loop without specifying the direction of reading. In the sequel, we will write $[w]$ to denote a bracelet containing a representative $w$ (a word over $\mathcal{A}$). For the reverse of $w$ we will write $I(w)$, so according to the definition of bracelets one has $[w] = [I(w)]$. The length of a bracelet is defined as the length of any among its representatives, which is written as $|w|$.
\vskip 4 pt

Denote $\mathfrak{b}(\mathcal{A})$ the set of non-empty ternary bracelets with the same number of occurrences of the letters $\bb{n}$ and $\bb{s}$ (which is allowed to be $0$), with the additional requirement that for any representative, if there is a pair of letters $\bb{n}$ (respectively, $\bb{s}$), there is necessarily the letter $\bb{s}$ (respectively, $\bb{n}$) in between. For example, $[\bb{s}],[\bb{n}\bb{n}\bb{s}\bb{s}]\notin \mathfrak{b}(\mathcal{A})$ and 
\begin{equation}
    [\bb{n}\bb{s}\bb{p}\bb{p}] = [\bb{s}\bb{p}\bb{p}\bb{n}] = [\bb{n}\bb{p}\bb{p}\bb{s}] = [\bb{p}\bb{p}\bb{n}\bb{s}] = [\bb{s}\bb{n}\bb{p}\bb{p}] = [\bb{p}\bb{n}\bb{s}\bb{p}] = [\bb{p}\bb{s}\bb{n}\bb{p}] \in \mathfrak{b}(\mathcal{A})\,,\quad |\bb{n}\bb{s}\bb{p}\bb{p}| = 4\,.
\end{equation}
\vskip 4 pt
\noindent
Consider the polynomial algebra $\mathbb{C}[\mathfrak{b}(\mathcal{A})]$, {\it i.e.} the $\mathbb{C}$-span over 
\begin{equation}
    \text{the basis monomials}\quad [w_1]\dots [w_r]\quad (r\geqslant 1\,, \;\;\text{with all}\quad [w_j]\in \mathfrak{b}(\mathcal{A}))\quad\text{and}\quad 1\,.
\end{equation}
We will write $[w]^m = \underbrace{[w]\dots [w]}_{m}$ for brevity, as well as $\bb{a}^{m} = \underbrace{\bb{a}\dots\bb{a}}_{m}$ for a letter $\bb{a}\in\mathcal{A}$. The degree of a monomial is defined as the sum of the lengths of the bracelets, $\deg\left([w_{1}]\dots [w_{r}]\right) = |w_{1}| + \dots + |w_{r}|$ (where by definition one puts $\deg(1) = 0$), and we denote by $\mathbb{C}[\mathfrak{b}(\mathcal{A})]_{n} \subset \mathbb{C}[\mathfrak{b}(\mathcal{A})]$ the subset of polynomials of a given degree $n$.
\vskip 4 pt

Consider the linear map $\Phi: B_{n}(\delta) \to \mathbb{C}[\mathfrak{b}(\mathcal{A})]_n$ defined on any single diagram $b\in B_{n}(\delta)$ as follows:
\begin{itemize}
    \item[{\it 1)}] label each line of the diagram $b$ by a letter from $\mathcal{A}$: the arcs in the upper (respectively, lower) row by $\bb{n}$ (respectively, $\bb{s}$), the vertical lines by $\bb{p}$;
    \item[{\it 2)}] identify the upper nodes with the corresponding lower nodes and straighten the obtained loops, which results in a set of bracelets $[w_{1}],\dots,[w_{r}]\in \mathfrak{b}(\mathcal{A})$;
    \item[{\it 3)}] define $\Phi(b) = \frac{1}{n!}\,[w_1] \dots [w_r]$.
\end{itemize}
\noindent For example, for the diagram $b_1$ of Section \ref{sec:Brauer_diagrams} one has the above sequence of transformations 

\begin{figure}[H]
\begin{equation*}
\raisebox{-.45\height}{\includegraphics[scale=0.7]{b1.pdf}}\mapsto \raisebox{-.45\height}{\includegraphics[scale=0.8]{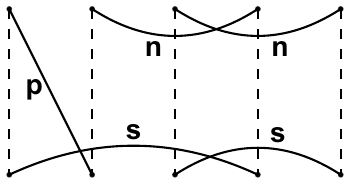}}\mapsto \Bigg\lbrace{\raisebox{-.35\height}{\includegraphics[scale=0.18]{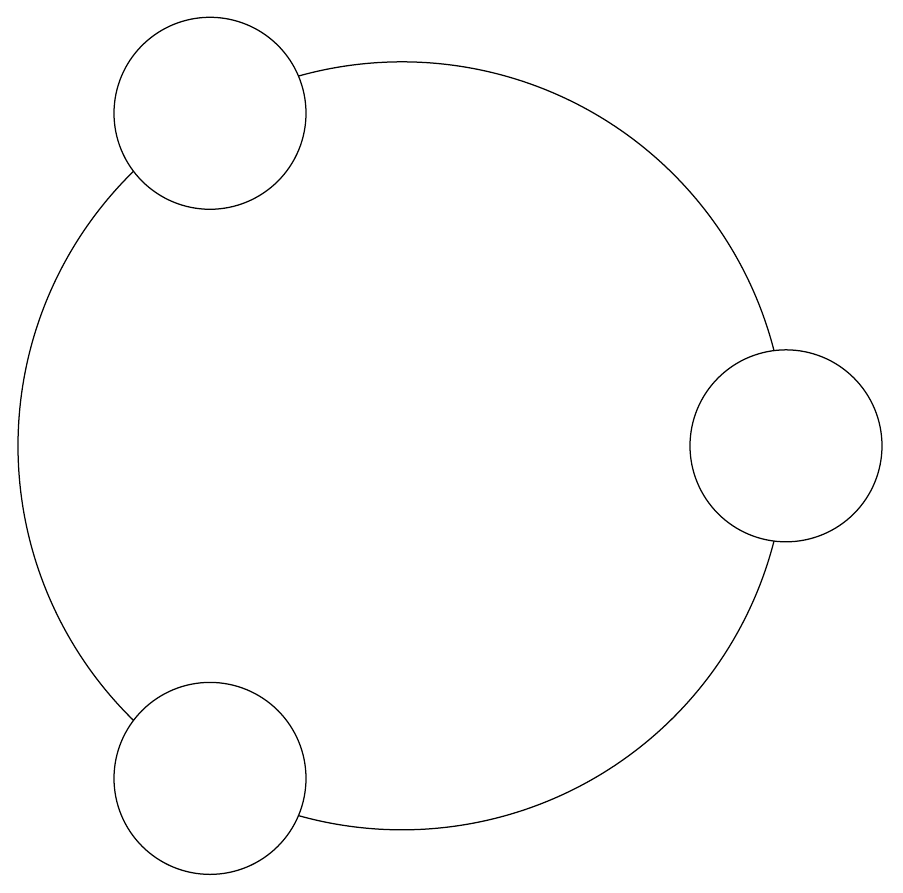}}\put(-39,22){$\bb{n}$}\put(-8,5){$\bb{s}$}\put(-39,-12){$\bb{p}$}\, , \, \raisebox{-.45\height}{\includegraphics[scale=0.18]{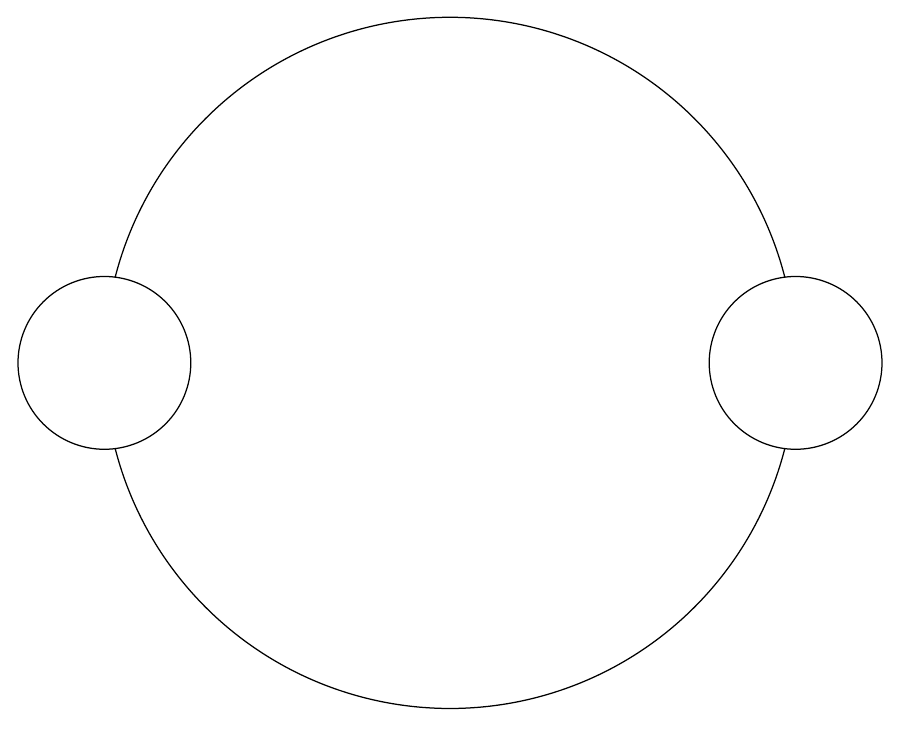}}\put(-45,0){$\bb{n}$}\put(-8,0){$\bb{s}$}\Bigg\rbrace}\mapsto \frac{1}{5!}\,[\bb{nsp}][\bb{ns}]
\end{equation*}
\end{figure}

\noindent The following result is a mere reformulation of \cite[Theorem 2.11]{Shalile_Br-center} (see also  \cite[Proposition 9]{KMP_central_idempotents} in terms of bracelets.
\begin{proposition}\label{prop:classes-bracelets}
    The map $\Phi$ is constant on the classes of conjugate elements in $B_{n}(\delta)$, so for any diagram $b\in B_{n}(\delta)$ the element $\Phi(\e_{b}) = n!\, \Phi(b)$ is a monic monomial in $[w_1]\dots [w_r]\in\mathbb{C}[\mathfrak{b}(\mathcal{A})]_n$. The restriction $\Phi\big|_{C_{n}(\delta)}$ is the isomorphism of linear spaces.
\end{proposition}
\noindent Note that for a permutation $s\in\Sn{n}$ one has $\Phi(\e_s) = [\bb{p}^{\lambda_1}] \dots [\bb{p}^{\lambda_r}]$. Without loss of generality, $\lambda_1 \geqslant \dots \geqslant \lambda_r$, so one arrives at a $r$-partition of $n$, which reflects the well-known fact that classes of conjugate elements of $\Sn{n}$ are in one-to-one correspondence with partitions of $n$.
\vskip 4 pt

The bijection described in Proposition \ref{prop:classes-bracelets} allows one to construct the basis in $C_{n}(\delta)$ labelled by monic monomials in $\mathbb{C}[\mathfrak{b}(\mathcal{A})]_n$ \cite[Corollary 2.12]{Shalile_Br-center} (see also \cite[Lemma 11]{KMP_central_idempotents}). Namely, by inverting $\Phi(\e_b) = \zeta$ we introduce the following linear map $e_{\zeta} = \e_{b} \in C_{n}(\delta)$ ($b\in B_{n}(\delta)$ is fixed modulo conjugacy equivalence).
    \begin{equation}\label{eq:basis_bracelets}
       \text{Then the set}\;\;\left\{e_{\zeta}\;:\; \zeta\in \mathbb{C}[\mathfrak{b}(\mathcal{A})]_n\;\;\text{(monic monomials)}\right\} \subset C_{n}(\delta)\;\; \text{is the sought basis.}
    \end{equation}
For example, one has the following basis in $C_{3}(N)$ parametrised by monomials $\mathfrak{b}\in\mathbb{C}[\mathfrak{b}(\mathcal{A})]_{3}$:
\begin{figure}[H]
\begin{equation}\label{eq:classes_Br_3}
    \begin{array}{rll}
        e_{[\bb{p}]^3} &= 6\, \ \raisebox{-.4\height}{\includegraphics[width=25pt,height=20pt]{Id3.pdf}} = \gamma_{1}\, 
        & \Phi(\gamma_{1}) = [\bb{p}][\bb{p}][\bb{p}]\,, \vspace{0.2cm}\\
        {e_{[\bb{pp}][\bb{p}]}} &= 2 \,\Bigl(\  \raisebox{-.4\height}{\includegraphics[width=25pt,height=20pt]{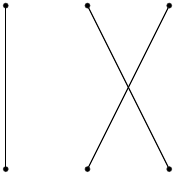}}+\raisebox{-.4\height}{\includegraphics[width=25pt,height=20pt]{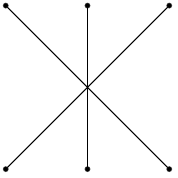}}+\raisebox{-.4\height}{\includegraphics[width=25pt,height=20pt]{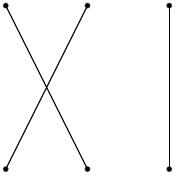}}\  \Bigr) = \e_{s_1}\,
        & \Phi(\gamma_{s_1}) = [\bb{p}^2][\bb{p}]\,,\vspace{0.2cm}\\
        e_{[\bb{ppp}]} &= 3\,\Bigl(\  \raisebox{-.4\height}{\includegraphics[width=25pt,height=20pt]{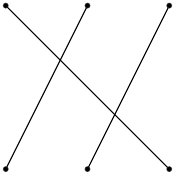}}+\raisebox{-.4\height}{\includegraphics[width=25pt,height=20pt]{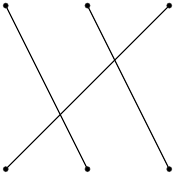}} \  \Bigr) = \e_{s_1 s_2}\,
        & \Phi(\e_{s_1 s_2}) = [\bb{p}^3]\,,\vspace{0.2cm}\\
        e_{[\bb{ns}][\bb{p}]} &= 2\,\Bigl(\  \raisebox{-.4\height}{\includegraphics[width=25pt,height=20pt]{b3_1.pdf}}+\raisebox{-.4\height}{\includegraphics[width=25pt,height=20pt]{b3_2.pdf}}+\raisebox{-.4\height}{\includegraphics[width=25pt,height=20pt]{b3_3.pdf}}\  \Bigr)\, = \e_{d_1} = 2A_3
        & \Phi(\e_{d_1}) = [\bb{n}\bb{s}][\bb{p}]\,,\vspace{0.2cm}\\
        e_{[\bb{nsp}]} &= \, \  \raisebox{-.4\height}{\includegraphics[width=25pt,height=20pt]{b3_4.pdf}}+\raisebox{-.4\height}{\includegraphics[width=25pt,height=20pt]{b3_5.pdf}}+\raisebox{-.4\height}{\includegraphics[width=25pt,height=20pt]{b3_6.pdf}}+\raisebox{-.4\height}{\includegraphics[width=25pt,height=20pt]{b3_7.pdf}}+\raisebox{-.4\height}{\includegraphics[width=25pt,height=20pt]{b3_8.pdf}}+\raisebox{-.4\height}{\includegraphics[width=25pt,height=20pt]{b3_9.pdf}} = \e_{d_1 s_2}  \, \hspace{1cm}
        & \Phi(\e_{d_1 s_2}) = [\bb{n}\bb{s}\bb{p}]\,.\\
    \end{array}
\end{equation}
\end{figure}
The basis \eqref{eq:basis_bracelets} allows us to express the left regular action of $A_n$ in $C_{n}(\delta)$ in terms of a linear operator $\Delta$ on $\mathbb{C}[\mathfrak{b}(\mathcal{A})]_n$:
\begin{equation}\label{eq:A_matrix_elements}
    A_n*e_{\zeta} = \sum_{\xi\;\text{(monic monomials in}\;\mathbb{C}[\mathfrak{b}(\mathcal{A})]_n\text{)}} e_{\xi}\,\Delta^{\xi}{}_{\zeta} = e_{\Delta(\zeta)}\,.
\end{equation}
The above formula serves as a definition which allows one to construct $\Delta$ by evaluating the products of Brauer diagrams in the left-hand-side of \eqref{eq:A_matrix_elements}. For example, by direct computation one finds the left regular action of $A_3$ on the above basis vectors of the above example, which in turn fixes the action of $\Delta$ on $\mathbb{C}[\mathfrak{b}(\mathcal{A})]_3$:
\begin{equation}\label{eq:delta_direct}
\def\arraystretch{1.3}
    \begin{array}{rlcl}
        A_3*e_{[\bb{p}]^3} & = 3\,e_{[\bb{ns}][\bb{p}]}\,, & \multirow{5}{*}{$\Rightarrow$} & \Delta\big([\bb{p}]^3\big) = 3\,[\bb{ns}][\bb{p}]\,,\\
        A_3*e_{[\bb{pp}][\bb{p}]} & = e_{[\bb{ns}][\bb{p}]} + 2\,e_{[\bb{nsp}]}\,, & & \Delta\big([\bb{pp}][\bb{p}]\big) = [\bb{ns}][\bb{p}] + 2\,[\bb{nsp}]\,,\\
        A_3*e_{[\bb{ppp}]} & = 3\,e_{[\bb{ns}][\bb{p}]}\,, & & \Delta\big([\bb{ppp}]\big) = 3\,[\bb{ns}][\bb{p}]\,,\\
        A_3*e_{[\bb{ns}][\bb{p}]} & = \delta\,e_{[\bb{ns}][\bb{p}]} + 2\,e_{[\bb{nsp}]}\,, & & \Delta\big([\bb{ns}][\bb{p}]\big) = \delta\,[\bb{ns}][\bb{p}] + 2\,[\bb{nsp}]\,,\\
        A_3*e_{[\bb{nsp}]} & = e_{[\bb{ns}][\bb{p}]} + (\delta + 1)\,e_{[\bb{nsp}]}\,, & &\Delta\big([\bb{nsp}]\big) = [\bb{ns}][\bb{p}] + (\delta + 1)\,[\bb{nsp}]\,.
    \end{array}
\end{equation}
Our next goal consists in describing the action of $\Delta$ directly in terms of bracelets, which will allow us to treat \eqref{eq:A_matrix_elements} other way around and to read off the left action of $A_n$ on $C_{n}(\delta)$ without addressing to diagram computations.

\subsection{Left regular action of $A_n$ on $C_{n}(\delta)$ via bracelets}

\paragraph{Derivation of bracelets over $\bar{\mathcal{A}}$.} Consider the polynomial algebra generated by bracelets over the extended alphabet $\bar{\mathcal{A}} = \mathcal{A}\cup \dot{\mathcal{A}} \cup \ddot{\mathcal{A}}$, where $\dot{\mathcal{A}} = \{\db{s},\db{p}\}$, $\ddot{\mathcal{A}} = \{\dd{s}\}$. Consider the (linear) derivation map $\partial$ which acts via Leibniz rule: on any monomial $[w_1]\dots [w_k]$ as
\begin{equation}
    \partial \big([w_1]\dots [w_p]\big) = \sum_{j=1}^{p} [w_1]\dots\partial\big([w_j]\big)\dots [w_p]\,,
\end{equation}
and on each bracelet $[w] = [\bb{a}_1\dots \bb{a}_{\ell}]$ ($\bb{a}_j\in \bar{\mathcal{A}}$ for all $j = 1,\dots,\ell$) as
\begin{equation}
    \partial[\bb{a}_1\dots \bb{a}_{\ell}] = \sum_{j=1}^{\ell} [\bb{a}_1\dots \partial (\bb{a}_j)\dots \bb{a}_{\ell}]\,.
\end{equation}
To fix $\partial$, we set
\begin{equation}
\def\arraystretch{1.4}
    \begin{array}{l}
        \partial(\bb{s}) = \db{s}\,,\;\;\partial(\bb{p}) = \db{p}\,,\;\; \partial(\db{s}) = \dd{s}\,,\\
        \text{and}\quad \partial(\bb{n}) = \partial(\db{p}) = \partial(\dd{s}) = 0\;\;\text{(any bracelet where $0$ occurs is put to $0$)}\,. 
    \end{array}
\end{equation}
In other words, the letters $\bb{n}$, $\db{p}$ and $\dd{s}$ are constants with respect to $\partial$.
\vskip 4 pt

Define the set $\mathfrak{b}(\bar{\mathcal{A}})\supset \mathfrak{b}(\mathcal{A})$ by extending $\mathfrak{b}(\mathcal{A})$ by all possibilities to substitute the undotted letters $\bb{s}$, $\bb{p}$ at some positions by their dotted counterparts in $\bar{\mathcal{A}}$. For example, the bracelet $[\bb{nsp}]$ gives rise to the following bracelets $[\bb{n}\db{s}\bb{p}], [\bb{n}\bb{s}\db{p}], [\bb{n}\db{s}\db{p}], [\bb{n}\dd{s}\bb{p}], [\bb{n}\dd{s}\db{p}]\in \mathfrak{b}(\bar{\mathcal{A}})$. As a result, $\mathbb{C}[\mathfrak{b}(\bar{\mathcal{A}})]$ contains all images of $\mathbb{C}[\mathfrak{b}(\mathcal{A})]$ upon consecutive application of $\partial$. The algebra $\mathbb{C}[\mathfrak{b}(\bar{\mathcal{A}})]$ is bi-graded: 
\begin{equation}\label{eq:bracelets_bi-grading}
    \mathbb{C}[\mathfrak{b}(\bar{\mathcal{A}})] = \bigoplus_{p\geqslant 0}\bigoplus_{q\leqslant p} \mathbb{C}[\mathfrak{b}(\bar{\mathcal{A}})]^{(q)}_{p}\,,
\end{equation}
where a monomial $[w_1] \dots [w_r]\in \mathbb{C}[\mathfrak{b}(\bar{\mathcal{A}})]^{(q)}_{p}$ has the total length $p$ ({\it i.e.} $|w_1|+\dots+|w_r| = p$) and carries the total amount $q$ of dots above the letters. For small $q$ the degree $(q)$ will be indicated by $q$ times the symbol $\prime$, and $\mathbb{C}[\mathfrak{b}(\bar{\mathcal{A}})]^{(0)}_{n} = \mathbb{C}[\mathfrak{b}(\mathcal{A})]_{n}$. In the sequel, omitting one of the bi-degrees of a component in \eqref{eq:bracelets_bi-grading} will imply the direct sum over all possible values of the omitted component.  For example, $[\db{p}], [\bb{n}\db{s}], [\bb{n}\db{s}][\bb{p}]\in \mathbb{C}[\mathfrak{b}(\bar{\mathcal{A})}]^{\prime}$, $[\bb{n}\dd{s}], [\bb{n}\dd{s}\bb{p}], [\bb{n}\db{s}][\db{p}\bb{p}]\in \mathbb{C}[\mathfrak{b}(\bar{\mathcal{A}})]^{\prime\prime}$ and $[\bb{n}\db{s}\bb{p}], [\bb{n}\dd{s}\bb{p}], [\bb{n}\db{s}][\db{p}] \in \mathbb{C}[\mathfrak{b}(\bar{\mathcal{A}})]_3$. The next lemma is a simple consequence of the definition of $\partial$ and the structure of $\mathfrak{b}(\bar{\mathcal{A}})$.
\begin{lemma}
    The map $\partial$ carries the bi-degree $\left({}^{1}_{0}\right)$,
\begin{equation}
    \partial: \mathbb{C}[\mathfrak{b}(\bar{\mathcal{A}})]^{(q)}_{p} \to \mathbb{C}[\mathfrak{b}(\bar{\mathcal{A}})]^{(q+1)}_{p}\,,
\end{equation}
and for any $[w]\in \mathbb{C}[\mathfrak{b}(\bar{\mathcal{A}})]^{(q)}_{p}$ there is $\partial^{p-q+1}[w] = 0$. 
\end{lemma}
\noindent For example,
\begin{equation}
\def\arraystretch{1.4}
\begin{array}{ll}
    \text{for}\;\;[\bb{n}\bb{s}][\bb{p}] \in \mathbb{C}[\mathfrak{b}(\bar{\mathcal{A}})]^{(0)}_{3}\;\;\text{one has} & \partial\big([\bb{n}\bb{s}][\bb{p}]\big) = [\bb{n}\db{s}][\bb{p}] + [\bb{n}\bb{s}][\db{p}]\in \mathbb{C}[\mathfrak{b}(\bar{\mathcal{A}})]^{\prime}_{3}\,,\\
    \hfill & \partial^2\big([\bb{n}\bb{s}][\bb{p}]\big) = [\bb{n}\dd{s}][\bb{p}] + 2\,[\bb{n}\db{s}][\db{p}]\in \mathbb{C}[\mathfrak{b}(\bar{\mathcal{A}})]^{\prime\prime}_{3}\,,\\
    \hfill & \partial^3\big([\bb{n}\bb{s}][\bb{p}]\big) = 3\,[\bb{n}\dd{s}][\db{p}] \in \mathbb{C}[\mathfrak{b}(\bar{\mathcal{A}})]^{\prime\prime\prime}_{3}\,,\\
    \hfill & \partial^4\big([\bb{n}\bb{s}][\bb{p}]\big) = 0\,.\\
\end{array}
\end{equation}

\paragraph{Trace operation.} To introduce the final ingredient for construction of $\Delta$, we consider the following $\mathbb{C}[\mathfrak{b}(\mathcal{A})]$-linear operation:
\begin{equation}
    \tau : \mathbb{C}[\mathfrak{b}(\bar{\mathcal{A}})]^{\prime\prime} \to \mathbb{C}[\mathfrak{b}(\mathcal{A})]\,,
\end{equation}
which is defined via the following rules. To formulate them, we accept a number of notations: {\it i)} we will write, for example, $[\bb{a}v]$ or $[\bb{a}u\bb{b}v]$ to specify particular letters $\bb{a},\bb{b}\in \bar{\mathcal{A}}$ in a bracelet, with the subwords $u,v$ being either empty or containing only letters from $\mathcal{A}$, {\it ii)} we will write $|w|_{\bb{a}}$ for the number of occurrences of the letter $\bb{a}$ in $w$ {\it iii)} we will say that $w$ is {\it fit} if it is either empty or $[w]\in\mathfrak{b}(\mathcal{A})$, and if each occurrence of $\bb{s}$ (if any) is followed by an occurrence of $\bb{n}$ at some position on the right. In the following expressions, the (sub)words $\db{p}u,\db{p}v$ on the left-hand-sides are assumed to have $u,v$ fit unless else is specified:
\begin{align}
        \tau : [\dd{s}u] \mapsto 2\delta\,[\bb{s}u] \,, & \hfill  \label{eq:trace_rule_1}\\
        \tau : [\db{s}u\db{s}v] \mapsto 2\,\big([\bb{s}u\bb{s}\,I(v)] + [\bb{s}u][\bb{s}v]\big)\,, & \quad\tau : [\db{s}u][\db{s}v] \mapsto 2\,\big([\bb{s}u\bb{s}v] + [\bb{s}u\bb{s}\,I(v)]\big)\,,\label{eq:trace_rule_2}\\
        \tau : [\db{p}u\db{s}v] \mapsto [\bb{p}u\bb{s}\,I(v)] + [\bb{p}u][\bb{s}v]\,, & \quad \tau : [\db{p}u][\db{s}v] \mapsto [\bb{p}u\bb{s}v] + [\bb{p}u\bb{s}\,I(v)]\,,\label{eq:trace_rule_3}\\
        \tau : [\db{p}u\db{p}v] \mapsto \left\{
        \begin{array}{ll}
            [\bb{n}u\bb{s}I(v)]\,, & \text{if $u,v$ are fit}\,, \\
            \left[\bb{n}u][\bb{s}v\right]\,, & |u|_{\bb{s}} > |u|_{\bb{n}}
        \end{array}
        \right.\,, & \quad\tau : [\db{p}u][\db{p}v] \mapsto [\bb{n}u\bb{s}I(v)]\,. \label{eq:trace_rule_4}
\end{align}
\begin{lemma}\label{lem:trace_rules}
    The rules \eqref{eq:trace_rule_1}-\eqref{eq:trace_rule_4} are correct and define $\tau$ unambiguously. Namely, {\it i)} the monomials entering the left-hand-sides of the rules form a $\mathbb{C}[\mathfrak{b}(\mathcal{A})]$-basis\footnote{In other words, any element in $\mathbb{C}[\mathfrak{b}(\bar{\mathcal{A}})]^{\prime\prime}$ is given by a unique combination of the left-hand-sides of \eqref{eq:trace_rule_1}-\eqref{eq:trace_rule_4}  with the coefficients in $\mathbb{C}[\mathfrak{b}(\mathcal{A})]$.} in $\mathbb{C}[\mathfrak{b}(\bar{\mathcal{A}})]^{\prime\prime}$,
    {\it ii)} if there are several representatives of a kind, the rules nevertheless lead to the same result.
\end{lemma}
\noindent For the proof see Appendix \ref{app:proof_lem_rules}.
\vskip 4 pt

We are in a position to express the operator $\Delta$ defined in \eqref{eq:A_matrix_elements} as a second-order differential operator on bracelets (see Appendix \ref{app:proof_Laplace} for proof).
\begin{theorem}\label{thm:Laplace}
    The operator $\Delta$ defined in \eqref{eq:A_matrix_elements} is given by
    \begin{equation}\label{eq:Laplace}
        \Delta = \frac{1}{2}\,\tau\circ \partial^{2}\,.
    \end{equation}
\end{theorem}

For example, consider the algebra $C_{3}(\delta)\subset B_{3}(\delta)$. As a matter of consistency let us check that $A_3 * e_{[\bb{p}]^3} = A_3*6 = 3\, e_{[\bb{ns}][\bb{p}]}$. Indeed,
\begin{equation}\label{eq:Laplace_exp_Br_3}
    \partial^2 [\bb{p}]^3 = 6\,[\db{p}]^2[\bb{p}]\quad \Rightarrow \quad \Delta\left([\bb{p}]^3\right) = 3\,[\bb{n}\bb{s}][\bb{p}]\,.
\end{equation}
For the rest one has:
\begin{equation}\label{eq:Laplace_Br_3}
\def\arraystretch{1.4}
    \begin{array}{lcl}
        \partial^2 \left[\bb{p}\bb{p}\right]\left[\bb{p}\right] = 2\,\left[\db{p}\db{p}\right]\left[\bb{p}\right] + 4\, \left[\db{p}\bb{p}\right]\left[\db{p}\right] & \multirow{4}{*}{$\Rightarrow$} & \Delta \big(\left[\bb{p}\bb{p}\right]\left[\bb{p}\right]\big) = \left[\bb{n}\bb{s}\right]\left[\bb{p}\right] + 2\,\left[\bb{n}\bb{s}\bb{p}\right]\,,  \\
        \partial^2\left[\bb{p}\bb{p}\bb{p}\right] = 6\,\left[\db{p}\db{p}\bb{p}\right] & & \Delta\big(\left[\bb{p}\bb{p}\bb{p}\right]\big) = 3\,\left[\bb{n}\bb{s}\bb{p}\right]\,,\\
        \partial^2\left[\bb{n}\bb{s}\right]\left[\bb{p}\right] = \left[\bb{n}\dd{s}\right]\left[\bb{p}\right] + 2\,\left[\bb{n}\db{s}\right]\left[\db{p}\right] & & \Delta\big(\left[\bb{n}\bb{s}\right]\left[\bb{p}\right]\big) = \delta\,\left[\bb{n}\bb{s}\right]\left[\bb{p}\right] + 2\,\left[\bb{n}\bb{s}\bb{p}\right]\,,\\
        \partial^2\left[\bb{n}\bb{s}\bb{p}\right] = \left[\bb{n}\dd{s}\bb{p}\right] + 2\,\left[\bb{n}\db{s}\db{p}\right] & & \Delta\big(\left[\bb{n}\bb{s}\bb{p}\right]\big) = \left[\bb{n}\bb{s}\right]\left[\bb{p}\right] + (\delta+1)\,\left[\bb{n}\bb{s}\bb{p}\right]\,,\\
    \end{array}
\end{equation}
which reproduces the result \eqref{eq:delta_direct} obtained by direct diagram-wise computation. From the third line in \eqref{eq:Laplace_Br_3} one obtains $(A_3)^2 = \delta\,A_3 + e_{[\bb{nsp}]}$, which gives a shortcut towards the final expression in \eqref{eq:projector_Br_3}.
\vskip 4 pt

\paragraph{Cardinality $|C_{\Sn{n}}(b)|$ via symmetries of $\Phi(\e_b)$.} In the previous paragraphs we developed a technique to express the left regular action of $A_n$ on the basis $e_{\zeta} \in C_{n}(\delta)$ (where  $\zeta\in\mathbb{C}[\mathfrak{b}(\mathcal{A})]$ are monic monomials) avoiding diagram-wise computations. Nevertheless, in order to apply the final expanded form of $P_n$ \eqref{eq:projector} to a tensor, the basis elements $e_{\zeta} = \e_b$ should be expressed as sums of conjugate Brauer diagrams \eqref{eq:normalised_classes}. In this respect it is convenient to rescale the basis $e_{\zeta}$ and use the normalised sums
\begin{equation}\label{eq:basis_normalised}
    \check{e}_{\zeta} = \frac{1}{|C_{\Sn{n}}(b)|}\, \e_b\,,
\end{equation}
which are sums of conjugate diagrams with the overall coefficient $1$ (recall \eqref{eq:normalised_classes}). In particular, $A_n = \check{e}_{[\bb{ns}][\bb{p}]^{n-2}}$. It appears that the coefficient $|C_{\Sn{n}}(b)|$ can be expressed in terms of certain symmetry properties of the representatives in bracelets, which allows one to avoid diagram computations again. 
\vskip 4 pt

Consider the group of cyclic permutations $\mathbb{Z}_{\ell}$ which acts on the words of the length $\ell$. Recall that $I(w)$ denotes the inversion of a word $w$. Given any word $w$ of length $\ell$, define its {\it turnover stabilizer}:
\begin{equation}\label{eq:stabilizer_word}
    S(w) = \left\{h\in \mathbb{Z}_{\ell}\;\; : \;\; \text{either}\;\; h(w) = w\;\; \text{or}\;\; h(w) = I(w)\right\}\,.
\end{equation}
Note that both conditions in \eqref{eq:stabilizer_word} are simultaneously satisfied by an element $h\in\mathbb{Z}_{\ell}$ only if $w = I(w)$, {\it i.e.} the word is inversion-invariant. In the case of representatives of bracelet from $\mathfrak{b}(\mathcal{A})$, inversion-invariant words appear only as representatives of $[\bb{p}\dots\bb{p}]$. In this case $\big|S(\underbrace{\bb{p}\dots\bb{p}}_{\ell})\big| = \ell$. In the other cases, when $\bb{n}$ and $\bb{s}$ occur in a bracelet from $\mathfrak{b}(\mathcal{A})$, none of its representatives is inversion-invariant, so a permutation $h\in \mathbb{Z}_{\ell}$ satisfies at most one of the conditions in \eqref{eq:stabilizer_word}.
\vskip 4 pt

As a side remark, note the following properties concerning the definition of the stabilizer \eqref{eq:stabilizer_word}. It can happen that the conditions in \eqref{eq:stabilizer_word} are satisfied only for the trivial transformation: for example, in the case of the bracelet $[\bb{n}\bb{s}\bb{p}]$. Another observation is that the existence of a transformation $h\in\mathbb{Z}_{\ell}$ such that $h(w) = w$ (respectively, $h(w) = I(w)$) does not imply necessarily the existence of $h^{\prime}\in\mathbb{Z}_{\ell}$ such that $h^{\prime}(w) = I(w)$ (respectively, $h^{\prime}(w) = w$): as an example, consider $[\bb{n}\bb{s}\bb{p}\bb{n}\bb{s}\bb{p}]$ (respectively, $[\bb{n}\bb{s}\bb{n}\bb{p}\bb{s}\bb{p}]$).
\vskip 4 pt

For a bracelet $[w]\in\mathfrak{b}(\mathcal{A})$, the cardinality $|S(w)|$ does not depend on particular choice of a representative, so the function $\mathrm{st}\big([w]\big) = |S(w)|$ is well-defined on bracelets.
\vskip 4 pt

For a monomial in $\mathbb{C}[\mathfrak{b}(\mathcal{A})]$ define its {\it turnover stability index} as follows:
\begin{equation}
    \mathrm{st}\big([w_1]^{m_1}\dots [w_r]^{m_r}\big) = \prod_{j=1}^{r} \mathrm{st}\big([w_j]\big)\,m_{j}!\,.
\end{equation}
The following lemma gives a convenient method of calculating the coefficient in \eqref{eq:normalised_classes} (see Appendix \ref{app:lemma_sym} for proof).
\begin{lemma}\label{lem:sym} For any diagram $b\in B_{n}(\delta)$ holds $|C_{\Sn{n}}(b)| = \mathrm{st}\big(\Phi(\e_b)\big)$. Thus, the change of basis \eqref{eq:basis_normalised} is written as
\begin{equation}
    \check{e}_{\zeta} = \frac{1}{\mathrm{st}(\zeta)}\,e_{\zeta}\,.
\end{equation}
\end{lemma}

\paragraph{(Non)-commutativity of $C_{n}(\delta)$.} Restriction to the subalgebra $\mathbb{C}\Sn{n}\subset B_{n}(\delta)$ is the central ingredient in the proposed construction of the traceless projector. The algebra $C_{n}(\delta)$ centralises $\mathbb{C}\Sn{n}$ in $B_{n}(\delta)$ and {\it vice versa}. Along the same lines as explained in Section \ref{sec:basics} about the centraliser algebras, $C_{n}(\delta)$ takes care of the multiplicities in the decomposition of a simple $B_{n}(\delta)$-module into simple $\mathbb{C}\Sn{n}$-modules: this is because $C_{n}(\delta)$ contains all intertwiners between equivalent $\mathbb{C}\Sn{n}$-submodules. By applying Schur's lemma, $C_{n}(\delta)$ is commutative iff multiplicities in the decompositions in question are bounded by $1$. 
\vskip 4 pt

In order to analyse commutativity of $C_{n}(\delta)$ we make use of the flip operation $(\,\cdot\,)^{*}$ \eqref{eq:flip} (defined as a linear map for any $\delta\in\mathbb{C}$), such that $C_{n}(\delta)$ is the algebra with an anti-involution. Then by \cite[Lemma 2.3]{OV_2}, the algebra $C_{n}(\delta)$ is commutative iff for any $u\in C_{n}(\delta)$ holds $u^{*} u = u u^{*}$. To analyse the latter condition we will make use of the basis \eqref{eq:basis_bracelets} and introduce an involution on bracelets which commutes with the isomorphism in Proposition \ref{prop:classes-bracelets}. Namely, upon $B_{n}(\delta)\ni b \mapsto b^{*}$ the arcs in the upper (respectively, lower) row are placed to the lower (respectively, upper) row, while vertical lines remain vertical. Then by construction of the map $\Phi$, the image $\Phi(b^{*})$ is obtained from $\Phi(b)$ by mapping $\bb{n} \mapsto \bb{s}$, $\bb{s} \mapsto \bb{n}$ in each bracelet\footnote{In other words, one extends the involutive map $\bb{n}\mapsto \bb{s}$, $\bb{s} \mapsto \bb{n}$ on $\mathcal{A}$ to the automorphism of the monoid of words over $\mathcal{A}$, and then to $\mathbb{C}[\mathfrak{b}(\mathcal{A})]$ by linearity.} (see \cite[Lemma 2.14]{Shalile_Br-center}). Abusing notation once again, 
\begin{equation}\label{eq:ns-flip}
    \text{for any}\quad \xi\in\mathbb{C}[\mathfrak{b}(\mathcal{A})]\quad \text{define}\quad \xi^{*}\quad \text{as obtained by}\quad \bb{n}\mapsto \bb{s}\,,\;\; \bb{s} \mapsto \bb{n}.
\end{equation}
\vskip 4 pt

With \eqref{eq:ns-flip} at hand, $C_{n}(\delta)$ is commutative iff $e_{\xi} e_{\xi^{*}} =  e_{\xi^{*}} e_{\xi}$ for any monic monomial $\xi \in\mathbb{C}[\mathfrak{b}(\mathcal{A})]_{n}$. The analysis of the latter condition constitutes the proof of the following lemma (see Appendix \ref{app:lemma_commutativity} for proof).
\begin{lemma}\label{lem:commutativity}
    For any $\delta\in\mathbb{C}$ the algebra $C_{n}(\delta)\subset B_{n}(\delta)$ is commutative for $n \leqslant 5$, and is non-commutative for $n\geqslant 6$. More precisely, for $n\geqslant 6$ and $f_{\mathrm{max}} = \lfloor\tfrac{n}{2}\rfloor$:
    \begin{itemize}
        \item[{\it i)}] the quotient algebra $C_{n}(\delta)\slash J^{(f)}$ is commutative for $f = 1,2$ and is non-commutative for $f = 3,\dots,f_{\mathrm{max}}$;
        \item[{\it ii)}] the subalgebra $C_{n}(\delta)\cap J^{(f)}$ is commutative for $f = f_{\mathrm{max}}$ and is non-commutative for $f = 0,\dots,f_{\mathrm{max}} - 1$.
    \end{itemize}
\end{lemma}

As a particular example, $C_{n}(\delta)\slash J^{(1)} \cong C_{n}(\delta)\cap \mathbb{C}\Sn{n}$, which is known to be the center of $\mathbb{C}\Sn{n}$. For any $s\in \Sn{n}$ one has $\Phi(\gamma_{s}) = [\bb{p}^{\lambda_1}]\dots [\bb{p}^{\lambda_r}]$, so $\Phi(\gamma_{s}) = \Phi(\gamma_{s})^{*}$. Since $(\,\cdot\,)^{*}$ commutes with the isomorphism in Proposition \ref{prop:classes-bracelets}, one has $\gamma_s = (\gamma_s)^{*}$. Therefore, for any any central element $u\in\mathbb{C}\Sn{n}$ holds $u^{*} = u$ and $\mathfrak{r}(u)^{*} = \mathfrak{r}(u)$. In particular, for central Young symmetrisers (utilised in \eqref{eq:projector_semi_sum}) one has $(z^{(\mu)})^{*} = z^{(\mu)}$.
\vskip 4 pt

\section{Splitting idempotent}\label{sec:splitting_idempotent}

\paragraph{Traceless subspace as a quotient space.} In order to introduce the subject of this section, let us recall that the subspace of traceless tensors can be identified with equivalence classes $V^{\otimes n}\slash W$ (of $V^{\otimes n}$ modulo $W$, the subspace of tensors proportional to metric). If one denotes the canonical embedding $\iota : W \rightarrow V^{\otimes n}$ and the projection $\pi : V^{\otimes n} \rightarrow V^{\otimes n}\slash W$, then the existence of a traceless projection is equivalent to existence of a section $\sigma:  V^{\otimes n}\slash W \rightarrow V^{\otimes n}$, such that the following short exact sequence splits:
\begin{equation}\label{eq:split_tensor}
\def\arraystretch{1.4}
\begin{array}{l}
    0\xrightarrow{\quad} W \xrightarrow{\quad\iota\quad} V^{\otimes n} \xrightleftarrows[\text{exists}\;\; \sigma]{\pi} V^{\otimes n}\slash W \xrightarrow{\quad} 0\,, \;\; \text{where}\;\; \pi\circ \sigma = \mathrm{id}\,, \\
    \text{and for}\;\; \mathfrak{P}_n = \sigma\circ \pi\;\;\text{one has}\;\; V^{\otimes n} = W \oplus \mathfrak{P}_n V^{\otimes n}\,.
\end{array}
\end{equation}
In this context, the projector $\mathfrak{P}_n$ is referred to as {\it splitting idempotent} of the above short exact sequence. By Theorem \ref{thm:projector}, $\mathfrak{P}_n$ is be expressed as a $\mathfrak{r}$-image of an element $P_n\in B_{n}(\varepsilon N)$ \eqref{eq:projector}. 
\vskip 4 pt

In this section we are interested in the properties of the element $P_n$ itself, as an operator in the left regular $B_{n}(\delta)$-module ({\it i.e.} its action on the vector space $B_{n}(\delta)$ via left multiplication). We assume the semisimple regime for $B_{n}(\delta)$ by allowing $\delta\in \mathbb{C}\backslash \{0\}$, while for $\delta \in \mathbb{Z}\backslash\{0\}$ we assume $\delta \leqslant - 2(n-1)$ and $\delta \geqslant n-1$ (for the semisimplicity criterion see \cite{Rui_Br_semisimple} and references therein).

\paragraph{$P_n$ as a splitting idempotent.} In analogy with the traceless projection of $V^{\otimes n}$, which is annihilated by the action of $J^{(1)}$, one can be interested in constructing the ``traceless'' projection of the left regular $B_{n}(\delta)$-module. Due to the equivalence of the left and right regular modules of $B_{n}(\delta)$ via $(\,\cdot\,)^{*}$, the sought projector should be a central idempotent $\bar{P}_n\in B_{n}(\delta)$ which provides the decomposition $B_{n}(\delta) = J^{(1)}\oplus \bar{P}_n B_{n}(\delta)$ with $\bar{P}_n B_{n}(\delta) \cong \mathbb{C}\Sn{n}$ (direct product of algebras). In other words, one looks for a splitting idempotent for the following exact sequence similar to the one in \eqref{eq:split_tensor}:
\begin{equation}\label{eq:splitting_sequence_B}
    \begin{array}{l}
        0\xrightarrow{\quad} J^{(1)} \xrightarrow{\quad} B_{n}(\delta) \xrightleftarrows[\text{exists}\;\; \bar{\sigma}]{\bar{\pi}} B_{n}(\delta)\slash J^{(1)}\,(\cong \mathbb{C}\Sn{n}) \xrightarrow{\quad} 0\,, \;\; \text{where}\;\; \bar{\pi}\circ \bar{\sigma} = \mathrm{id}\,, \\
        \text{and for}\;\; \bar{P}_n = \bar{\sigma}\circ \bar{\pi}\;\;\text{one has}\;\; B_{n}(\delta) = J^{(1)} \oplus \bar{P}_n B_{n}(\delta)\,.
    \end{array}
    \end{equation}
Such idempotent exists and is unique \cite{KMP_central_idempotents}. 
\begin{theorem}\label{thm:splitting_idempotent}
    Let $\delta\in\mathbb{C}\backslash\{0\}$ and $n\in \mathbb{N}$ are such that $B_{n}(\delta)$ is semisimple. Then $\bar{P}_n$ is constructed as \eqref{eq:projector} by dropping all restrictions on the numbers of rows/columns in the Young diagrams in the definitions \eqref{eq:def_Lambda}, \eqref{eq:def_Sigma}, \eqref{eq:LR-closure} and \eqref{eq:skew-shape}.
\end{theorem}
\begin{proof}
    First, let us show that $\bar{P}_n\in Z(B_{n}(\delta))$. By construction, $\bar{P}_n \in C_{n}(\delta)$, so it commutes with $\mathbb{C}\Sn{n}\subset B_{n}(\delta)$. We will prove that $\bar{P}_n J^{(1)} = J^{(1)} \bar{P}_n = 0$, which will imply in turn that $\bar{P}_n$ commutes also with any diagram from $J^{(1)}$. 
    \vskip 4 pt
    
    The fact that $B_{n}(\delta)$ is semisimple means that the left regular $B_{n}(\delta)$-module decomposes as a direct sum of simple $B_{n}(\delta)$-modules $M^{(\lambda)}_n$ indexed by $\lambda \vdash n - 2f$ for all $f = 0,\dots, \lfloor \tfrac{n}{2} \rfloor$ (each module  occurring at least once). Upon restriction to $\mathbb{C}\Sn{n}$, occurrence of simple $\mathbb{C}\Sn{n}$-modules $L^{(\mu)}\subset M^{(\lambda)}_n$ is described by Lemma \ref{lem:restriction_decomposition}, which becomes the criterion due to simplicity of the standard modules (with all restrictions on the size of Young diagrams omitted).
    \vskip 4 pt
    
    The eigenvalue $0$ of $A_n$ occurs only for $f = 0$ in \eqref{eq:spec_master_class}. This is obvious for non-integer values of $\delta$ due to Lemma \ref{lem:D_block_diagonal}, while for integer ones the bounds of the semisimple regime of $B_{n}(\delta)$ allow one to apply the alternative proof of Lemma \ref{lem:KerA}, which is presented in Appendix \ref{app:proof_KerA_Br}. By construction, left multiplication of $B_{n}(\delta)$ by $\bar{P}_n$ annihilates exactly the subspaces with non-zero eigenvalues of $A_n$, so $\bar{P}_n J^{(1)} = 0$. To prove that $J^{(1)} \bar{P}_n = 0$ as well, note that $(J^{(1)})^{*} = J^{(1)}$. The fact that $(A_n)^{*} = A_n$ implies that $(\bar{P}_n)^{*} = \bar{P}_n$, so $(J^{(1)}\bar{P}_n)^{*} = \bar{P}_n J^{(1)} = 0$. As an immediate consequence, $(\bar{P}_n)^2 = \bar{P}_n$ because $A_n \in J^{(1)}$.
    \vskip 4 pt
    
    To sum up, $\bar{P}_n$ is a central idempotent. The fact that the eigenvalue $0$ of $A_n$ happens only for $f = 0$ implies that $\bar{P}_{n} B_{n}(\delta) = \bar{P}_{n}\,\mathbb{C}\Sn{n} \cong \mathbb{C}\Sn{n}$, so $\bar{P}_{n}$ splits the short exact sequence \eqref{eq:splitting_sequence_B}.
\end{proof}

\section{Summary: the traceless projector for $n=4$}\label{sec:summary}
Let us sum up the proposed approach of constructing the traceless projector by constructing the splitting idempotent $\bar{P}_n\in B_{n}(\delta)$ step by step. Throughout this section the parameter $\delta \neq 0$ is assumed to be generic complex, such that the algebra $B_{n}(\delta)$ is semisimple. To get the traceless projector \eqref{eq:projector}, one simply puts $\delta\mapsto \varepsilon N$ and ignores the elements of $\mathrm{spec}^{\times}(A_4)$ which either vanish or carry the sign $-\varepsilon$.

\paragraph{Eigenvalues of $A_4$.} To arrange the set of eigenvalues let us make use of the fact that 
\begin{equation}
\def\arraystretch{0.7}
    \mathrm{spec}^{\times}(A_4) = \mathrm{spec}^{\times}_{(4)}(A_n) \cup \mathrm{spec}^{\times}_{(3,1)}(A_n) \cup \mathrm{spec}^{\times}_{(2,2)}(A_n) \cup \mathrm{spec}^{\times}_{(2,1,1)}(A_n)\,,
\end{equation}
so one uses \eqref{eq:spec_reduced} defined in Section \ref{sec:traceless_GL}. One has 
\begin{equation}
\begin{aligned}
    \mathrm{spec}_{(4)}^{\times}(A_4) &=\lbrace{\delta+4\, ,\, 2(\delta+2)}\rbrace, \qquad \mathrm{spec}_{(3,1)}^{\times}(A_4) =\lbrace{\delta\,,\,\delta+2}\rbrace \qquad\\
    \mathrm{spec}_{(2,2)}^{\times}(A_4) &=\lbrace{\delta-2\,,\, 2(\delta-1)}\rbrace, \qquad 
    \mathrm{spec}_{(2,1,1)}^{\times}(A_4) =\lbrace{\delta-2}\rbrace\,,
\end{aligned}
\end{equation}
A this level one can have zeros for particular non-zero integer values $\delta\in \{-4,-2,1,2\}$, which mark exactly the non-semisimple regime of $B_{n}(\delta)$. Also for $\delta = 1$, $\mathrm{spec}_{(2,2)}^{\times}(A_4)$ and $\mathrm{spec}_{(2,1,1)}^{\times}(A_4)$ contain a negative element to be excluded from the consideration. Nevertheless we proceed by assuming that all the above elements enter $\mathrm{spec}^{\times}(A_4)$. In order to obtain the additive form of the projector we now need to expand the product of factors $\left(1-\alpha^{-1} A_4\right)$ over all $\alpha \in\mathrm{spec}^{\times}(A_4)$. To do so we take advantage of the algorithm described in section (\ref{sec:bracelets}).
\paragraph{Action of $A_4$ on $C_{4}(\delta)$ and expansion of the factorised form of $P_4$.} 
The second derivative of the elements in $\mathbb{C}[\mathfrak{b}(\mathcal{A})]$ relevant for the construction of $P_4$ reads: 
\begin{equation*}
\def\arraystretch{1.4}
    \begin{array}{lcl}
    \partial^2 \left[\bb{n}\bb{s}\right]\left[\bb{p}\right]\left[\bb{p}\right] = \left[\bb{n}\dd{s}\right]\left[\bb{p}\right]\left[\bb{p}\right] + 4\,\left[\bb{n}\db{s}\right]\left[\db{p}\right]\left[\bb{p}\right]+ 2\,\left[\bb{n}\bb{s}\right]\left[\db{p}\right]\left[\db{p}\right]\,, 
    & &
    \partial^2\left[\bb{n}\bb{p}\bb{s}\bb{p}\right] = \left[\bb{n}\bb{p}\dd{s}\bb{p}\right] + 4\, \left[\bb{n}\bb{p}\db{s}\db{p}\right] + 2\,\left[\bb{n}\db{p}\bb{s}\db{p}\right]\,,\\
    
    \partial^2\left[\bb{n}\bb{s}\right]\left[\bb{p}\bb{p}\right] = \left[\bb{n}\dd{s}\right]\left[\bb{p}\bb{p}\right] + 4\,\left[\bb{n}\db{s}\right]\left[\db{p}\bb{p}\right]+2\left[\bb{n}\bb{s}\right]\left[\db{p}\db{p}\right]\,,
    & &
    \partial^2\left[\bb{n}\bb{s}\right]\left[\bb{n}\bb{s}\right] =2\,(\left[\bb{n}\dd{s}\right]\left[\bb{n}\bb{s}\right]+\left[\bb{n}\db{s}\right]\left[\bb{n}\db{s}\right])\,,\\
    \partial^2\left[\bb{n}\bb{s}\bb{p}\right]\left[\bb{p}\right] = \left[\bb{n}\dd{s}\bb{p}\right]\left[\bb{p}\right] + 2\,(\,\left[\bb{n}\db{s}\db{p}\right]\left[\bb{p}\right]+\left[\bb{n}\db{s}\bb{p}\right]\left[\db{p}\right]+\left[\bb{n}\bb{s}\db{p}\right]\left[\db{p}\right]\,) \,, 
    & &\partial^2\left[\bb{n}\bb{s}\bb{n}\bb{s}\right] = 2\,(\left[\bb{n}\dd{s}\bb{n}\bb{s}\right] +\left[\bb{n}\db{s}\bb{n}\db{s}\right])\,,\\
    \partial^2\left[\bb{n}\bb{s}\bb{p}\bb{p}\right] = \left[\bb{n}\dd{s}\bb{p}\bb{p}\right] + 2\,(\left[\bb{n}\db{s}\db{p}\bb{p}\right]+\left[\bb{n}\db{s}\bb{p}\db{p}\right]+\left[\bb{n}\bb{s}\db{p}\db{p}\right])\,.\\
    \end{array}
\end{equation*}
Evaluating the trace $\tau$ \eqref{eq:trace_rule_1}-\eqref{eq:trace_rule_4} on the previous expressions and using Theorem \ref{thm:Laplace} yields: 

\begin{equation}\label{eq:Laplace_Br_4}
\begin{aligned}
    & A_n * e_{\left[\bb{n}\bb{s}\right]\left[\bb{p}\right]\left[\bb{p}\right]} = \delta\,e_{\left[\bb{n}\bb{s}\right]\left[\bb{p}\right]\left[\bb{p}\right]} + 4\,e_{\left[\bb{n}\bb{s}\bb{p}\right]\left[\bb{p}\right]}+e_{\left[\bb{n}\bb{s}\right]\left[\bb{n}\bb{s}\right]}\,,  \\
    &A_n *e_{\left[\bb{n}\bb{s}\right]\left[\bb{p}\bb{p}\right]} =\delta\,e_{\left[\bb{n}\bb{s}\right]\left[\bb{p}\bb{p}\right]}+ 4\,e_{\left[\bb{n}\bb{s}\bb{p}\bb{p}\right]}+e_{\left[\bb{n}\bb{s}\right]\left[\bb{n}\bb{s}\right]}\,,\\
    &A_n *e_{\left[\bb{n}\bb{s}\bb{p}\right]\left[\bb{p}\right]} =\left( \delta+1 \right)\, e_{\left[\bb{n}\bb{s}\bb{p}\right]\left[\bb{p}\right]} + e_{\left[\bb{n}\bb{s}\right]\left[\bb{p}\right]\left[\bb{p}\right]}+e_{\left[\bb{n}\bb{s}\bb{p}\bb{p}\right]}+e_{\left[\bb{n}\bb{p}\bb{s}\bb{p}\right]}+e_{\left[\bb{n}\bb{s}\bb{n}\bb{s}\right]}\,,\\
    &A_n *e_{\left[\bb{n}\bb{s}\bb{p}\bb{p}\right]} =\left(\delta+1\right)\, e_{\left[\bb{n}\bb{s}\bb{p}\bb{p}\right]} +e_{\left[\bb{n}\bb{s}\right]\left[\bb{p}\bb{p}\right]} +e_{\left[\bb{n}\bb{s}\bb{p}\right]\left[\bb{p}\right]} +e_{\left[\bb{n}\bb{p}\bb{s}\bb{p}\right]}+e_{\left[\bb{n}\bb{s}\bb{n}\bb{s}\right]}\,,\\
     &A_n *e_{\left[\bb{n}\bb{p}\bb{s}\bb{p}\right]} =\delta\, e_{\left[\bb{n}\bb{p}\bb{s}\bb{p}\right]} +2\, e_{\left[\bb{n}\bb{s}\bb{p}\right]\left[\bb{p}\right]} + 2\,e_{\left[\bb{n}\bb{s}\bb{p}\bb{p}\right]}+e_{\left[\bb{n}\bb{s}\right]\left[\bb{n}\bb{s}\right]}\,,\\
    &A_n *e_{\left[\bb{n}\bb{s}\right]\left[\bb{n}\bb{s}\right]} = \, 2\delta\, e_{\left[\bb{n}\bb{s}\right]\left[\bb{n}\bb{s}\right]} + 4\,e_{\left[\bb{n}\bb{s}\bb{n}\bb{s}\right]}\,,\\
    &A_n *e_{\left[\bb{n}\bb{s}\bb{n}\bb{s}\right]} = \,2(\delta+1)\,e_{\left[\bb{n}\bb{s}\bb{n}\bb{s}\right]}+ 2\,e_{\left[\bb{n}\bb{s}\right]\left[\bb{n}\bb{s}\right]}\,.
\end{aligned}
\end{equation}

Before expanding the factorized form of the projector we need to express the above expression in the normalized basis $\check{e}_{\zeta}=\dfrac{1}{\mathrm{st}(\zeta)}\,e_{\zeta}$ described in Lemma \ref{lem:sym}.
The turnover stability index of each basis element of $\mathbb{C}[\mathfrak{b}(\mathcal{A})]_4$ reads:
\begin{equation*}
\def\arraystretch{1.4}
    \begin{array}{llllc}
    \mathrm{st}(\left[\bb{n}\bb{s}\right]\left[\bb{p}\right]\left[\bb{p}\right])=4\,, & \mathrm{st}(\left[\bb{n}\bb{s}\right]\left[\bb{p}\bb{p}\right])=4\,,  & 
    \mathrm{st}(\left[\bb{n}\bb{s}\bb{p}\right]\left[\bb{p}\right])=1\,, &
    \mathrm{st}(\left[\bb{n}\bb{s}\bb{p}\bb{p}\right])=1\,, \\ \mathrm{st}(\left[\bb{n}\bb{p}\bb{s}\bb{p}\right])=2\,, &
    \mathrm{st}(\left[\bb{n}\bb{s}\right]\left[\bb{n}\bb{s}\right])=8\,,&
    \mathrm{st}(\left[\bb{n}\bb{s}\bb{n}\bb{s}\right])=4\,. &\\
    \end{array}
\end{equation*}
Therefore, in the normalized basis the relations \eqref{eq:Laplace_Br_4} transform into : 
\begin{equation}\label{eq:A_action_normalized}
\begin{aligned}
    & A_n * \check{e}_{\left[\bb{n}\bb{s}\right]\left[\bb{p}\right]\left[\bb{p}\right]} = \delta\,\check{e}_{\left[\bb{n}\bb{s}\right]\left[\bb{p}\right]\left[\bb{p}\right]} + \check{e}_{\left[\bb{n}\bb{s}\bb{p}\right]\left[\bb{p}\right]}+2 \, \check{e}_{\left[\bb{n}\bb{s}\right]\left[\bb{n}\bb{s}\right]}\,,  \\
    &A_n *\check{e}_{\left[\bb{n}\bb{s}\right]\left[\bb{p}\bb{p}\right]} =\delta\,\check{e}_{\left[\bb{n}\bb{s}\right]\left[\bb{p}\bb{p}\right]}+ \check{e}_{\left[\bb{n}\bb{s}\bb{p}\bb{p}\right]}+2\, \check{e}_{\left[\bb{n}\bb{s}\right]\left[\bb{n}\bb{s}\right]}\,,\\
    &A_n *\check{e}_{\left[\bb{n}\bb{s}\bb{p}\right]\left[\bb{p}\right]} =\left( \delta+1 \right)\, \check{e}_{\left[\bb{n}\bb{s}\bb{p}\right]\left[\bb{p}\right]} + 4\, \check{e}_{\left[\bb{n}\bb{s}\right]\left[\bb{p}\right]\left[\bb{p}\right]}+\check{e}_{\left[\bb{n}\bb{s}\bb{p}\bb{p}\right]}+2\, \check{e}_{\left[\bb{n}\bb{p}\bb{s}\bb{p}\right]}+4\, \check{e}_{\left[\bb{n}\bb{s}\bb{n}\bb{s}\right]}\,,\\
    &A_n *\check{e}_{\left[\bb{n}\bb{s}\bb{p}\bb{p}\right]} =\left(\delta+1\right)\, \check{e}_{\left[\bb{n}\bb{s}\bb{p}\bb{p}\right]} +4\,\check{e}_{\left[\bb{n}\bb{s}\right]\left[\bb{p}\bb{p}\right]} +\check{e}_{\left[\bb{n}\bb{s}\bb{p}\right]\left[\bb{p}\right]} +2\,\check{e}_{\left[\bb{n}\bb{p}\bb{s}\bb{p}\right]}+4\,\check{e}_{\left[\bb{n}\bb{s}\bb{n}\bb{s}\right]}\,,\\
     &A_n *\check{e}_{\left[\bb{n}\bb{p}\bb{s}\bb{p}\right]} =\delta\, \check{e}_{\left[\bb{n}\bb{p}\bb{s}\bb{p}\right]} + \check{e}_{\left[\bb{n}\bb{s}\bb{p}\right]\left[\bb{p}\right]} +\check{e}_{\left[\bb{n}\bb{s}\bb{p}\bb{p}\right]}+4\,\check{e}_{\left[\bb{n}\bb{s}\right]\left[\bb{n}\bb{s}\right]}\,,\\
    &A_n *\check{e}_{\left[\bb{n}\bb{s}\right]\left[\bb{n}\bb{s}\right]} = \, 2\delta\, \check{e}_{\left[\bb{n}\bb{s}\right]\left[\bb{n}\bb{s}\right]} + 2\,\check{e}_{\left[\bb{n}\bb{s}\bb{n}\bb{s}\right]}\,,\\
    &A_n *\check{e}_{\left[\bb{n}\bb{s}\bb{n}\bb{s}\right]} = \,2(\delta+1)\,\check{e}_{\left[\bb{n}\bb{s}\bb{n}\bb{s}\right]}+ 4\,\check{e}_{\left[\bb{n}\bb{s}\right]\left[\bb{n}\bb{s}\right]}\,,\\  
\end{aligned}
\end{equation}

Expanding pairs of factors of the projector sequentially starting from the right leads to the additive form of the projector
\begin{equation}\label{eq:projector4}
    P_4=1+\sum_{\zeta\;\;\text{monic monomials in}\;\;\mathbb{C}[\mathfrak{b}(\mathcal{A})]_4} a_{\zeta}\, \check{e}_{\zeta}
\end{equation}
with
\begin{equation*}
\def\arraystretch{1.8}
    \begin{array}{rclrcl}
    a_{\left[\bb{n}\bb{s}\right]\left[\bb{p}\right]\left[\bb{p}\right]} & = & -\,\dfrac{\delta^2(\delta+4)-4}{(\delta-2)\delta(\delta+2)(\delta+4)}\,, &
    a_{\left[\bb{n}\bb{s}\right]\left[\bb{p}\bb{p}\right]} & = &\dfrac{4}{(\delta-2)\delta(\delta+2)(\delta+4)}\,, \\
    a_{\left[\bb{n}\bb{s}\bb{p}\right]\left[\bb{p}\right]} & = & \dfrac{\delta+3}{(\delta-2)(\delta+2)(\delta+4)}\,, &
    a_{\left[\bb{n}\bb{s}\bb{p}\bb{p}\right]} & = & -\,\dfrac{1}{(\delta-2)(\delta+2)(\delta+4)}\,, \\
    a_{\left[\bb{n}\bb{p}\bb{s}\bb{p}\right]} & = & -\,\dfrac{2}{(\delta-2)\delta(\delta+4)}\,, &
    a_{\left[\bb{n}\bb{s}\right]\left[\bb{n}\bb{s}\right]} & = & \dfrac{\delta(\delta+3)+6}{(\delta-2)(\delta-1)(\delta+2)(\delta+4)}\,, \\
    \multicolumn{6}{c}{a_{\left[\bb{n}\bb{s}\bb{n}\bb{s}\right]}=-\,\dfrac{3\delta+2}{(\delta-2)(\delta-1)(\delta+2)(\delta+4)}\,.}
    \end{array}
\end{equation*}
The expressions of the splitting idempotent for $n=5$ and $n=6$ can be found in the appendix of \cite{KMP_central_idempotents}, and in the Mathematica notebook joined with this article where there is also the expression for $n=7$.

\section*{Acknowledgements}
We are grateful to X. Bekaert and N. Boulanger for valuable comments and advices during preparation of the manuscript. Y.G. and D.B. are much indebted to O.~V. Ogievetsky for enlightening and instructive discussions concerning diagram algebras and their representation theory. The work of Y.G. is supported by a joint grant ``50/50'' UMONS -- Universit\'e Fran\c{c}ois Rabelais de Tours. 

\begin{appendix}
    \section{Examples of traceless projections}\label{app:examples}
\paragraph{Example 7 (arbitrary hook tensors).}
For the partition $\mu=(m,\underbrace{1,...,1}_{n-m})$ with $m\geqslant 2$ one constructs $\mathrm{spec}^{\times}_{\mu}(A_n)$ by reconstructing the skew shape diagrams $\mu\backslash\lambda$ (with $|\mu\backslash\lambda| = 2f$) using the reverse jeu de taquin. The only starting tableaux are
\begin{displaymath}
\ytableausetup{mathmode,boxsize=0.9em,centertableaux}
\begin{ytableau}{\scriptstyle}
      \one &\one & \none[\scriptstyle{\cdots}]& \one & \times &\none[\scriptstyle{\cdots}]& \times \\
      \times\\
    \none[\svdots]\\
      \times\\
      \times\\
\end{ytableau}\qquad \text{(with $2f$ entries `1', for all $f = 1,\dots, \lfloor \tfrac{m}{2}\rfloor$),}
\end{displaymath}
which leads to the following possibilities:
\begin{displaymath}
\ytableausetup{mathmode,boxsize=0.9em,centertableaux}
\mu\backslash\lambda=
\begin{ytableau}{\scriptstyle}
      \times &\times & \none[\scriptstyle{\cdots}]& \times & &\none[\scriptstyle{\cdots}]&\\
      \times\\
    \none[\svdots]\\
      \times\\
      \\
\end{ytableau}
\quad\text{(for all $f = 1,\dots, \lfloor \tfrac{m}{2}\rfloor$)}
\qquad
\text{and}
\qquad
\mu\backslash\lambda=\begin{ytableau}{\scriptstyle}
      \times & \none[\scriptstyle{\cdots}]& \times & &\none[\scriptstyle{\cdots}]&\\
       \times\\
       \times\\
    \none[\svdots]\\
       \times\\
\end{ytableau} 
\quad\text{(for all $f = 1,\dots, \lfloor \tfrac{m-1}{2}\rfloor$)}
\end{displaymath}
\begin{equation}
\text{so,}\quad\mathrm{spec}_{(\mu)}^{\times}(A_n) =\left\{f\,\big(N+2\ \left(m-f\right)\big)-n  \right\} \  \cup \  \left\{f\,\big(N+2\ \left(m-1-f\right)\big) \right\}
\end{equation}
Then, the reduced projector \eqref{eq:projector_reduced}, is given by 
\begin{equation}\label{eq:projector_hook1}
    P_n^{(\mu)} = \displaystyle{\prod_{f=1}^{\lfloor \tfrac{m}{2}\rfloor}} \left(1 - \frac{1}{f(N+2\ (m-f))-n}\,A_n\right)\displaystyle{\prod_{f=1}^{\lfloor \tfrac{m-1}{2}\rfloor}} \left(1 - \frac{1}{f(N+2\ (m-1-f))}\,A_n\right)\,.
\end{equation}

\paragraph{Example 8 (The metric-affine conformal Weyl tensor)} In a Riemannian geometry $(\mathcal{M},g)$ the affine connection $\nabla$ on the tangent bundle of the manifold $\mathcal{M}$ is the Levi-Civita connection associated to the metric $g$. The Riemann tensor is defined with respect the Levi-Civita connection which depends on the metric only and as such we refer to it as the {\it metric Riemann tensor}. From the symmetry of the metric Riemann tensor on can easily conclude that it is a simple $GL$-module $V^{(2,2)}$. More in detail, for $R_{ab,cd} = g_{am} R^{m}{}_{b,cd}$ one has
\begin{equation}
    R_{ab,cd} + R_{ba,cd} = R_{ab,cd} + R_{ab,dc} = 0\,,\quad \text{and}\quad R_{ab,cd} + R_{bc,ad} + R_{ca,bd} = 0\quad \text{(the Bianchi identity)}\,.
\end{equation}
The trace decomposition of the Riemann tensor corresponds to the restriction of the representation $V^{(2,2)}$ to the (local) orthogonal group. For the spacetime dimensions $N\geqslant 4$, according to the Littlewood restriction rules \cite{Littlewood}, this leads to the metric Weyl tensor (the totally traceless projection $D^{(2,2)}$), the traceless Ricci tensor (single-trace projection $D^{(2)}$) and the scalar curvature (the double trace $D^{(0)}$), recall \eqref{eq:branching}. The expression for the metric Weyl tensor in the above decomposition is very well known \cite{weyl1918reine} (see also \cite[Chapter II]{eisenhart1997riemannian}). For the lower dimensions $N = 2,3$ the metric Weyl tensor vanishes identically (as the module $D^{(2,2)}$ does \cite{Weyl}, recall Theorem \ref{thm:reduced_projector}). Besides, the Littewood-Richardson restriction rule \cite{Littlewood} does not apply in this case, and the branching rules are no more expressed in terms of the Littlewood-Richardson coefficients (recall \eqref{eq:branching}). For $N = 3$ both the Ricci tensor and scalar curvature are present ({\it i.e.} the Riemann tensor can be expressed in terms of the Ricci tensor and the scalar curvature). For $N = 2$ only the scalar curvature enters the decomposition, despite $D^{(2)}$ is a non-trivial module in this case (this is due to the fact that locally, any metric in two dimensions is conformally flat).
\vskip 4 pt

In a metric-affine geometry $(\mathcal{M},\bar{\nabla},g)$ \cite{eisenhart1929non} the tangent bundle of the manifold $\mathcal{M}$ is endowed with a supplementary affine connection $\bar{\nabla}$. In full generality this connection does not preserve the metric under parallel transport $\bar{\nabla} g \neq 0 $ and has torsion: $\bar{\nabla}_X Y-\bar{\nabla}_Y X -[X,Y] \neq 0$ for two vector fields $X$ and $Y$. The non-metric Riemann tensor associated to  $\bar{\nabla}$ enjoys the only symmetry:
\begin{equation}
\bar{R}_{a,b,cd} + \bar{R}_{a,b,dc} = 0\,\quad  \text{where} \quad \bar{R}_{a,b,cd} = g_{am} \bar{R}^{m}{}_{b,cd}\;.
\end{equation}
Thus, the $GL$-irreducible components are given by the following plethysm, which is computed by applying the Littlewood-Richardson rule \eqref{eq:LRRule}:
\begin{equation}\label{eq:riemannGL}
\Yboxdim{9pt}\yng(1)\,\LRp\,\yng(1)\,\LRp\yng(1,1)\,=\yng(3,1)\,+\,\yng(2,2)\,+\,2\; \yng(2,1,1)\,+\,\yng(1,1,1,1) \;,\quad \text{where one assumes $N\geqslant  4$.}
\end{equation}
Upon restriction to $O(N)$, the right-hand-side of \eqref{eq:riemannGL} contains a significantly larger variety of components than in the metric case. Along with the lines of the present work, let us focus on the traceless part. It is natural to define the non-metric Weyl tensor as the totally traceless projection (with respect to the metric $g$) of the non-metric Riemann tensor:
\begin{equation}\label{eq:non-metric_Weyl}
    \bar{W}_{a,b,cd} + \bar{W}_{a,b,dc} = 0\,,\quad g^{ab}\bar{W}_{a,b,cd} = g^{ab}\bar{W}_{a,c,bd} = g^{ab}\bar{W}_{c,a,bd} = 0\,.
\end{equation}
Its explicit expression can be obtained by applying the projector $\mathfrak{P}_4$ whose step-by-step construction is described in Section \ref{sec:summary}. For the lower dimensions $N = 2,3$ one ignores the components whose traceless part vanishes identically. Namely, when $N=2$ none of the Young diagrams on the right-hand-side of \eqref{eq:riemannGL} has a non-trivial traceless projection (recall Theorem \ref{thm:reduced_projector}), so the non-metric conformal Weyl tensor vanishes identically, which is reminiscent of the metric case. For $N = 3$, the right-hand-side of \eqref{eq:riemannGL} (and thus the non-metric Riemann tensor) contains one $GL(3)$-irreducible component $(3,1)$ with a non-zero traceless projection. The latter can be constructed via Corollary \ref{cor:projector_reduced_sum}: for the index set $\mathcal{R} = \big\{(3,1),(2,2),(2,1^2)\big\}$ one constructs $\mathrm{spec}^{\times}_{(3,1)}(A_4) = \{3,5\}$, $\mathrm{spec}^{\times}_{(2,2)}(A_4) = \{1,4\}$, $\mathrm{spec}^{\times}_{(2,1^2)}(A_4) = \{1\}$, so $\mathrm{spec}^{\times}_{\mathcal{R}}(A_4) = \{1,3,4,5\}$ and one acts by the following operator on the components $\bar{R}_{a,b,cd}$:
\begin{equation}
    P^{(\mathcal{R})}_4 = \big(1 - A_4\big)\big(1 - \tfrac{1}{3}\,A_4\big)\big(1 - \tfrac{1}{4}\,A_4\big)\big(1 - \tfrac{1}{5}\,A_4\big)\,.
\end{equation}
Choosing the semi-additive form of the projector described in Corollary \ref{cor:projector_semi-sum} will yield the same result in a much more economical way: one uses the operator $P_4^{(3,1)}z^{(3,1)}$, where
\begin{align*}\label{eq:projector_hook}
    P_4^{(3,1)} &= \big(1 - \tfrac{1}{3}\,A_4\big)\big(1 - \tfrac{1}{5}\,A_4\big)\\
    &=\check{e}_{\left[\bb{p}\right]\left[\bb{p}\right]\left[\bb{p}\right]\left[\bb{p}\right]}+\dfrac{2}{15}\, \check{e}_{\left[\bb{n}\bb{s}\right]\left[\bb{n}\bb{s}\right]} + \dfrac{1}{15}\, \check{e}_{\left[\bb{n}\bb{s}\bb{p}\right]\left[\bb{p}\right]} - 
 \dfrac{1}{3}\, \check{e}_{\left[\bb{n}\bb{s}\right]\left[\bb{p}\right]\left[\bb{p}\right]}.
\end{align*}
\vskip 4 pt

In \cite[Section 4]{Alvarez:2016} the authors choose to define the non-metric Weyl tensor by utilising the non-metric Riemann tensor (as well as the corresponding symmetric Ricci tensor and scalar curvature) in the standard expression of the metric Weyl tensor. The resulting tensor is however not traceless. To justify this definition it is asserted that no modification of the expression of the metric Weyl tensor satisfying the properties \eqref{eq:non-metric_Weyl} exists. However, our analysis suggests that the sought modification does exist and is obtained via application of the traceless projector $\mathfrak{P}^{(\mathcal{R})}_{4}$ to the non-metric Riemann tensor. In particular, the non-metric Weyl tensor \eqref{eq:non-metric_Weyl} is non-zero whenever $N\geqslant 3$.

\section{Proofs}

\subsection{Proof of Lemma \ref{lem:KerA} continued}\label{app:KerA_proof_Sp}
\begin{proof}
    Let us consider the case of an antisymmetric metric on $V$. With the definition of an adjoint operator \eqref{eq:scalar_product_tensors_trace} at hand, one can check that $\mathfrak{r}(d_{ij})^{*} = \mathfrak{r}(d_{ij})$. To analyse the diagonal structure of the latter we make use of the canonical isomorphism $V \cong U^{*}\oplus U$, such that $ \langle \varphi + u , \psi + v\rangle = \varphi(v) - \psi(u)$. This induces block structure on the matrices of $\langle \,\cdot,\cdot\,\rangle$ and any $F\in\mathrm{End}(V)$:
    \begin{equation}
    \def\arraystretch{1.2}
        \langle \,\cdot,\cdot\,\rangle \sim
        \left(\begin{array}{ccc}
        0 & \vline & \mathbb{1} \\ \hline
        -\mathbb{1} & \vline & 0
    \end{array}\right)\,, \quad F \sim 
        \left(\begin{array}{ccc}
        A & \vline & B \\ \hline
        C & \vline & D
    \end{array}\right) \quad\quad \Rightarrow\quad 
        \quad F^{*} \sim 
        \left(\begin{array}{ccc}
        D^{t} & \vline & -B^{t}\\ \hline
        -C^{t} &\vline &  A^{t}
    \end{array}\right)\,,
    \end{equation}
    where $(\cdot)^{t}$ denotes the transpose of a matrix. The same structure holds when one considers the metric and operators on $V^{\otimes n}$, with the only difference that depending on $n$ the metric is either symmetric or anti-symmetric: for any $S,T\in V^{\otimes n}$ one has $\langle S,T \rangle = (-)^{n} \langle T,S \rangle$. Set $r = n\,(\mathrm{mod}\,2)$ and consider the space
    \begin{equation}\label{eq:tensors_even-odd}
        \mathcal{T}_{n}(V) = V^{\otimes n} \oplus V^{\otimes n-2} \oplus\dots\oplus V^{\otimes r}\quad \text{(by definition, $V^{\otimes 0} = \mathbb{C}$),}
    \end{equation}
    The generalisation of the above block structure for $\langle\,\cdot,\cdot\,\rangle$ and $F\in\mathrm{End}\big(\mathcal{T}_{n}(V)\big)$ reads
    \begin{equation}
    \def\arraystretch{1.2}
        \langle \,\cdot,\cdot\,\rangle \sim
        \left(\begin{array}{ccc}
        0 & \vline & \mathbb{1} \\ \hline
        (-)^{r}\mathbb{1} & \vline & 0
    \end{array}\right)\,, \quad F \sim 
        \left(\begin{array}{ccc}
        A & \vline & B \\ \hline
        C & \vline & D
    \end{array}\right) \quad\quad \Rightarrow\quad 
        \quad F^{*} \sim 
        \left(\begin{array}{ccc}
        D^{t} & \vline & (-)^{r}B^{t}\\ \hline
        (-)^{r} C^{t} &\vline &  A^{t}
    \end{array}\right)
    \end{equation}
    
    Fix a pair $i<j$. The space $\mathcal{T}_{n}(V)$ is preserved by $\mathrm{tr}^{(g)}_{ij}$, so in the sequel we identify $\mathrm{tr}^{(g)}_{ij}$ with its restriction to $\mathcal{T}_{n}(V)$. This allows us to have $\mathrm{tr}^{(g)*}_{ij}$ on $\mathcal{T}_{n}(V)$, such that $\mathfrak{r}(d_{ij}) = - \mathrm{tr}^{(g)*}_{ij}\mathrm{tr}^{(g)}_{ij}$. Upon a suitable reordering of the basis elements in $\mathcal{T}_{n}(V)$ one has the following block structure of $\mathrm{tr}^{(g)}_{ij}$:
    \begin{equation}
    \def\arraystretch{1.2}
        \mathrm{tr}^{(g)}_{ij} \sim \left(\begin{array}{ccc}
        0 & \vline & B \\ \hline
        (-)^{r} B & \vline & 0
    \end{array}\right)
    \quad\quad\Rightarrow\quad\quad
    \mathrm{tr}^{(g)*}_{ij} \sim \left(\begin{array}{ccc}
        0 & \vline & (-)^{r} B^{t} \\ \hline
        B^{t} & \vline & 0
    \end{array}\right)\,,\quad\quad \mathrm{tr}^{(g)*}_{ij}\mathrm{tr}^{(g)}_{ij} \sim \left(\begin{array}{ccc}
        B^{t}B & \vline & 0 \\ \hline
        0 & \vline & B^{t}B
    \end{array}\right)\,,
    \end{equation}
    where $B^{t}B$ is block-diagonal with respect to the decomposition of $\mathcal{T}_n(V)$ into direct summands. The transformation $B^{t}B$ is diagonalisable by an orthogonal matrix $S$ ({\it i.e.} such that $S^{t} = S^{-1}$). Therefore, $B^{t}B$ is also diagonalisable by the contragredient transformation $(S^{-1})^{t} = S$. As a result, both non-zero blocks of $\mathfrak{r}(d_{ij}) = - \mathrm{tr}^{(g)*}_{ij}\mathrm{tr}^{(g)}_{ij}$ (one being a transformation of $U$, and the other of $U^{*}$) are brought to the same diagonal form, all the eigenvalues being non-positive. The same conclusion is easily extended to the action of $A_n$ on $V^{\otimes n}$.
\end{proof}

\subsection{Alternative proof of Lemma \ref{lem:KerA} in the semisimple regime of $B_{n}(\varepsilon N)$}\label{app:KerA_proof_semisimple}
\begin{proof}
    Recall that the semisimple regime for $B_{n}(\varepsilon N)$ occurs for $N\geqslant n - 1$ when $\varepsilon = 1$ and $N\geqslant 2(n - 1)$ when $\varepsilon = -1$. Let us analyse the set in \eqref{eq:spec_master_class} without intersecting it with $\varepsilon \mathbb{N}$ and show that it nevertheless coincides with \eqref{eq:spec_master_class}. Denote $\alpha_{\mu\backslash\lambda}$ the eigenvalue from Lemma \ref{lem:D_block_diagonal} which corresponds to $L^{(\mu)} \subset M^{(\lambda)}_n$.
    \vskip 4 pt
    
    For $\varepsilon = 1$ let us concentrate on the minimal possible value of the content $c(\mu\backslash\lambda)$. First, note that because $\lambda^{\prime}_1 \leqslant |\lambda| = n - 2f$, in the semisimple regime $\lambda^{\prime}_1 + f \leqslant (n - 2f) + f \leqslant N$ when $f \geqslant 1$. So the minimal possible content among skew-shapes $\mu\backslash\lambda \in \mathrm{cl}^{(f)}_N(\lambda)$ is achieved when $\mu$ is obtained via adding two columns of $f$ boxes to the first and the second columns of $\lambda$. In this case $c(\mu\backslash\lambda) = -(\lambda^{\prime}_1 + \lambda^{\prime}_2 + f-2)f$, and the corresponding eigenvalue $\alpha_{\mu\backslash\lambda} = (N + 1 - f - \lambda^{\prime}_1 - \lambda^{\prime}_2)f$. Assuming that $\alpha_{\mu\backslash\lambda} \leqslant 0$ one gets $\lambda^{\prime}_1 + \lambda^{\prime}_2 + 2f\geqslant N+1+f$, and on the other hand $\lambda^{\prime}_1 + \lambda^{\prime}_2 + 2f \leqslant n$. By combining the two estimates one gets $N \leqslant n - 1 - f \leqslant n-2$, which is in contradiction with the semisimplicity of $B_{n}(N)$.
    \vskip 4 pt
    
    For $\varepsilon = -1$ we look for the maximal possible value of the content $c(\mu\backslash\lambda)$. Note that $\lambda_1 \leqslant |\lambda| = n - 2f$, so in the semisimple regime $\lambda_1 + 2f \leqslant \frac{N}{2} + 1 \leqslant N$ when $f \geqslant 1$ and $N\in 2\mathbb{N}$. So the maximal possible content among skew-shapes $\mu\backslash\lambda \in \mathrm{cl}^{(f)}_N(\lambda)$ is achieved when $\mu$ is obtained by adding a row of $2f$ boxes to the first row of $\lambda$. In this case $c(\mu\backslash\lambda) = (2\lambda_1 + 2f - 1)f$, and the corresponding eigenvalue $\alpha_{\mu\backslash\lambda} = (N + 1 - f - \lambda^{\prime}_1 - \lambda^{\prime}_2)f$. Assuming that $\alpha_{\mu\backslash\lambda} \geqslant 0$ one gets $\frac{N}{2} \leqslant \lambda_1 + f - 1$, and on the other hand $\lambda_1 + 2f \leqslant n$. By combining the two estimates one gets $\frac{N}{2} \leqslant n - 1 - f \leqslant n-2$, which is in contradiction with the semisimplicity of $B_{n}(-N)$.
    \vskip 4 pt

    To this end, we conclude that the only possibility to have $\alpha_{\mu\backslash\lambda} = 0$ occurs for $f = 0$, so $\lambda = \mu$. This corresponds to $D^{(\hat{\lambda})}\otimes M^{(\lambda)}_n$ in the decomposition \eqref{eq:Schur-Weyl} with $|\lambda| = n$, which are annihilated by all $d_{ij}$. So $\mathrm{Ker}\,A_n$ is indeed the subspace of traceless tensors.
\end{proof}

\subsection{Proof of Corollary \ref{cor:KerA_Br}}\label{app:proof_KerA_Br}
According to the conditions in the assertion, Schur-Weyl duality applies, so for $L^{(\mu)}\subset M^{(\lambda)}_n$ the first part of the assertion holds by Lemmas \ref{lem:KerA} and \ref{lem:D_block_diagonal}. For the second part consider $L^{(\mu)} \subset \Delta^{(\lambda)}_n$ upon restriction to $\mathbb{C}\Sn{n}$, such that
\begin{equation}\label{eq:A_zero_in_Delta}
    \alpha_{\mu\backslash\lambda} = (\delta - 1) \,\frac{|\mu| - |\lambda|}{2} + c(\mu\backslash\lambda)\,.
\end{equation}
The above expression can be interpreted as $A_n L^{(\mu)} = 0$ along the following lines. Choose a basis in $\Delta^{(\lambda)}_n$ such that any element of $B_{n}(\delta)$ is represented by a block-upper-triangular matrix in $\mathrm{End}\big(\Delta^{(\lambda)}_n\big)$, with the composition factors $M^{(\sigma_i)}_n$ (with $1\leqslant i\leqslant r$) appearing on the diagonal (see \cite[paragraph below Corollary 2.3]{CDVM_Br_blocks}). There is exactly one matrix block with $\sigma_1 = \lambda$, while for the others (when $i\neq 1$) holds $|\sigma_i| > |\lambda|$. By \cite[Proposition 4.2]{CDVM_Br_blocks}, one has 
\begin{equation}\label{eq:embedding_condition}
    (\delta - 1) \,\frac{|\sigma_i| - |\lambda|}{2} + c(\sigma_i\backslash\lambda) = 0\,.
\end{equation}
The embedding $L^{(\mu)} \subset \Delta^{(\lambda)}_n$ is realised via combinations of vectors from modules $L^{(\mu)}\subset M^{(\sigma_{i_p})}_n$ (for a subset of indicies $1 \leqslant i_1 < \dots < i_{\ell}\leqslant r$). By Lemma \eqref{lem:D_block_diagonal}, the eigenvalue of $A_n$ on each embedding is
\begin{equation}
    (\delta - 1) \frac{|\mu| - |\sigma_i|}{2} + c(\mu\backslash\sigma_i) = (\delta - 1) \frac{|\mu| - |\lambda|}{2} + c(\mu\backslash\lambda) = \alpha_{\mu\backslash\lambda}
\end{equation}
(where we have used \eqref{eq:embedding_condition}). Therefore $A_n$ is proportional to identity on any embedding $L^{(\mu)}\subset \Delta^{(\lambda)}_n$ with the same eigenvalue $\alpha_{\mu\backslash\lambda}$ \eqref{eq:A_zero_in_Delta}. Due to Lemma \ref{lem:KerA}, $L^{(\mu)}$ can not appear in $M^{(\lambda)}_n$ if $\alpha_{\mu\backslash\lambda} < 0$ or $\alpha_{\mu\backslash\lambda} = 0$ with $\lambda \neq \mu$, and thus any its occurrence in $\Delta^{(\lambda)}_n$ is factored out due to $M^{(\lambda)} \cong \Delta^{(\lambda)}_n\slash K$.

\subsection{Proof of Lemma \ref{lem:trace_rules}}\label{app:proof_lem_rules}
    \begin{proof}
    The diversity of expressions on the left-hand-sides of the rules \eqref{eq:trace_rule_1}-\eqref{eq:trace_rule_4} which are necessary to fix $\tau$, is essentially limited by $\mathfrak{b}(\mathcal{A})$-linearity. Representatives of a particular form on the left-hand-sides of the rules \eqref{eq:trace_rule_1}, \eqref{eq:trace_rule_2} are always accessible via cyclic permutations. Taking into account the inversion, the same conclusion is valid for the rules in \eqref{eq:trace_rule_3}, \eqref{eq:trace_rule_4} due to the following simple facts.
        \paragraph{Fact 1.} For $[\db{p}u] = [\db{p}I(u)]$ either $u$ or $I(u)$ is fit because $[\bb{p}u]\in \mathfrak{b}(\mathcal{A})$.
        \paragraph{Fact 2.} For $[\db{p}u\db{s}v] = [\db{p}I(v)\db{s}I(u)]$ either $u$ or $I(v)$ is fit. Indeed, since $[\bb{p}u\bb{s}v]\in \mathfrak{b}(\mathcal{A})$, either one of the bracelets $[u]$, $[v]$ is in $\mathfrak{b}(\mathcal{A})$ or one of them is empty.
        \paragraph{Fact 3.} For $[\db{p}u\db{s}v] = [\db{p}I(v)\db{s}I(u)]$ either $[u],[v]\notin \mathfrak{b}(\mathcal{A})$ or $[u],[v]\in \mathfrak{b}(\mathcal{A})$. In the former case either $|u|_{\bb{n}} > |u|_{\bb{s}}$ or $|v|_{\bb{n}} > |v|_{\bb{s}}$, while in the latter $u,v$ or $I(u), I(v)$ are both fit.
    Note that the admissible representatives on the left-hand-sides of the rules \eqref{eq:trace_rule_3} are fixed unambiguously. 
    \vskip 4 pt
    
    Representatives on the left-hand-sides of \eqref{eq:trace_rule_1}-\eqref{eq:trace_rule_4} are sufficient to define $\tau$. Let us check that the definition is independent of a particular choice among admissible representatives.
    \begin{itemize}
        \item[{\it 1)}] For the rule \eqref{eq:trace_rule_1}, the only alternative representative is $[\dd{s} I(u)]$. Application of \eqref{eq:trace_rule_1} leads to $\tau\big([\dd{s} I(u)]\big) = 2\delta\,[\bb{s} I(u)]$. But the latter equals to $\tau\big([\dd{s} u]\big)$ by inversion of the representative combined with a cyclic permutation.
        
        \item[{\it 2)}] For the first rule in \eqref{eq:trace_rule_2}, the equivalent representatives are $[\db{s}v\db{s}u] = [\db{s}I(u)\db{s}I(v)] = [\db{s}I(v)\db{s}I(u)]$. By direct application of \eqref{eq:trace_rule_2} one calculates $\tau\big([\db{s}v\db{s}u]\big) = 2\big([\db{s}v\db{s}I(u)] + [\db{s}v][\db{s}u]\big)$. But $[\db{s}v\db{s}I(u)] = [\db{s}u\db{s}I(v)]$ (the two representatives are related by inversion followed by a cyclic permutation), which reproduces $\tau\big([\db{s}u\db{s}v]\big)$. Other alternatives are analysed along the same lines.
        
        For the second rule in \eqref{eq:trace_rule_2}, one has to analyse $[\db{s}v][\db{s}u]$, together with the alternatives $[\db{s}I(u)]$ and $[\db{s}I(v)]$ for each factor. By direct application of \eqref{eq:trace_rule_2} one gets $\tau\big([\db{s}v][\db{s}u]\big) = 2\big([\db{s}v\db{s}u] + [\db{s}v\db{s}I(u)]\big)$. But $[\db{s}v\db{s}u] = [\db{s}u\db{s}v]$ (by a cyclic permutation) and $[\db{s}v\db{s}I(u)] = [\db{s}u\db{s}I(v)]$ (by composition of inversion and a cyclic permutation), which leads to $\tau\big([\db{s}u][\db{s}v]\big)$. Consider also $\tau\big([\db{s}u][\db{s}I(v)]\big) = 2\big([\db{s}u\db{s}I(v)] + [\db{s}u\db{s}v]\big)$ which equals $\tau\big([\db{s}u][\db{s}v]\big)$ directly. Other alternatives are analysed along the same lines.
        
        \item[{\it 3)}] For the first rule in \eqref{eq:trace_rule_4}, let $u,v$ are fit. Then for the alternative representative one has $\tau\big([\db{p}v\db{p}u]\big) = [\bb{n}v\bb{s}I(u)] = [\bb{n}u\bb{s}I(v)]$. Else, let $|u|_{\bb{s}} > |u|_{\bb{n}}$. Then one checks $\tau\big([\db{p}I(u)\db{p}I(v)]\big) = [\bb{n}I(u)][\bb{s}I(v)] = [\bb{n}u][\bb{s}v]$.
        
        For the second rule in \eqref{eq:trace_rule_4} one considers $\tau\big([\db{p}v][\db{p}u]\big) = [\bb{n}v\bb{s}I(u)] = [\bb{n}u\bb{s}I(v)]$.
    \end{itemize}
\end{proof}

\subsection{Proof of Theorem \ref{thm:Laplace}}\label{app:proof_Laplace}
\begin{proof}
The main idea consists in analysing the following formula which follows from the definition of the average \eqref{eq:average}: for any diagram $b\in B_{n}(\delta)$
\begin{equation}
    A_n\gamma_b = \gamma_{(A_n b)}\,.
\end{equation}
The product $A_n b = \sum_{1\leqslant i < j\leqslant n} d_{ij} b$ consists in summing over all possibilities of placing an arc at a pair of lower nodes of $b$, which is mimicked by imposing Leibniz rule in the definition of $\partial$. Without loss of generality, we assume a convenient representative $b$ in the conjugacy class, such that all bracelets in $\Phi(\gamma_b)$ come from mutually non-intersecting cycles, and such that each independent cycle is either a cycle permutation, or, if $b\in J^{(1)}$, there are no intersections among the vertical lines, while each arc in the upper or lower row is incident to a pair of nodes $(i,i+1)$ (for $i \leqslant n-1$) or $(1,n)$ (for example, see \eqref{eq:convenient_representative}). The independent cycles will be schematically illustrated by rectangular blocks, with only particular significant lines specified explicitly. For example, 
    \begin{equation}
    \raisebox{-.4\height}{\includegraphics[scale=0.4]{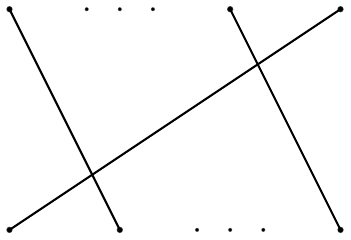}}\rightarrow \raisebox{-.4\height}{\includegraphics[scale=0.4]{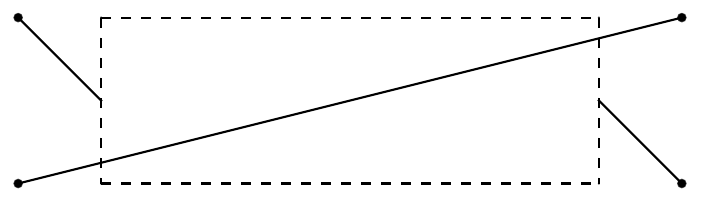}} \;, \hspace{2cm} \raisebox{-.4\height}{\includegraphics[scale=0.4]{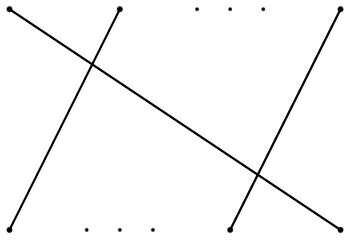}}\rightarrow \raisebox{-.4\height}{\includegraphics[scale=0.4]{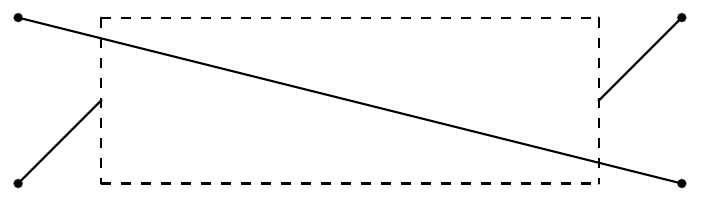}}\;,
    \end{equation}
        \begin{equation}
    \raisebox{-.4\height}{\includegraphics[scale=0.4]{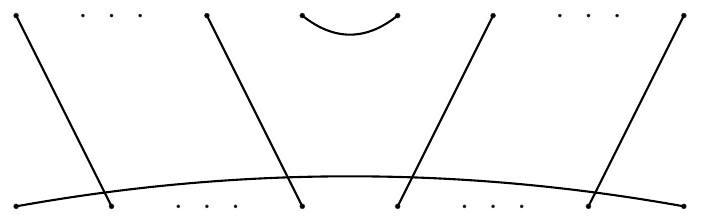}}\rightarrow \raisebox{-.4\height}{\includegraphics[scale=0.4]{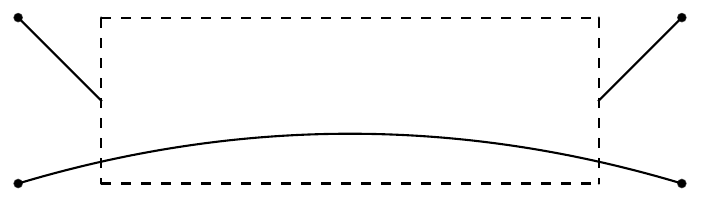}}\;, \hspace{1cm}
    \raisebox{-.4\height}{\includegraphics[scale=0.4]{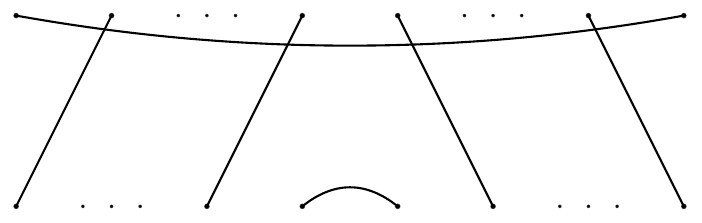}}\rightarrow \raisebox{-.4\height}{\includegraphics[scale=0.4]{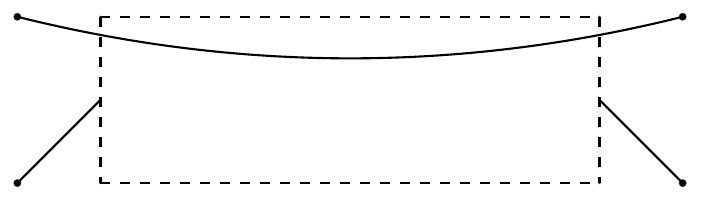}}
    \end{equation}
There are two major cases how one can place the arc upon multiplication by $d_{ij}$: either it  joins two nodes within a single cycle, or the two nodes belong to two different cycles. Only the bracelets coming from these blocks will be affected by multiplication by $d_{ij}$, so all other blocks can be ignored while analysing each particular $i<j$. This property is manifested in the definition of $\tau$ as a $\mathbb{C}[\mathfrak{b}(\mathcal{A})]$-linear operation. We will work in terms of particular representatives in the bracelets, and read off the resulting representatives coming form the initial ones. To fix a representative, we will highlight the starting point by a circle around a node, such that one starts reading along the adjacent line (if there is a diagram attached below, one ignores it). We always imply the word $u$ read off by following the lines hidden behind the first block, while for the second block (if any) the word is $v$. Additional arrows entering each block fix the direction of reading which leads to the corresponding word. Passing the block in the opposite direction leads to $I(u)$ instead of $u$ and $I(v)$ instead of $v$. For example, from the following schema one reads off the representative $[\bb{n}u\bb{s}I(v)]$:
\begin{equation}
\raisebox{-.4\height}{\includegraphics[scale=0.7]{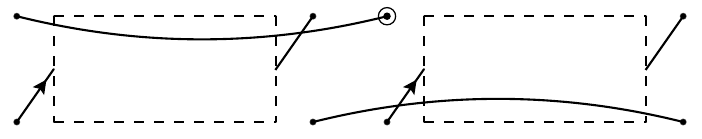}}
\end{equation}

Each letter in a word parametrising a bracelet corresponds to a line with two endpoints. Upon multiplication $d_{ij} b$, put $0$, $1$ or $2$ dots above a letter in each representative depending on how many endpoints are occupied by the $(i,j)$-arc attached to the lower row of $b$. It is clear that $\bb{n}$ can never acquire a dot, $\bb{p}$ can acquire at most $1$ dot, while $\bb{s}$ can acquire up to $2$ dots. This is exactly encoded in the definition of $\partial$ by putting $\partial(\bb{n}) = \partial(\db{p}) = \partial(\dd{s}) = 0$.  Without loss of generality, one can consider the following schemas for the possibilities of attaching the arc.
\begin{itemize}
    \item[{\it 1)}] The arc is attached to another arc, which leads to a cycle:
    \begin{equation}
\begin{aligned}
\tau_n\big(\left[\dd{s}u\right]\big) \;\;=\;\; \hspace{0.1cm}\raisebox{-.4\height}{\includegraphics[scale=0.4]{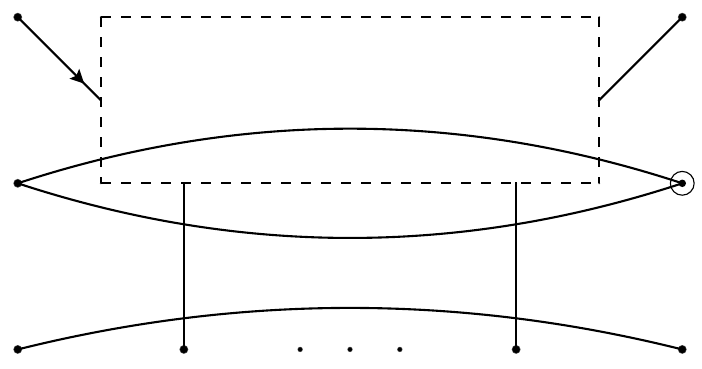}}\;\;=\;\;\delta\;  \raisebox{-.4\height}{\includegraphics[scale=0.4]{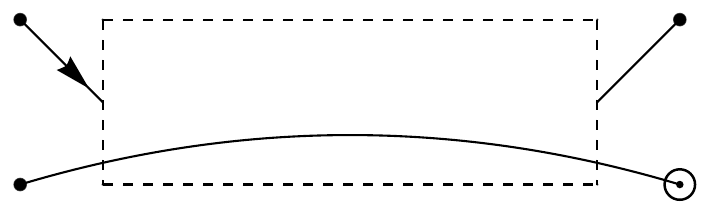}} \;\;=\;\; \delta\, [\bb{s}u]\,,
\end{aligned}    
\end{equation}
so one arrives at the rule \eqref{eq:trace_rule_1}.

\item[{\it 2)}] The arc is attached to one of the endpoints of two particular arcs of $b$. One sums over the four possibilities in this case:
\begin{equation}
\def\arraystretch{2}
\begin{array}{ll}
     \tau\big(\left[\db{s}u\db{s}v\right]\big) & =\;\;\hspace{0.1cm}\raisebox{-.4\height}{\includegraphics[scale=0.6]{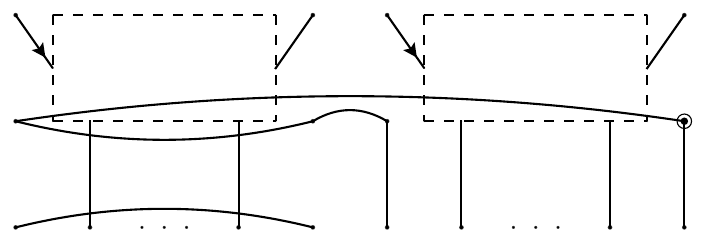}}\;+\;\raisebox{-.4\height}{\includegraphics[scale=0.6]{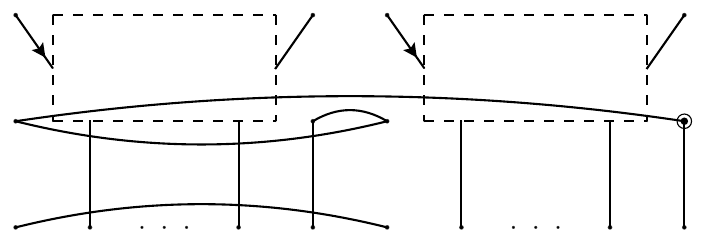}}\;+ \\
    \hfill &\,\hspace{0.6cm}\raisebox{-.4\height}{\includegraphics[scale=0.6]{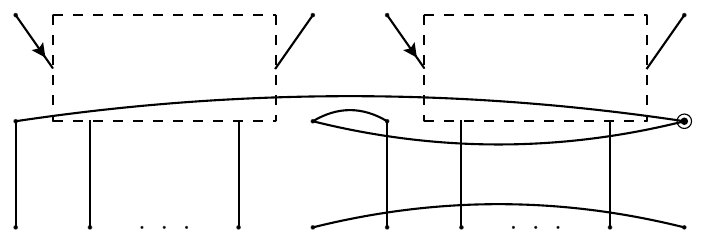}}\;+\;\raisebox{-.4\height}{\includegraphics[scale=0.6]{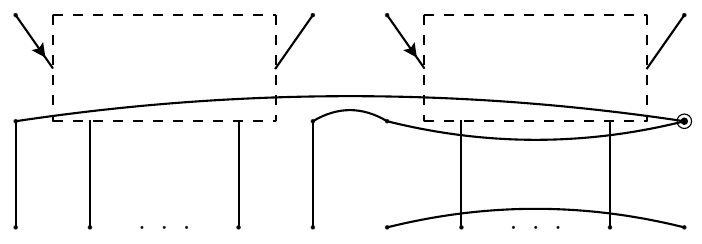}}\\
    \multicolumn{2}{l}{ = \;2\,\raisebox{-.4\height}{\includegraphics[scale=0.6]{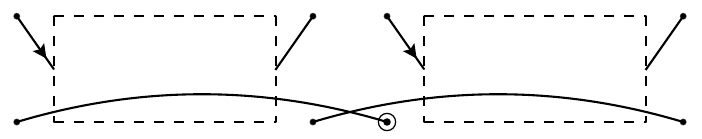}}\; + \;2\,\raisebox{-.4\height}{\includegraphics[scale=0.6]{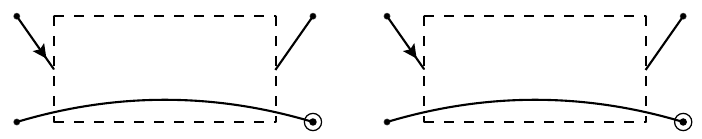}}\;=\;2\,[\bb{s}u\bb{s}I(v)] + 2\,[\bb{s}u][\bb{s}v]\,,}
\end{array}    
\end{equation}
so one reproduces the first rule in \eqref{eq:trace_rule_2}.

\item[{\it 3)}] Next, let us consider the case when one endpoint of the arc is attached to a passing line, while the other one occupies one of the endpoints of a particular arc in $b$. Then one sums over the two possibilities, where the structure of the representatives assumes that $u$ is fit
\begin{equation}
\begin{array}{ll}
    \tau\big(\left[\db{p}u\db{s}v\right]\big) & =\hspace{0.1cm} \raisebox{-.4\height}{\includegraphics[scale=0.6]{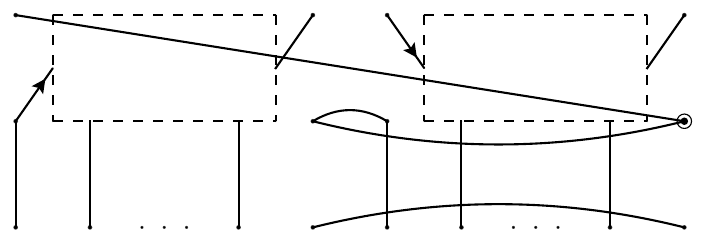}}  \;+\; \raisebox{-.4\height}{\includegraphics[scale=0.6]{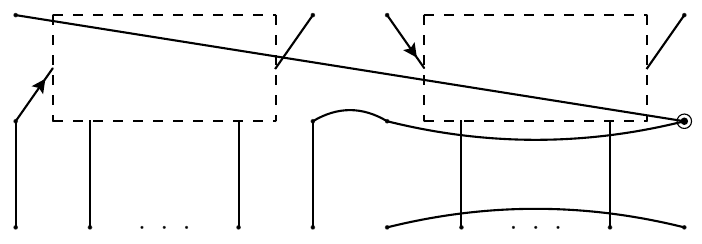}}\\
    \multicolumn{2}{l}{= \raisebox{-.4\height}{\includegraphics[scale=0.6]{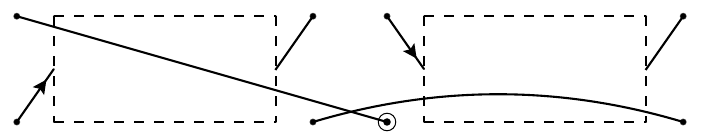}}\;+\;\raisebox{-.4\height}{\includegraphics[scale=0.6]{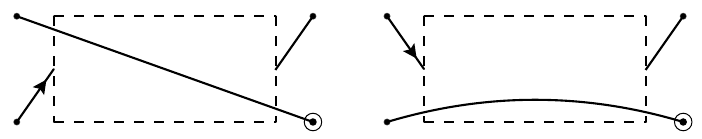}} = [\bb{p}u\bb{s}I(v)]\;+\;[\bb{p}u][\bb{s}v]\,,}
\end{array}    
\end{equation}
so one recovers the first rule in \eqref{eq:trace_rule_3}.
\item[{\it 4)}] Finally, the arc can occupy the lower endpoints of two passing lines. First, suppose that $u$ is fit (so is $v$), which leads to
\begin{equation}
\begin{aligned}
    \tau\big(\left[\db{p}u\db{p}v\right]\big) & =\hspace{0.1cm}\raisebox{-.4\height}{\includegraphics[scale=0.6]{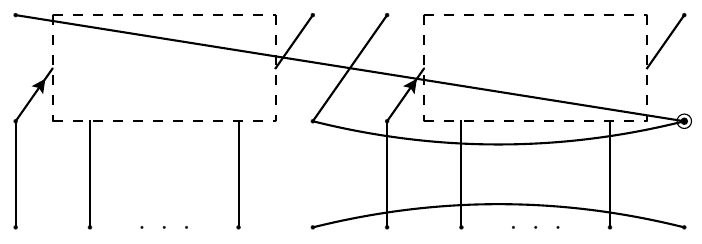}}\;=\;\hspace{0.1cm}\,\raisebox{-.4\height}{\includegraphics[scale=0.6]{theo3_f_r1fit.pdf}}\;=\; [\bb{n}u\bb{s}I(v)]\,,
\end{aligned}    
\end{equation}
in agreement with the first rule in \eqref{eq:trace_rule_4} for the case of a fit representative. In the other case, when neither $u$ nor $v$ is not fit, with $|u|_{\bb{s}} > |u|_{\bb{n}}$, one has
\begin{equation}
\begin{aligned}
    \tau\big(\left[\db{p}u\db{p}v\right]\big) & =\;\hspace{0.1cm}\raisebox{-.4\height}{\includegraphics[scale=0.6]{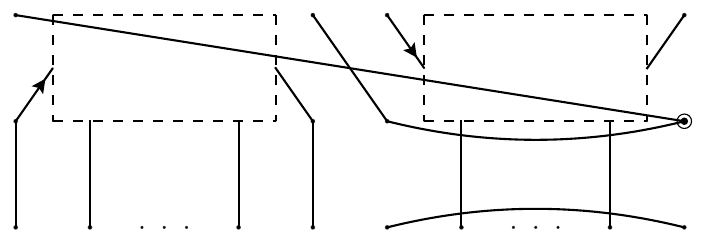}}\;=\;\hspace{0.1cm}\,\raisebox{-.4\height}{\includegraphics[scale=0.6]{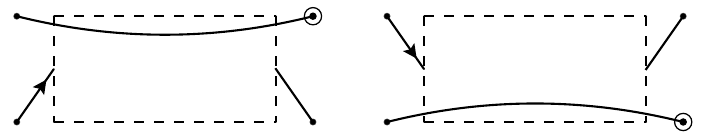}}\;=\; [\bb{n}u][\bb{s}v]\,,
\end{aligned}    
\end{equation}
which is again in agreement with the first rule in \eqref{eq:trace_rule_4}. 
\end{itemize}
We are left with the cases when the arc connects two independent cycles.
\begin{itemize}
    \item[{\it 5)}] If the arc occupies the endpoints of two particular arcs in $b$, one sums over four possibilities
    \begin{equation}
    \def\arraystretch{2}
    \begin{array}{ll}
    \tau\big(\left[\db{s}u\right]\left[\db{s}v\right]\big) = & \;\raisebox{-.4\height}{\includegraphics[scale=0.6]{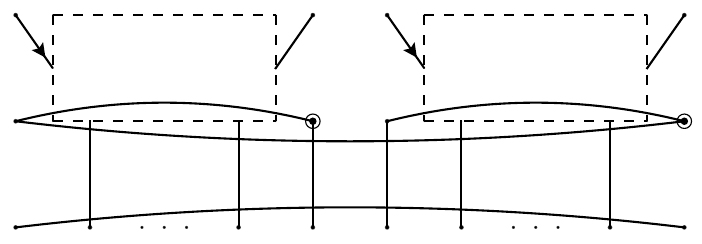}}\;+\;\raisebox{-.4\height}{\includegraphics[scale=0.6]{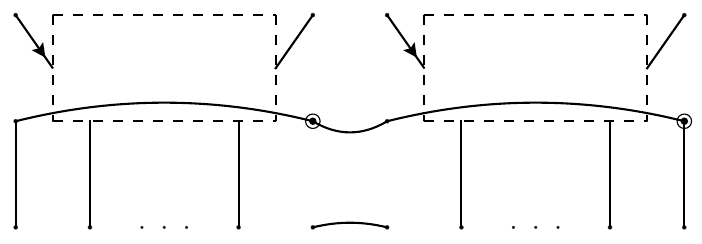}}\\
    \hfill & + \;\raisebox{-.4\height}{\includegraphics[scale=0.6]{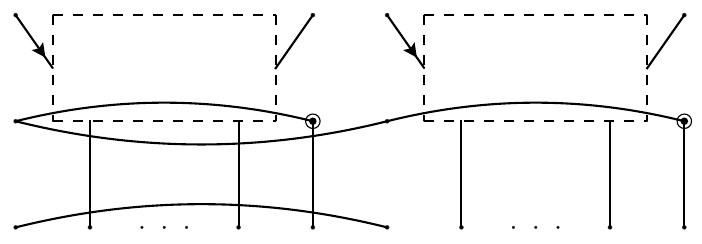}}\;+\;\raisebox{-.4\height}{\includegraphics[scale=0.6]{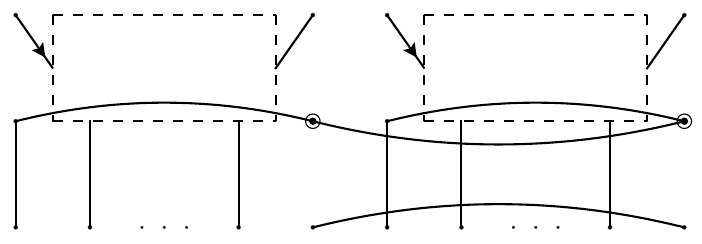}} \\
    \multicolumn{2}{l}{=\hspace{0.1cm}\,2\,\raisebox{-.4\height}{\includegraphics[scale=0.6]{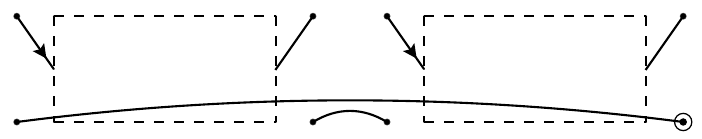}}\,+\,2\,\raisebox{-.4\height}{\includegraphics[scale=0.6]{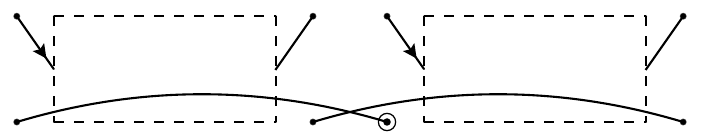}}\;=\; 2\, [\bb{s}u\bb{s}v]\;+\;2\, [\bb{s}u\bb{s}I(v)]\,,}
    \end{array}
\end{equation}
which reproduces the second rule in \eqref{eq:trace_rule_2}.
\item[{\it 6)}] If the arc occupies the passing line in one cycle and an endpoint of an arc in the other cycle of $b$, one sums over two possibilities ($u$ is fit)
\begin{equation}
\def\arraystretch{2}
\begin{array}{ll}
    \tau\big(\left[\db{p}u\right]\left[\db{s}v\right]\big) = & \raisebox{-.4\height}{\includegraphics[scale=0.6]{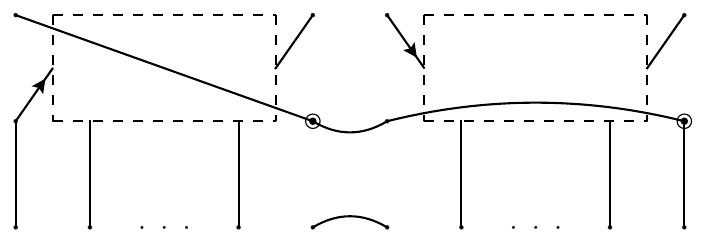}}\;+\;\raisebox{-.4\height}{\includegraphics[scale=0.6]{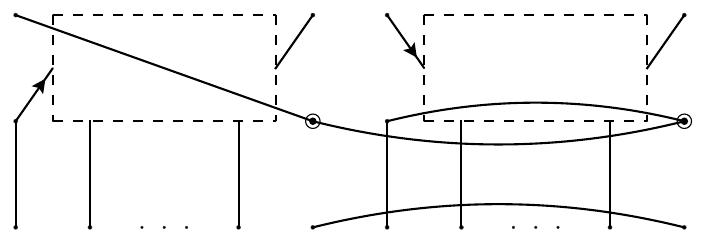}} \\
    \multicolumn{2}{l}{=\hspace{0.1cm}\,\raisebox{-.4\height}{\includegraphics[scale=0.6]{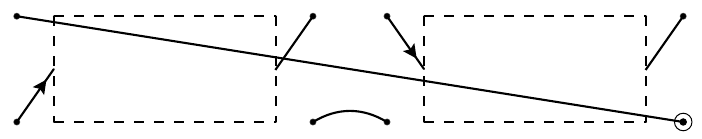}}\,+\,\raisebox{-.4\height}{\includegraphics[scale=0.6]{theo3_e_r2.pdf}}\; = \; [\bb{p}u\bb{s}v] \;+\; [\bb{p}u\bb{s}I(v)]\,,}
\end{array}    
\end{equation}
in agreement with the second rule in \eqref{eq:trace_rule_3}.
\item[{\it 7)}] Finally, if the arc occupies two passing lines in two independent cycles of $b$ (with both $u,v$ fit) one has
\begin{equation}
\begin{aligned}
    \tau\big(\left[\db{p}u\right]\left[\db{p}v\right]\big) & = \raisebox{-.4\height}{\includegraphics[scale=0.6]{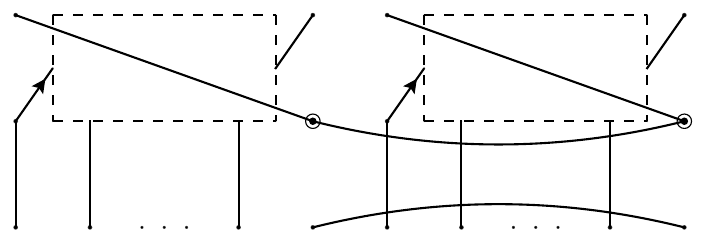}}\;=\;\hspace{0.1cm}\,\raisebox{-.4\height}{\includegraphics[scale=0.6]{theo3_f_r1fit.pdf}}\;=\; [\bb{n}u\bb{s}I(v)]\,,
\end{aligned}    
\end{equation}
which coincides with the second rule in \eqref{eq:trace_rule_4}.
\end{itemize}
Due to Lemma \ref{lem:trace_rules}, the considered cases are sufficient to prove the assertion.
\end{proof}

\subsection{Proof of Lemma \ref{lem:sym}}\label{app:lemma_sym}
\begin{proof}
    Let us first prove the assertion for $b\in B_{n}(\delta)$ such that $\zeta = \Phi(\gamma_b)$ is a single bracelet. Let us show that upon a convenient choice of $b$ among the conjugate diagrams, any $t \in C_{\Sn{n}}(b)$ can be written as $t = c^{m} r^{p}$ for some $m = 0,\dots, n-1$ and $p = 0,1$, where
    \begin{equation}
    c\; =\; \raisebox{-.4\height}{\includegraphics[scale=0.6]{csym.pdf}}\;, \hspace{3cm} r\;= \;\raisebox{-.4\height}{\includegraphics[scale=0.6]{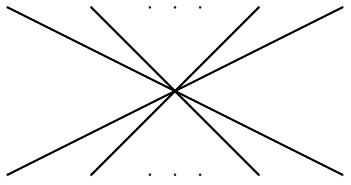}}
    \end{equation}
    (the fact that $c = r$ for $n = 2$ does not lead to any problem). Let us say that two lines in a diagram are {\it adjacent} if there exists a pair of vertically aligned nodes which are endpoints of these lines, {\it i.e.} these endpoints are identified by the step 2 in the definition of the map $\Phi$ in Section \ref{sec:classes-bracelets_map}. For any $s\in \Sn{n}$, conjugation $b\to sb s^{-1}$ preserves the adjacency relation, which is exactly the property encoded by the bracelet $\Phi(\gamma_b)$ (by placing letters along an unoriented circle, one exactly defines the nearest neighbours). If a conjugation transformation preserves the diagram (with $t\in C_{\Sn{n}}(b)$), then the  result of such transformation can be described as follows: each passing line takes the place of another passing line and each upper (respectively, lower) arc takes the place of another upper (respectively, lower) arc. Without loss of generality, take the convenient diagram in the conjugacy class: $b = c$ when $b\notin J^{(1)}$, while for $b\in J^{(1)}$ take
    \begin{equation}\label{eq:convenient_representative}
    b=\hspace{0.1cm}\raisebox{-.4\height}{\includegraphics[scale=0.7]{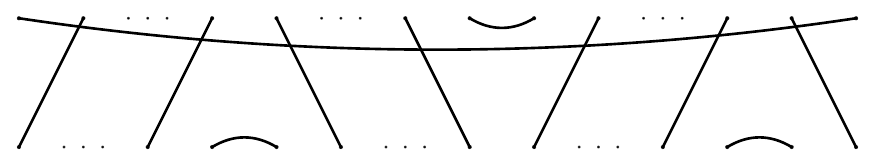}}\;.
    \end{equation}
    In view of preservation of adjacency of lines, the particularly chosen diagram $b$ can by preserved by adjoint transformations composed only by cyclic permutations of nodes and inversion. 
    \vskip 5 pt
    
    With this at hand, let us establish the bijection between the centraliser $C_{\Sn{n}}(b)$ and the turnover stabiliser $S(w_b)$ for a particular representative $w_b\in \zeta$. Namely, to read off the word $w_b$ (as described in steps 1, 2 in the definition of the map $\Phi$), start at the upper right node and follow the adjacent edge. With a slight abuse of notation, define the action of permutations $c$ and $r$ on words of the length $n$: $c(\bb{a}_1\dots\bb{a}_n) = \bb{a}_{n} \bb{a}_{1}\dots \bb{a}_{n-1}$ and $r(\bb{a}_1\dots\bb{a}_n) = cI(\bb{a}_1\dots\bb{a}_n) = \bb{a}_{1}\bb{a}_{n}\dots\bb{a}_{2}$. Then it is straightforward that $c^{m}r^{p}\in C_{\Sn{n}}(b)$ iff $c^{m}r^{p}\in S(w_b)$, which proves the assertion.
\end{proof}

\subsection{Proof of Lemma \ref{lem:commutativity}}\label{app:lemma_commutativity}
\begin{proof}
    One checks directly that for $n\leqslant 5$ any $\xi\in \mathbb{C}[\mathfrak{b}(\mathcal{A})]_{n}$ verifies $\xi^{*} = \xi$. As a result, any $u\in C_{n}(\delta)$ verifies $u^{*} = u$, so $C_{n}(\delta)\subset B_{n}(\delta)$ is commutative, and so is any $C_{n}(\delta)\slash J^{(f)}$ and $C_{n}(\delta)\cap J^{(f)}$. To show that $C_{n}(\delta)$ is non-commutative for any $n \geqslant 6$ take $\zeta_n = [\bb{nsnpsp}][\bb{p}]^{n-6}$ such that $\zeta_n^{*} = [\bb{snspnp}][\bb{p}]^{n-6} \neq \zeta_n$, and compare $A_n e_{\zeta_n}$ with $e_{\zeta_n} A_n$. To do so, notice that $e_{\zeta_n} A_n = (A_n e_{\zeta_n^{*}})^{*}$, so we make use of Theorem \ref{thm:Laplace} and compare $\Delta(\zeta_n)$ with $\big(\Delta(\zeta_n^{*})\big)^{*}$. It is simple to check that $\Delta(\zeta_n) = 2(\delta + 1)\,\zeta_n + \dots$ and $\big(\Delta(\zeta_n^{*})\big)^{*} = 2\delta\,\zeta_n + \dots$, so $\Delta(\zeta_n) \neq \big(\Delta(\zeta_n^{*})\big)^{*}$ for any $\delta\in \mathbb{C}$.
    \vskip 4 pt
    
    Let us prove the second part. For the point {\it (i)}, note that bracelets which parametrise $C_{n}(\delta)\slash J^{(f)}$ with $f = 1,2$ contain no more that one letter $\bb{n}$ (equivalently, $\bb{s}$), which means that they are self-dual with respect to $(\cdot)^{*}$, so the quotient algebra is commutative. Now, for $f \geqslant 3$ we aim at comparing $A^{(2)}_n e_{\zeta_n}$ and $e_{\zeta_n} A^{(2)}_n$, where $A^{(2)}_n$ is the double-arc generalisation of $A_n$ defined in \eqref{eq:class_A_f}. With no analog of Theorem \ref{thm:Laplace} available at this point, we analyse diagrammatic expressions with the aid of the following formula, which follows from the definition of the average \eqref{eq:average} (see the proof of Theorem \ref{thm:Laplace}):
    \begin{equation}
        \text{for any $u\in C_{n}(\delta)$ and for any diagram $b\in B_{n}(\delta)$ one has}\quad\quad u \gamma_{b} = \gamma_{u b}\quad \text{and}\quad \gamma_{b}\, u = \gamma_{b u}\,.
    \end{equation}
    When applied to the products in question, one fixes a representative in $e_{\zeta_n}$ and attaches two arcs to the lower/upper row of $b$ in all possible ways. By collecting only the basis vectors for $\zeta_n$ and $\xi_n = [\bb{nsp}]^2[\bb{p}]^{n-6}$ one can obtain
    \begin{equation}
        A^{(2)}_n e_{\zeta_n} = (\delta^2 + 2\delta)\,e_{\zeta_n} + (2+\delta)\,e_{\xi_n} + \dots\,,\quad e_{\zeta_n} A^{(2)}_n = (\delta^2 + \delta + 2)\,e_{\zeta_n} + 2(\delta + 1)\,e_{\xi_n} + \dots\,,
    \end{equation}
    so $A^{(2)}_n e_{\zeta_n} \neq e_{\zeta_n} A^{(2)}_n$ for any $\delta \in \mathbb{C}$.
    \vskip 4 pt
    
    For the point {\it (ii)}, note that bracelets which parametrise $C_{n}(\delta)\cap J^{(f_{\mathrm{max}})}$ contains no more that one letter $\bb{p}$, which means that they are self-adjoint, so the algebra is commutative. Now, for $f \leqslant f_{\mathrm{max}} - 1$ we consider $\tilde{\zeta}_n = [\bb{nsnpsp}][\bb{ns}]^{l}[\bb{p}]^{r}$ (with $r = n - 2f_{\mathrm{max}}$ and $r + l + 6 = n$) and aim at comparing $A^{(f_{\mathrm{max}} - 1)}_n e_{\tilde{\zeta}_n}$ and $e_{\tilde{\zeta}_n} A^{(f_{\mathrm{max}} - 1)}_n$. Similarly to the previous case, we collect only the basis elements for $\tilde{\zeta}_n$ and $\tilde{\xi}_n = [\bb{nsp}]^2[\bb{ns}]^{l}[\bb{p}]^{r}$, which gives
    \begin{equation}
        A^{(f_{\mathrm{max}} - 1)}_n e_{\tilde{\zeta}_n} = \delta (\delta^2 + 2\delta + 8)\,e_{\tilde{\zeta}_n} + (\delta^2 + 2\delta + 8)\,e_{\tilde{\xi}_n} + \dots\,,\quad e_{\tilde{\zeta}_n} A^{(f_{\mathrm{max}} - 1)}_n = \delta(\delta^2 + \delta + 4)\,e_{\tilde{\zeta}_n} + 2(\delta^2 + \delta + 2)\,e_{\tilde{\xi}_n} + \dots\,,
    \end{equation}
    so $A^{(f_{\mathrm{max}} - 1)}_n e_{\tilde{\zeta}_n} \neq e_{\tilde{\zeta}_n} A^{(f_{\mathrm{max}} - 1)}_n$ for any $\delta \in \mathbb{C}$.
\end{proof}

\end{appendix}


\end{document}